\documentclass[12pt]{amsart}
\usepackage{hyperref}
\usepackage{graphicx}
\usepackage{epstopdf}
\usepackage{amssymb}
\usepackage{color}
\usepackage{enumitem}

\setlist[enumerate,1]{label={(\alph*)}}

\DeclareMathOperator{\ess}{ess}
\DeclareMathOperator{\id}{id}

\newcommand{\R}{\mathbb{R}}
\newcommand{\C}{\mathbb{C}}
\newcommand{\Z}{\mathbb{Z}}

\newcommand{\vectornorm}[1]{\|#1\|}

\newcommand{\restrict[3]}{\left. {#2}\right|_{#3}}

\newtheorem{tm}{Theorem}[section]
\newtheorem{pr}[tm]{Proposition}
\newtheorem{lm}[tm]{Lemma}
\newtheorem{cy}[tm]{Corollary}
\newtheorem{cm}[tm]{Claim}

\newtheorem{as}[tm]{Assumption}
\theoremstyle{definition}
\newtheorem{df}[tm]{Definition}

\theoremstyle{remark}
\newtheorem{rem}[tm]{Remark}

\title{A reverse isoperimetric inequality for J-holomorphic curves}
\author{Yoel Groman and Jake P. Solomon}
\date{June 2014}
\begin{document}
\begin{abstract}
We prove that the length of the boundary of a $J$-holomorphic curve with Lagrangian boundary conditions is dominated by a constant times its area. The constant depends on the symplectic form, the almost complex structure, the Lagrangian boundary conditions and the genus. A similar result holds for the length of the real part of a real $J$-holomorphic curve. The infimum over $J$ of the constant properly normalized gives an invariant of Lagrangian submanifolds. We calculate this invariant to be $2\pi$ for the Lagrangian submanifold $\R P^n \subset \C P^n.$ We apply our result to prove compactness of moduli of $J$-holomorphic maps to non-compact target spaces that are asymptotically exact. In a different direction, our result implies the adic convergence of the superpotential.
\end{abstract}

\maketitle
\tableofcontents
\section{Introduction}\label{SecIntroduction}
\subsection{A consequence of the Cauchy-Crofton formula}\label{sec:cc}

We begin with a bound on the length of a real algebraic curve in terms of its degree, which we learned from~\cite{Gu}.
Let $\gamma$ be a one dimensional sub-manifold of $\mathbb{RP}^n$. Let $h$ be the round metric on $\mathbb{RP}^n$ normalized so the length of a line is~$1.$ Let $dV$ be the volume form on $Gr(n,n+1)$, the Grassmanian of hyperplanes in $\mathbb{R}^{n+1}$, that is invariant under the induced action of the isometry group of $h$ and satisfies
\[
\int_{Gr(n,n+1)}dV = 1.
\]
For $H\in Gr(n,n+1)$ denote by $N(H)$ the number of intersection points between $H$ and $\gamma,$
\[
N(H) = N(H,\gamma)=|\{\gamma\cap H\}|.
\]
By transverality, $N(H)$ is finite for generic $H$. Let $\ell(\gamma;h)$ denote the length of $\gamma$ with respect to $h$. The Cauchy-Crofton formula \cite[eq. 12]{Sant} asserts that
\begin{align}\label{crofton}
\int_{H\in Gr(n,n+1)}N(H)dV=\ell(\gamma;h).
\end{align}
For a quick sanity check, note that when $\gamma$ is a straight line, both sides of equation \eqref{crofton} are equal to 1.

Assume now that $\gamma$ is an algebraic curve of degree $d$. As observed in \cite[p.45]{Gu}, the Cauchy-Crofton formula~\eqref{crofton} implies
\begin{align}\label{CCEst}
d\geq \ell(\gamma;h).
\end{align}

We interpret inequality~\eqref{CCEst} as a reverse isoperimetric inequality. The Fubini-Study metric on $\C P^n$ normalized so the area of a complex line is $\frac{1}{\pi}$ induces the metric $h$ on $\R P^n,$ so we denote it also by $h.$ Let $\alpha \in H_2(\C P^n,\R P^n)\simeq \Z$ be the generator that pairs positively with the K\"ahler form of $h.$ Let $(\Sigma,\partial\Sigma)$ be a Riemann surface with boundary, and denote by $[\Sigma,\partial\Sigma] \in H_2(\Sigma,\partial\Sigma)$ the relative fundamental class. Let
\[
u:(\Sigma,\partial\Sigma)\rightarrow (\mathbb{C}P^n,\mathbb{R}P^n)
\]
be a holomorphic map with $[u]: = u_*([\Sigma,\partial\Sigma]) = d\alpha.$ By Wirtinger's theorem,
\[
Area(u;h) = \frac{d}{2\pi}.
\]
Schwarz reflection implies that $u|_{\partial\Sigma}$ is real algebraic. So equation~\eqref{CCEst} gives the estimate
\begin{align}\label{RIIProjSpace}
2\pi Area(u;h)\geq\ell(u|_{\partial\Sigma};h).
\end{align}
The present paper extends inequality~\eqref{RIIProjSpace} to general symplectic manifolds $M$, Lagrangian submanifolds $L\subset M$ , and $J$-holomorphic maps $(\Sigma,\partial\Sigma) \to (M,L)$.

\subsection{The compact case}

In the following let $(M,\omega)$ be a symplectic manifold, let $L\subset M$ be a Lagrangian submanifold and let $J$ be an almost complex structure on $M$ for which the form $\omega(\cdot,J\cdot)$ is positive definite. Denote by $g_J$ the symmetrization of the form $\omega(\cdot,J\cdot)$. For $\Sigma$ a Riemann surface with boundary we denote by $\Sigma_{\C}$ the complex double of $\Sigma$.

\begin{tm}\label{TmRIIBountdary}
Suppose $M$ is compact. There are constants $f_1=f_1(J,\omega,L)$ and $g_1=g_1(J,\omega,L)$, homogeneous of degrees $-\frac1{2}$ and $\frac1{2}$ respectively in $\omega$, with the following significance. For any Riemann surface $\Sigma$ with boundary, and for any $J$-holomorphic curve
\[
u:(\Sigma,\partial\Sigma)\rightarrow (M,L),
\]
we have
\begin{align}\label{eqRIIBoundary}
\ell(u|_{\partial \Sigma};g_J)\leq f_1{Area(u;g_J)}+ g_1genus(\Sigma_{\C}).
\end{align}
\end{tm}
Let  $(M,L)= (\mathbb{C}P^n,\mathbb{R}P^n)$, let $\omega=\omega_{FS}$ be the Fubini-Study form and let $J=J_{st}$ be the standard complex structure. By the discussion in Section~\ref{sec:cc}, we may take
\[
f_1(J_{st},\omega_{FS},L)=2\pi,\quad g_1(J_{st},\omega_{FS},L)=0.
\]
At this point it is not clear whether the genus dependence in \eqref{eqRIIBoundary} can be eliminated in the general case. On the other hand, the monotonicity inequality~\cite{Si} gives a lower bound on $Area(u;g_J).$ So the constant $g_1$ can be set to zero at the expense of making $f_1$ genus dependent.

Recently, real symplectic geometry has attracted considerable attention. See, for example,~\cite{We05,PS08,IKS09,BM09}. Theorem~\ref{TmRIIBountdary} has the following parallel in the real-symplectic setting.
Recall that a \textbf{real} symplectic manifold is a triple $(M,\omega,\phi)$ where $(M,\omega)$ is a symplectic manifold and $\phi:M\rightarrow M$ is an anti-symplectic involution, that is, $\phi^*\omega=-\omega.$ The natural compatibility condition for almost complex structures is that $\phi^*J=-J$. A
\textbf{real} Riemann surface is a pair $(\Sigma,\psi),$ where $\Sigma$ is a Riemann surface and $\psi:\Sigma\rightarrow \Sigma$ is an anti-holomorphic involution. We denote by $\Sigma_\R$ the fixed point set of $\psi.$ A $J$-holomorphic curve $u:\Sigma\rightarrow M$ is called \textbf{real} if $\phi\circ u=u\circ\psi.$
\begin{tm}\label{TmRIIReal}
Suppose $(M,\omega,\phi)$ is a compact real symplectic manifold. There are constants $f_2=f_2(J,\omega)$ and $g_2=g_2(J,\omega)$, homogeneous of degrees $-\frac1{2}$ and $\frac1{2}$ respectively in $\omega$, with the following significance. For any real Riemann surface $(\Sigma,\psi)$ and any real $J$-holomorphic curve $u : \Sigma \to M,$ we have
\begin{align}\label{eqRIIReal}
\ell( u|_{\Sigma_\R};g_J)\leq f_2{Area(u;g_J)}+ g_2genus(\Sigma).
\end{align}

\end{tm}

We remark  that a forward isoperimetric inequality does not hold for holomorphic curves. Indeed, consider degree $2$ curves in $(\mathbb{CP}^2,\mathbb{RP}^2)$. For $t > 0,$ let $C_t$ be the closure of one of the two connected components of the non-real solutions of the equation $X^2+Y^2-t=0$. Then $C_t$ has constant area but arbitrarily small boundary length as $t$ goes to 0.

\subsection{The optimal isoperimetric constant}\label{sec:oic}

The preceding theorems, though they involve Riemannian length measurements, lead to a purely symplectic invariant of Lagrangian submanifolds. For a given $J$, denote by $F_1(\omega,J,L)$ the optimal value of the constant $f_1(\omega,J,L)$ of Theorem~\ref{TmRIIBountdary} when $u$ ranges over $J$-holomorphic maps from a surface of genus $0.$ Define the constant $h_1(M,L,\omega)$ by
\[
h_{1}(M,L,\omega)=\inf_{J}\frac{F_1(\omega,J,L)}{2 Diam(L;g_J)},
\]
where the infimum is over all $J$ tamed by $\omega.$ Similarly, for $(M,\omega,\phi)$ a real symplectic manifold, for given $\phi$ anti-invariant $J,$ denote by $F_2(\omega,J)$ the optimal value of the constant $f_2(\omega,J)$ of Theorem~\ref{TmRIIReal} when $u$ ranges over maps from the surface of genus $0.$ Define the constant $h_2(M,\phi,\omega)$ by
\[
h_2(M,\phi,\omega) = \inf_J \frac{F_2(\omega,J)}{2 Diam(Fix(\phi);g_J)},
\]
where the infimum is over all tame $J$ such that $\phi^*J = -J.$

Clearly, $h_1$ and $h_2$ are symplectic invariants. Theorems~\ref{TmRIIBountdary} and~\ref{TmRIIReal} imply that $h_i(M,L,\omega) < \infty.$ As seen in the proof of the following proposition, lower bounds follow from open Gromov-Witten theory.
\begin{pr}\label{pr:ex}
Let $M=\mathbb{CP}^n$, let $L=\mathbb{RP}^n$ and let $\omega=\omega_{FS}$ be the Fubini-Study form normalized so the area of a line is $\frac{1}{\pi}.$ Then $h_1(M,L,\omega)=2\pi.$
\end{pr}

For lower bounds on $h_2,$ we can use the rapidly developing theory of Welschinger invariants~\cite{We05,IKS09,BM09,HS12}. In the projective real algebraic case, the discussion of Section~\ref{sec:cc} gives upper bounds on $h_2.$ However, to minimize the discrepancy between upper and lower bounds, it is necessary understand how to maximize diameter within a deformation class of projective real algebraic varieties.

There are various ways to generalize $h_i$ to higher genus, and it seems interesting to study the resulting invariants. Also, restricting to $u$ with non-trivial boundary degree, it could be interesting to consider the ratio of the isoperimetric constant $F_1$ and a version of the $1$-systole of $L$. As we will see in Section~\ref{sec:adic}, such a ratio arises naturally in the context of the Fukaya category. See~\cite{Ka07} for background on systolic geometry. We leave these problems for future research.

\subsection{The general case}
We show how Theorems~\ref{TmRIIBountdary} and~\ref{TmRIIReal} generalize when $M$ and $L$ are not compact. In the process, we characterize more precisely the dependence of the isoperimetric constants $f_i$ on the geometry of $(M,\omega,J,L).$

Denote by $R$ the curvature of $M,$ by $B$ the second fundamental form of $L$ and by $i$ the radius of injectivity of $M$, all with respect to the metric $g_J$. For a tensor $A$ on $M$ or on $L,$ we denote by $\vectornorm{A}_m$ the $C^m$ norm of $A$ with respect to $g_J$. We say that $(M,\omega,J,L)$ has \textbf{$K$-bounded} geometry if
\[
\max\left\{\vectornorm{R}_2,\vectornorm{J}_2,\vectornorm{B}_2,\frac1{i}\right\}<K.
\]
For any Riemannian manifold $X$ with submanifold $Y$ and $\epsilon>0$, we say that $Y$ is $\epsilon$-Lipschitz if
\[
\frac{d_X(x,y)}{\min\{1,d_Y(x,y)\}} \geq\epsilon\qquad \forall x\neq y\in Y.
\]

\begin{tm}\label{Estimates}
There are functions $f_1=f_1(K)$ and $g_1=g_1(K)$, with the following significance. Theorem \ref{TmRIIBountdary} holds upon replacing the assumption that $M$ is compact by the assumption that
$(M,\omega,J,L)$ has $K$-bounded geometry as well as one of the following:
\begin{enumerate}
\item \label{CaseA} $L$ is $\frac1{K}$-Lipschitz.
\item \label{CaseB} Consider the conformal metric on $\Sigma$ of constant curvature $0,\pm 1,$ of unit area in case of zero curvature, such that $\partial\Sigma$ is totally geodesic. Then $\partial\Sigma$ and each connected component of $L$ are $\frac1{K}$-Lipschitz.
 \end{enumerate}
\end{tm}

Theorem~\ref{TmRIIReal} generalizes to the non-compact case under considerably weaker assumptions.
\begin{tm}\label{Estimates2}
There are functions $f_2(K)$ and $g_2(K)$ such that Theorem \ref{TmRIIReal} holds upon replacing the requirement that $M$ be compact with the requirement that
\[
\max\left\{\vectornorm{R},\vectornorm{J}_2,\frac1{i}\right\}<K.
\]
\end{tm}

There is also an a priori estimate on the diameter of $J$-holomorphic curves.
\begin{tm}\label{tmDiamEst}
There are functions  $f_3=f_3(K)$ and $g_3=g_3(K)$ such that under the same assumptions as in Theorem \ref{Estimates} and denoting $b:=|\pi_0(\partial\Sigma)|$,
\begin{align}\label{DiamEst}
Diam(u(\Sigma);g_J)\leq (b+1)[f_3{Area(u;g_J)}+ g_3genus(\Sigma_{\C})].
\end{align}
\end{tm}
In the case of closed curves, which includes conjugation invariant curves, as well as in case~\ref{CaseA} of Theorem \ref{Estimates}, the diameter estimate \eqref{DiamEst}, without any genus dependence, was proved by Sikorav \cite{Si} using the monotonicity inequality. However, Sikorav's technique is image oriented, so we could not see how it would allow one to utilize the bounds on the domain necessary in case~\ref{CaseB} of Theorem \ref{Estimates}. More importantly, we could not see how to generalize Sikorav's technique to obtain results on boundary length.
\subsection{An example}

The following example, due to~\cite{Liu}, illustrates the role of conditions~\ref{CaseA} and~\ref{CaseB} of Theorem~\ref{Estimates}.
Consider the special Lagrangian fibration of $M = \C^3$ discussed in \cite{HaLa} and \cite{KV}. Namely, let $H:\C^3\rightarrow \R^3$, be given by
\[
(z_1,z_2,z_3)\mapsto (|z_1|^2-|z_3|^2,|z_2|^2-|z_3|^2, Im(z_1z_2z_3)).
\]
It can be shown that for each $c\in\C^3$ the fiber $H^{-1}(c)$ is Lagrangian. Moreover, letting $J_{0}$ and $\omega_0$ denote the standard complex and symplectic structures on $\C^3,$ it can be shown that $(M,\omega_0,J_{0},H^{-1}(c))$ has bounded geometry and that $H^{-1}(c)$ is Lipschitz. Thus the reverse isoperimetric inequality applies to curves with boundary in $H^{-1}(c)$.

On the other hand, consider a Lagrangian $L\subset\mathbb{C}^3$ which is the union of two or more fibers of $H$. Then the components of $L$ are arbitrarily close to each other at infinity. Thus $L$ as a whole is not Lipschitz, but each component of $L$ is. We construct a counterexample to the reverse isoperimetric inequality as follows. Let $L_0=H^{-1}(0,0,0)$ and $L_1=H^{-1}(-1,-1,0)$.  For any $a\in \R$ we construct a holomorphic annulus with one boundary component in $L_0$ and the other one in $L_1$. For any $a\in \R$ let $r_a$ be a positive solution of the equation
\[
a^2r^6+r^4-a^2=0.
\]
Let $\Sigma_a$ be the annulus in the plane with radii $1$ and $r_a$.  Consider the map
\[
(\Sigma_a,\partial\Sigma_a)\rightarrow (\C^3,L_0\cup L_1)
\]
given by
\[
z\mapsto  \left(az,az,\frac{a}{z^2}\right).
\]
Allowing $a$ to approach infinity, the boundary length is unbounded while the area, being a homological invariant, remains constant. Thus there is no reverse isoperimetric inequality in this case. Theorem~\ref{Estimates}\ref{CaseB} implies that the Lipschitz constant of $\partial\Sigma_a$ goes to $0$ as $a$ approaches infinity. This can indeed be verified directly by noting that
\[
\lim_{a\rightarrow\infty}r_a=1,
\]
so $Mod(\Sigma_a)=\ln r_a\rightarrow 0$.

\subsection{Application to compactness}
We apply Theorem \ref{Estimates} to deduce compactness of moduli spaces of $J$-holomorphic maps in the following scenario. Our argument can be seen as a quantitative version of the idea of~\cite{Kosh}. We say that the symplectic form $\omega$ is \textbf{asymptotically exact} if  for a point $p\in M$ there exists a 1-form $\lambda$ such that
 \[
 \lim _{R\rightarrow \infty}\vectornorm{\restrict{(\omega-d\lambda)}{M\backslash B_R(p)}}_{C^0}=0.
 \]
Here and below we use the $C^0$ norm induced by $g_J$. Similarly, we say that $L$ is an asymptotically exact Lagrangian submanifold if $\omega$ is asymptotically exact and there is a function $f:L\rightarrow \mathbb{R}$ so that
 \[
 \lim _{R\rightarrow \infty}\vectornorm{\restrict{(\lambda-df)}{L\backslash B_R(p)}}_{C^0}=0.
 \]
Applying Stokes theorem, the following is immediate.
\begin{cy}\label{CptApp}
Let $A\in H_2(M,L)$, $g\in\mathbb{N}\cup\{0\},$ and $p\in M$. Assume $\omega$ and $L$ are asymptotically exact and $(M,\omega,J,L)$ has $K$-bounded geometry. Then there is an $R=R(M,L,p,K,A)$ such that any $J$-holomorphic
\[
u:(\Sigma,\partial\Sigma)\rightarrow (M,L)
\]
with $[u]=A$ and $genus(\Sigma_\C)\leq g$ that satisfies the conditions of Theorem \ref{Estimates} also satisfies $u(\Sigma)\subset B_R(p)$.
\end{cy}
Given this corollary, standard Gromov compactness implies compactness of the moduli space of stable maps of degree $A$ and
\[
genus(\Sigma_\C)\leq g.
\]
In a paper to appear subsequently, we prove that toric Calabi Yau manifolds along with the Lagrangian submanifolds of Aganagic Vafa~\cite{AV00} are asymptotically exact and have bounded geometry. Thus, we apply the construction of~\cite{Li03} to define open Gromov-Witten invariants for general toric Calabi Yau 3-folds.

\subsection{Application to adic convergence}\label{sec:adic}
The linearity in the estimate of Theorem \ref{TmRIIBountdary} is important for proving adic convergence results in the $A_\infty$-algebra associated with $L$. As an example of the utility of Theorem \ref{TmRIIBountdary} in this connection, we give an alternative proof of a result of \cite{FukAdic,FukCount}. For simplicity, we restrict attention to cases where the moduli space of $J$-holomorphic disks up to reparametrization is of dimension zero, such as when $M$ is a compact Calabi-Yau manifold of dimension three and $L$ a Lagrangian sub-manifold of Maslov index zero. We limit our discussion to a simplified version of the superpotential. Nevertheless, the ideas we present are not limited to this simplified setting.

Let $T$ be a formal variable and write
\[
\Lambda_{0,nov}:= \left\{\left.\sum_{i}a_iT^{\lambda_i}\right|a_i\in\C,\lambda_i\in\mathbb{R}_{\geq 0},\lim_{i\to\infty}\{\lambda_i\}=\infty\right\},
\]
and
\[
\Lambda_{nov}:= \left\{\left.\sum_{i}a_iT^{\lambda_i}\right|a_i\in\C,\lambda_i\in\mathbb{R},\lim_{i\to\infty}\{\lambda_i\}=\infty\right\}.
\]
Let
\[
val\left(\sum_{i}a_iT^{\lambda_i}\right):=\inf\{\lambda_i|a_i\neq 0\}.
\]
The field $\Lambda_{nov}$ is equipped with a non-Archimedean norm $|\cdot|$ defined by
\[
|x|:=e^{-val(x)}.
\]

To count $J$-holomorphic disks in $(M,L),$ we use a domain dependent $J$ as follows. Let $J = \{J_z\}_{z\in D^2}$ be a family of almost complex structures compatible with $\omega.$ Denote by $j$ the standard complex structure on $D^2.$ Let $u : (D^2, \partial D^2) \to (M,L).$ We say $u$ is $J$-holomorphic if
\[
du_z + J_{u(z),z} \circ du_z \circ j_z = 0, \qquad z \in D^2.
\]
Our theorem applies to such maps via the following graph construction. Think of $D^2$ as the northern hemisphere of $S^2$ and denote by $j$ the standard complex structure on $S^2$ as well. Extend the family of almost complex structures $J_{z}$ arbitrarily to $z \in S^2.$ Let $\widetilde M = M \times S^2,$ let $\widetilde J$ be the almost complex structure on $\widetilde M$ given by $\widetilde J_{x,z} = J_{x,z} \oplus j_z,$ and let $\widetilde L = L \times \partial D^2.$  Define $\tilde u : (D^2,\partial D^2) \to (\widetilde M,\widetilde L)$ by $\tilde u = u \times \id.$ Then $u$ is $J$-holomorphic if and only if $\tilde u$ is $\widetilde J$ holomorphic, and our theorem applies directly to $\tilde u.$ For a generic choice of $J,$ it is shown in~\cite{MS2} that all $J$-holomorphic maps $u$ are regular. Imposing appropriate divisor constraints on $u$ as in~\cite{So06}, we obtain a discrete space. So, for $\beta \in H_2(M,L;\Z),$ we define $N_\beta$ to be the number of $J$-holomorphic disks with boundary in $L$ representing~$\beta$. In general, the numbers $N_\beta$ depend on the choice of $J.$

The superpotential is a map $\psi:H^1(L;\C)\to \Lambda_{0,nov}$ defined as follows. Let $\{e_i\}_{i=1}^m$ be a basis for the integral lattice in $H^1(L;\C).$ For a relative homology class $\beta\in H_2(M,L;\Z),$ denote by $E(\beta)$ the symplectic area of $\beta,$ and denote by $(\partial\beta,e_i)$ the integral of $e_i$ over the boundary of $\beta.$ For
\[
b=\sum_{i=1}^m b_i e_i\in H^1(L;\C),
\]
the superpotential is given by
\[
\psi(b):=\sum_{\beta\in H_2(M,L;\Z)}N_\beta e^{\sum_{i=1}^m b_i{(\partial\beta,e_i)}}T^{E(\beta)}.
\]

Let $\tilde{\psi}:\left(\C^\times\right)^m\to \Lambda_{0,nov}$ be the map
\begin{equation}\label{GWpot}
\tilde{\psi}(t_1,...,t_m)=\sum_{\beta\in H_2(M,L;\Z)}N_\beta \prod_{i=1}^m t_i^{(\partial\beta,e_i)}T^{E(\beta)}.
\end{equation}
Thus, $\psi$ factors through the exponential map $\exp:H^1(L;\C)\to(\C^\times)^m$ via $\tilde{\psi}$.
For the intrinsic construction of the SYZ mirror to $M$ as outlined in~\cite{FOOO,WuTu}, it is desirable to extend the domain of $\tilde{\psi}$ and other similar power series to an appropriate annulus in $(\Lambda_{nov})^m.$ For $\alpha$ a $1$-form on $L,$ let $\| \alpha \|_{L^\infty}$ be the norm given by taking the supremum over the $L^\infty$ norms associated with the Riemannian metrics $g_{J_z}$ for $z \in D^2.$ For $a \in H^1(L),$ let $\|a\|_{L^\infty}$ denote the infimum of $\|\alpha\|_{L^\infty}$ over forms $\alpha$ representing $a.$ The following theorem is proved in~\cite{FukAdic,FukCount} via a different route.

\begin{tm}\label{tm:supp}
Let $\delta = \frac1{f_1m\max{\vectornorm{e_i}_{L^\infty}}}.$ Then the power series \eqref{GWpot} converges for $(t_1,...,t_m)\in\Lambda_{nov}^m$ with $val(t_i)\in(-\delta,\delta)$.
\end{tm}
\begin{proof}
A series $\sum x_i$ of elements $x_i$ in a non-Archimedean field converges if and only if $\lim |x_i|=0$. See~\cite{Bosch} for relevant background. In the case of $\Lambda_{nov}$, this just means that for any $E \in \R$ there exists a finite set $A_E$ such that for any index $i\not\in A_E$ and any $E'\leq  E,$ the  coefficient of $T^{E'}$ in $x_i$ is 0. For the question at hand, let $A_E$ be the set of classes $\beta\in H_2(M,L;\Z)$ such that $N_{\beta}\neq 0$ and
\[
  \left(1-mf_1\max_i\left\{|val(t_i)|\vectornorm{e_i}_{L^\infty}\right\}\right) E(\beta)\leq E.
\]
Suppose  $val(t_i)\in(-\delta,\delta)$. Then by Gromov compactness, $A_E$ is a finite set. On the other hand, by Theorem~\ref{TmRIIBountdary} we have
\[
  (\partial\beta,e_i)\leq f_1E(\beta)\vectornorm{e_i}_{L^\infty}
\]
for any $\beta$ with $N_{\beta}\neq 0$. In particular, for any $\beta\not\in A_E$ and $E'\leq E$ the coefficient of $T^{E'}$ in $N_\beta \prod_{i=1}^m t_i^{(\partial\beta,e_i)}T^{E(\beta)}$ vanishes.
 \end{proof}

The proof given in \cite{FukAdic,FukCount} for the same result avoids the heavy analysis going into the proof of Theorem~\ref{TmRIIBountdary}. However, it is not clear how to extend that proof to the $A_\infty$-algebra of immersed Lagrangians as defined in~\cite{AkahoJoyce}. On the other hand, the above proof applies to immersed Lagrangians if we extend Theorem~\ref{TmRIIBountdary} to $J$-holomorphic maps with strip-like ends asymptotic to a transverse intersection point. Such an extension should not be hard. We thank K. Fukaya for this idea.

In the spirit of Section~\ref{sec:oic}, Theorem~\ref{tm:supp} suggests we consider the purely symplectic invariant
\[
h_3(M,L,\omega) = \inf_J F_1(\omega,J,L) \max \|e_i\|_{L^\infty}.
\]
The norm $\|\cdot\|_{L^\infty}$ on cohomology has been studied in the context of systolic geometry~\cite{Ka07}. Norms of integral homology bases have been studied in~\cite{BuS02} for the case of surfaces. An upper bound on the convergence radius of the superpotential would imply a non-trivial lower bound for $h_3.$

\subsection{Idea of the proof}
We restrict attention to the case of real $J$-holomorphic maps from a real Riemann surface $(\Sigma,\psi),$ which for the time being, we assume to be a sphere. We assume the fixed point set $\Sigma_\R$ of $\psi$ is non-empty and abbreviate $\gamma = \Sigma_\R$. We equip $\Sigma$ with a round metric invariant under $\psi$ of radius $1.$ In particular, $\gamma$ is a great circle.

Let $u: \Sigma\rightarrow M$ be real $J$-holomorphic. Let $f:\gamma\rightarrow [0,\infty)$ be given by $f(x):=\vectornorm{du(x)}$. Then $\ell(u(\gamma))$ is the area of the hypograph of $f$,
\[
H=\{(x,t)\in\gamma\times[0,\infty)|t \leq f(x) \}.
\]
The bubbling phenomenon implies that there is no a priori bound on the $L_\infty$ norm of $f$ in terms of the energy of $u$. Thus our argument relies on a close analysis of the set $H$.

In the following sketch of our proof, we make the simplifying assumption that $f$ has a unique local maximum on $\gamma$. For $t\in[1,\infty)$ write $w_t=\{x\in\gamma|f(x)\geq t\}$ and let $W_t\subset\Sigma$ be the image under the exponential map of the ball of radius $\frac1{t}$ in the normal bundle of $w_t$. Our simplifying assumption implies that $w_t$ is connected. Note also that if $t'>t$, then $w_{t'}\subset w_t$. For any $t$ write
\[
a_t:=\left\lfloor\frac{t\ell(w_t)}{2}\right\rfloor.
\]

One of the main ways the fact that $u$ is $J$-holomorphic enters our proof is the following energy quantization property. Let $p\in\Sigma$ with $d=\vectornorm{du(p)} \geq 1,$ and let $B\subset\Sigma$ be the disk of radius $\frac1{d}$  centered at $p$. Then it is known~\cite[Lemma 4.3.1]{MS2} \footnote{See Remark \ref{rmGradApp} below.} that
\begin{align}\label{DefDenseDisc}
\int_{B}\vectornorm{du}^2\geq\delta,
\end{align}
where $\delta$ is a constant that depends only on the geometry of $(M,\omega,J,L).$ From now on, we call a disk $B\subset\Sigma$ satisfying \eqref{DefDenseDisc} a \textbf{dense disk}.

Let $R_i=w_{3^i}\times[0,3^{i+1}]$ for $0\leq i\leq \log_3\sup_{\partial\Sigma} f$. Clearly, the rectangles $R_i$ cover the area under the graph of $f$ wherever $f\geq1$. Thus to deduce Theorem \ref{TmRIIReal}, it suffices to bound the sum
\[
S:=\sum_i3^{i+1}\ell(w_{3^i}).
\]
To get such a bound, note that for any $i\in\mathbb{N}$, by the energy quantization property, the set $W_{3^i}\setminus W_{3^{i+1}}$ contains at least $n_i$ disjoint dense disks, where

\[
n_i:=a_{3^i}-\left\lceil\frac1{3}(a_{3^{i+1}}+1)\right\rceil-1.
\]
See Figure~\ref{wt}. The term in the middle of the formula is an upper bound on the number of disks of radius $3^{-i}$ required to cover $W_{3^{i+1}}$. The $-1$ in the formula expresses the fact that an arbitrarily small neighborhood of each end point of $W_{3^{i+1}}$ can knock out a whole dense disk of $W_{3^{i}}.$
\begin{figure}[h]
\centering
\includegraphics[scale=0.6]{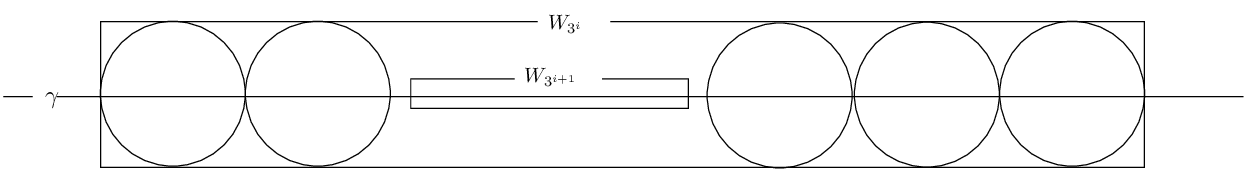}
\caption{}\label{wt}
\end{figure}
Since area coincides with energy for $J$-holomorphic maps, estimate~\eqref{DefDenseDisc} implies the bound
\begin{equation*}
\sum n_i\leq\frac{Area(\Sigma;g_J)}{\delta}.
\end{equation*}
Assume momentarily that for each $i$ we have
\begin{align}\label{EqAssumption}
a_{3^i}\geq 4.
\end{align}
Then
\[
a_{3^{i+1}}-\left\lceil\frac1{3}(a_{3^{i+1}}+1)\right\rceil-1\geq \frac{1}{4}a_{3^{i+1}}.
\]
So,
\begin{equation*}
\sum n_i\geq \frac{1}{4}\sum a_{3^i} =  \frac1{4}\sum\left\lfloor\frac1{2}3^{i}{\ell(w_{3^i})}\right\rfloor.
\end{equation*}
But, again using assumption~\eqref{EqAssumption}, we have
\begin{equation*}
\left\lfloor\frac1{2}3^{i}{\ell(w_{3^i})}\right\rfloor \geq \frac{4}{5}\cdot\frac{1}{2}3^i\ell(w_{3^i}).
\end{equation*}
So,
\begin{equation*}
\frac{Area(\Sigma;g_J)}{\delta} \geq \sum n_i \geq \frac{1}{4}\cdot\frac{4}{5}\cdot \frac{1}{2} \sum 3^i\ell(w_{3^i})
\geq \frac1{5\cdot6}S,
\end{equation*}
giving the required bound. The argument breaks down, however, if we remove assumption~\eqref{EqAssumption} as in Figure \ref{wt2}.
\begin{figure}[h]
\includegraphics[scale=0.5]{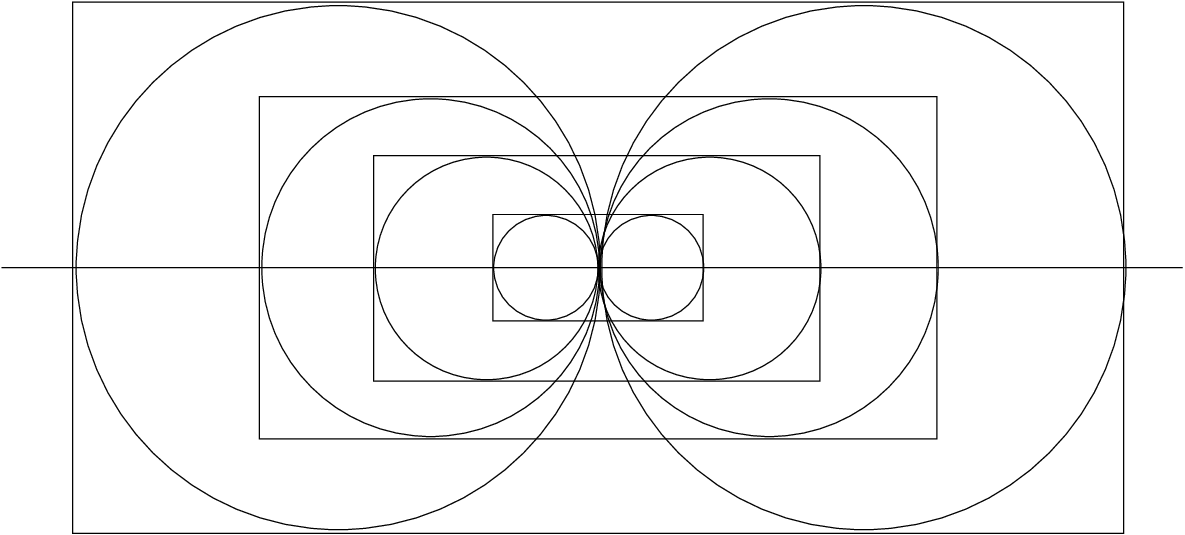}
\caption{}\label{wt2}
\end{figure}

To deal with this problem, we partition $H$ into the sets
\[
K=\left\{(x,t)\in H\left| \ell(w_t)\geq \frac8{t},t\geq1\right.\right\},
\]
and $N=H\setminus K$. These are the thick and thin parts of the hypograph respectively. Figure \ref{graphs} shows a possible alternation between thick and thin. The solid line is the graph of $f$, and the dashed line is the graph of the function
\[
x\mapsto \frac8{d(x,x_0)},
\]
where $x_0$ is the point where $f$ obtains its maximum. Conceptually, the components of the thick part should be thought of as parts of $\gamma$ which lie on bubbles of $u$, and those of the thin part as the necks separating the bubbles.  A slight modification of the above argument shows the same linear lower bound on the area of the thick part. For the thin part we utilize the bound on the width of the graph and well known properties of $J$-holomorphic annuli.
\begin{figure}[h]
\includegraphics[scale=0.5]{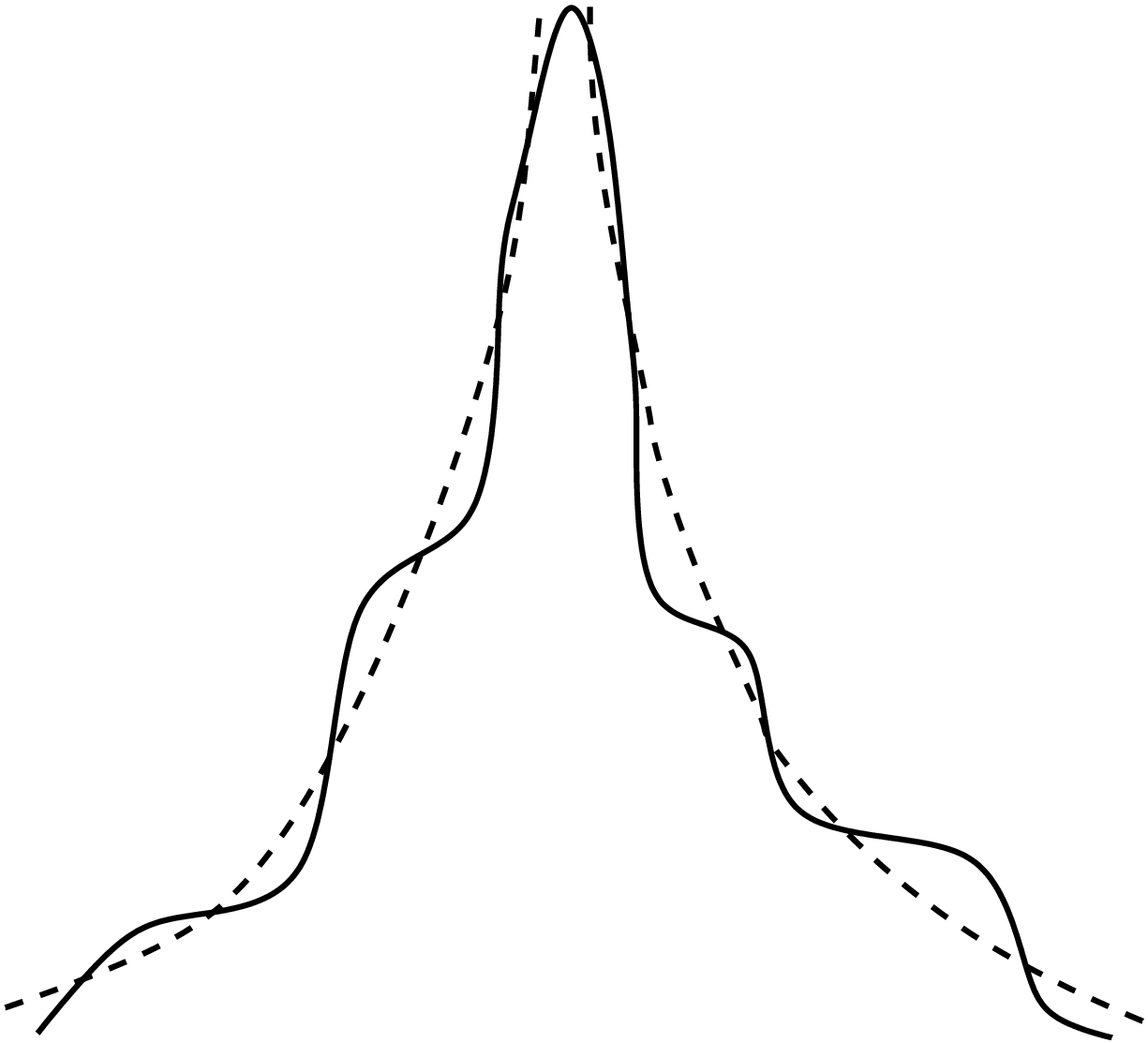}
\caption{}\label{graphs}
\end{figure}

When we consider maps from surfaces with genus greater than 0, we also have to deal with the fact that the domain has unbounded length. For values of the derivative which are small relative to the radius of injectivity, the above argument again breaks down. Integrating a small number over an unbounded domain gives an unbounded value. This is again dealt with by utilizing the properties of holomorphic annuli. Our proof thus exhibits a pleasing symmetry between the thin part of the graph and the thin part of the surface.

The main technical difficulties in the proof arise from the a priori arbitrary arrangement of the critical points of $f$. A large part of the proof is devoted to constructing a partition of the hypograph of $f$ into a thick part and a thin part in such away that the above arguments apply.

The paper is organized as follows. Section \ref{SecPrelimimaries} reviews basic notions of the conformal geometry of surfaces as well as the thick thin decomposition for Riemann surfaces with negative Euler characteristic equipped with a hyperbolic metric. Section \ref{SecThickThin} presents the concept of thick thin measures. The measure one should have in mind is the energy of a $J$-holomorphic map from the surface. Section \ref{SecState} formulates Theorem \ref{TmApLenBound}, which is a generalization of the theorems in this introduction. It addresses arbitrary conjugation invariant geodesics in the complex double $\Sigma_{\C}$. Section \ref{SecTreePart} discusses a partition for hypographs of continuous functions whose elements form a tree. This partition is the key to the discussion of the thick thin partition of the next section. Section \ref{SecThickThinPart} presents the thick thin partition of the hypograph relevant to the proof. It then shows that the number of components of the thin part is linearly bounded by the energy. Moreover, each component lies in a suitable annulus. In Section \ref{SecTameGeo} we define a notion of tame geodesics in holomorphic annuli and discuss their properties. Roughly speaking, tame geodesics are those that do not wrap around the annulus too quickly. In Section \ref{SecProof} we prove Theorem \ref{TmApLenBound}. In Section \ref{SecApp} we deduce all the theorems of the introduction from Theorem~\ref{TmApLenBound}. Finally, in Section~\ref{sec:calc} we prove Proposition~\ref{pr:ex}.

\subsection{Acknowledgements}
The authors would like to thank Kenji Fukaya, Asaf Horev, Mikhail Katz, David Kazhdan, Melissa Liu, and Ran Tessler, for helpful conversations. The authors were partially supported by Israel Science Foundation grant 1321/2009 and Marie Curie International Reintegration Grant No. 239381.

\section{Preliminaries on conformal geometry.}\label{SecPrelimimaries}

Let $(I,j)$ be a compact doubly connected surface with complex structure $j$. The \textbf{modulus} of $(I,j)$, denoted by $Mod(I,j)$ or $Mod(I)$ when the complex structure is clear from the context, is the unique real number $r>0$ such that $(I,j)$ is conformally equivalent to $[0,r]\times S^1$ (see \cite{FaKr}). Here $S^1$ is taken to be a standard circle of length $2\pi$. In the sequel, we denote by $h_{st}$ the unique flat metric on $I$ with respect to which it has circumference $2 \pi.$

Let $(I,h)$ be doubly connected with Riemannian metric $h$. We call global cylindrical coordinates $(\rho,\theta)$ on $I,$
\[
a \leq \rho \leq b, \qquad 0 \leq \theta < 2\pi,
\]
\textbf{axially symmetric} if
\begin{equation}\label{eq:axsy}
h =d\rho^2+h_{\theta}(\rho)^2d\theta^2.
\end{equation}
We say $h$ is \textbf{axially symmetric} if $I$ has axially symmetric coordinates.
In this case, the conformal length of $(I,h)$ is given by
\begin{equation*}
Mod(I,h)=\int_a^b\frac1{h_{\theta}(\rho)}d\rho.
\end{equation*}

\begin{df}
Let $I$ be a doubly connected surface of conformal length $L$. Then there is a holomorphic map $f:[0,L]\times S^1\rightarrow I$ unique up to a rotation and a holomorphic reflection. For real numbers $a\leq b\in [0,L]$ with $a\leq L-b$, we write
\[
S(a,b;I):=f([a,b]\times S^1)\subset I,
\]
and
\[
C(a,b;I):=S(a,L-b;I).
\]
Note that composing $f$ with a holomorphic reflection of $[0,L]\times S^1$ replaces $S(a,b)$ with $S(L-b,L-a)$. The expression $C(a,a)$ is independent of the choice of $f$. A \textbf{subcylinder} of $I$ is a set of the form $I'=S(a,b;I)$.
\end{df}
\begin{df}\label{DfCofRad}
Let $U$ be a Riemann surface biholomorphic to the unit disk $D_1$. Let $h$ be a conformal metric on $U$ and let $z\in U$. Then there is a biholomorphism $\phi:U\rightarrow D_1$ with $\phi(z)=0,$ unique up to rotation. The \textbf{conformal radius of $U$ viewed from $z$} is defined to be
\[
r_{conf}(U,z;h):=1/\|d\phi(z)\|_h.
\]

\end{df}

Note that $r_{conf}(U,z;h)$ is not conformally invariant, since it depends on the metric at $z$. However, let $\nu_h$ denote the volume form of $h,$ let $\mu$ be an absolutely continuous measure on $U$ and denote by $\frac{d\mu(z)}{d\nu_{h}}$ the Radon-Nikodym derivative. Then the expression $\frac{d\mu(z)}{d\nu_{h}}r_{conf}^2(z)$ is conformally invariant.

The cases of interest for us will be conformal radii of geodesic disks with metrics of constant curvature $K,$ viewed from their center.
In these cases, the metric can be written in polar coordinates as
\begin{equation}\label{eq:pch}
h = d\rho^2 + h_\theta^2(\rho) d\theta^2,
\end{equation}
where
\begin{equation}\label{eq:forhth}
h_\theta(\rho) = \begin{cases}
\sinh(\rho), & K = -1, \\
\rho, & K = 0, \\
\sin(\rho), & K = 1.
\end{cases}
\end{equation}
So, the conformal radius of $B_r(p)$ viewed from $p$ is given by
\[
r_{conf} = \exp(f(r))
\]
where $f$ is the function defined by
\[
f'(r) = \frac{1}{h_\theta(r)}, \qquad f(r) = \log(r) + O(r) \text{  as $r \to 0$}.
\]
More explicitly,
\begin{equation}\label{eq:f}
f(r) = \log(r) + \int_0^r \left(\frac{1}{h_{\theta}(\rho)} - \frac{1}{\rho} \right)d\rho.
\end{equation}
It follows from equation~\eqref{eq:f} that
\begin{equation}\label{eq:bdrc1}
r_{conf} \geq r, \qquad K = 0,1,
\end{equation}
and for any $\kappa$ there exists a constant $c > 0$ such that
\begin{equation}\label{eq:bdrc2}
r_{conf} \geq c r, \qquad K = -1, \; r < \kappa.
\end{equation}

\begin{df}\label{dfCplxDbl}
 For any Riemann surface $\Sigma=(\Sigma,j)$, write $\overline{\Sigma}:=(\Sigma,-j)$. The \textbf{complex double} is the Riemann surface
 \[
 \Sigma_{\C} :=\Sigma\cup\overline{\Sigma},
  \]
 where the surfaces are glued together along the boundary by the identity. The complex structure on $\Sigma_{\C}$ is the unique one which coincides with $j$ and with $-j$ when restricted suitably. $\Sigma_{\C}$ is endowed with a natural antiholomorphic involution and for any $z\in\Sigma_{\C}$ we denote by $\overline{z}$ the image of $z$ under this involution. For more details about these constructions see \cite{AlGr}.
\end{df}

\begin{rem}
Note that for any connected Riemann surface $\Sigma,$ $\Sigma_{\C}$ is connected if and only if $\partial\Sigma\neq\emptyset$. Also, for any $\Sigma$, $\partial\Sigma_{\C}=\emptyset$.
\end{rem}

\begin{df}\label{df:cln}
Let $I\subset\Sigma_{\C}$ be doubly connected and conjugation invariant. If
\[
S\left(0,\frac1{2}Mod I;I\right)=\overline{S\left(\frac1{2}Mod I,Mod I;I\right)},
\]
we say the conjugation on $I$ is \textbf{latitudinal}. If for each $a,b\in [0,Mod I],$
\[
S(a,b;I)=\overline {S(a,b;I)},
 \]
we say the conjugation on $I$ is \textbf{longitudinal}.
\end{df}

\begin{lm}\label{lmMerLong}
Let $I\subset\Sigma_{\C}$ be doubly connected and conjugation invariant. Then the conjugation on $I$ is either latitudinal or longitudinal.
\end{lm}
\begin{proof}
The lemma is a consequence of the classification of the holomorphic automorphisms of the annulus and the fact that the composition of two anti-holomorphic automorphisms is holomorphic.
\end{proof}

For later reference we conclude with a statement of the thick thin decomposition for surfaces of genus $g>1$. In the following we assume the surfaces are endowed with their unique metric $h$ of constant curvature $-1.$

\begin{tm}\label{TmThTh}\cite[4.1.1]{Bu}
Let $S$ be a compact Riemann surface of genus $g\geq2$, and let $\gamma_1,...,\gamma_m$ be pairwise disjoint simple closed geodesics on $S$. Then the following hold:
\begin{enumerate}
\item
$m\leq 3g-3$.
\item
There exist simple closed geodesics $\gamma_{m+1},...,\gamma_{3g-3},$ which, together 	 with $\gamma_1,...,\gamma_m$, decompose S into pairs of pants.
\item\label{it:coll}
The collars
\begin{equation}
\mathcal{C}(\gamma_i)=\{p\in S|dist(p,\gamma_i)\leq w(\gamma_i)\}\notag
\end{equation}
of widths
\begin{equation}		 w(\gamma_i)=\sinh^{-1}\left(1/\sinh\left(\frac1{2}\ell(\gamma_i)\right)\right)\notag
\end{equation}
are pairwise disjoint for $i=1,...,3g-3$.
\item\label{it:spco}
Each $\mathcal{C}(\gamma_i)$ is isometric to the cylinder $[-w(\gamma_i),w(\gamma_i)]\times S^1$ with the Riemannian metric
\[
d\rho^2+\frac{\ell^2(\gamma_i)\cosh^2(\rho)}{4\pi^2} d\theta^2.
\]
\end{enumerate}
\end{tm}
Denote by $InjRad(S;h,p)$ the radius of injectivity of $S$ at $p\in S$, i.e. the supremum of all $r$ such that $B_r(p)$ is an embedded disk. If $h$ or $S$ is clear from the context, we may omit it from the notation.
\begin{tm}\label{TmThTh2}\cite[4.1.6]{Bu}
	Let $\beta_1,...,\beta_k$ be the set of all simple closed geodesics of length $\leq \sinh^{-1}1$ on $S$. Then $k\leq3g-3$ and the following hold.
	\begin{enumerate}
		\item The geodesics $\beta_1,...,\beta_k$ are pairwise disjoint.
		\item \label{it:lbd} $InjRad(S;p) > \sinh^{-1}1$ for all $p\in S-(\mathcal{C}(\beta_1)\cup...\cup\mathcal{C}(\beta_k))$.
		\item \label{it:fml} If $p\in\mathcal{C}(\beta_i)$,and $d=dist(p,\partial \mathcal{C}(\beta_i))$, then
\begin{align}\label{InjRadEq}
			\sinh (InjRad(S;p)) =\cosh\frac1{2}\ell(\beta_i)\cosh d-\sinh d.
		\end{align}
	\end{enumerate}
\end{tm}

\section{Thick thin measure}\label{SecThickThin}
For the rest of the discussion, fix constants $c_1,c_2,c_3,\delta_1,\delta_2>0$. Without loss of generality we will assume that $c_3\leq 1$ and that $\delta_2<\delta_1$. Given a metric $h$ on a measured Riemann surface $(\Sigma,j,\mu)$ we denote by $\frac{d\mu}{d\nu_h}$ the Radon Nikodym derivative of $\mu$ with respect to $\nu_h$, the volume form induced by $h$.
\begin{df}
Let $\Sigma$ be a Riemann surface. A subset $S\subset\Sigma_{\C}$ is said to be \textbf{clean} if either $S=\overline{S}$ or $S\cap \overline{S}=\emptyset$.
\end{df}
\begin{df}\label{dfThTh}
Let $(\Sigma,j)$ be a Riemann surface, possibly bordered. Let $\mu$ be a finite measure on $\Sigma$ and extend $\mu$ to a measure on $\Sigma_{\C}$ by reflection i.e.
\[
\mu(U):=\mu(\overline{U}),
\]
for $U\subset\overline{\Sigma}$  a measurable set. $\mu$  will be called \textbf{thick thin} if it satisfies the following conditions:
\begin{enumerate}
\item $\mu$ is absolutely continuous and has a continuous density $\frac{d\mu}{d\nu_h}$,  where $h$ is any Riemannian metric on $\Sigma_{\C}$.
\item \textbf{Gradient inequality.} Let $U\subset\Sigma_{\C}$ be a simply connected domain and let $z\in U$. Then for any conformal metric $h$ on $(\Sigma_{\C},j)$,
\begin{align}
\notag \mu(U)<\delta_1\quad \Rightarrow \quad &\frac{d\mu}{d\nu_h}(z)\leq c_1\frac{\mu(U)}{r_{conf}^2},
\end{align}
where $r_{conf}=r_{conf}(U,z;h)$.
\item \textbf{Cylinder inequality.} Let $I\subset\Sigma_{\C}$ be a doubly connected domain so that $Mod(I)> 2c_2$ and satisfying either $I\subset\Sigma$ or $I=\overline{I}$.
    Then
\begin{align}
\notag \mu\{I\}<\delta_2&\Rightarrow &\mu\{C(t,t;I)\}\leq e^{-c_3t}\mu\{I\},\notag\\&&\forall t\in(c_2,\frac1{2}Mod(I))\notag.
\end{align}
\end{enumerate}
\end{df}

\begin{rem}\label{rmGrad}
Let $\mu$ be a thick thin measure on $\Sigma,$ and let $h$ be a conformal metric of constant curvature $K = 0,\pm 1$ on $\Sigma_\C$. By inequalities~\eqref{eq:bdrc1} and~\eqref{eq:bdrc2}, there is a constant $c'_1$ depending linearly on $c_1$ such that for any $z\in\Sigma$ and $r\in(0,\min(\sinh^{-1}(1),InjRad(\Sigma;h,z)))$,
\begin{align}
\mu(B_r(z;h))<\delta_1\Rightarrow &\frac{d\mu}{d\nu_h}(z)\leq c'_1\frac{\mu(B_r(z;h))}{r^2}.
\end{align}
\end{rem}
\begin{rem}\label{rmGradApp}
Let $\Sigma,\mu$ and $h$ be as in Remark~\ref{rmGrad}.
We will apply the gradient inequality in the following way. For any point $z \in \Sigma_\C$, let $d= \frac{d\mu}{d\nu_h}(z)$ and let
\begin{align}\label{EqArDee}
r_d:=\sqrt{\frac{c'_1\delta_1}{d}}.
\end{align}
Suppose
\begin{equation*}
r_d \in (0,\min(\sinh^{-1}(1),InjRad(\Sigma;h,z))).
\end{equation*}

Then
\[
\mu(B_{r_d}(z;h))\geq\delta_1.
\]
Moreover, we have
\[
\frac{d\mu}{d\nu_h}(z) \leq c_1' \frac{\mu(B_r(z;h))}{r^2}, \qquad r \leq r_d.
\]

To simplify our formulas, we always scale $\mu$ so that $c'_1\delta_1=1$.
\end{rem}

We denote by $\mathcal{M}=\mathcal{M}(c_1,c_2,c_3,\delta_1,\delta_2)$ the family of measured Riemann surfaces $(\Sigma,j,\mu)$ such that $\mu$ is thick-thin.

\begin{lm}\label{ExpCylDEst}
There is a constant $a$ with the following significance. Let $(\Sigma,j,\mu)\in\mathcal{M}$. Let $I\subset\Sigma_{\C}$ be clean and doubly connected, and let $h = h_{st}.$ Suppose $\mu(I)<\delta_2$. Let $z\in C(c_2+\pi,c_2+\pi;I)$ be a point with cylindrical coordinates
\[
(s,t)\in\left[-\frac1{2}Mod(I)+c_2+\pi,\frac1{2}Mod(I)-c_2-\pi\right]\times S^1.
\]
Then,
\begin{align}
 \frac{d\mu}{d\nu_{h}}(z)<ae^{-c_3(\frac1{2}Mod(I)-|s|)}\mu(I).
\end{align}
\end{lm}
\begin{proof}
Combining the gradient inequality and the cylinder inequality,
\begin{align}
 \frac{d\mu}{d\nu_{h}}(z)&\leq \frac{c_1}{\pi^2}\mu\left([s-\pi,s+\pi]\times S^1\right)\\&\leq \frac{c_1}{\pi^2}\mu\left([-|s|-\pi,|s|+\pi]\times S^1\right)\notag\\&\leq \frac{c_1}{\pi^2}e^{-c_3(\frac1{2}Mod(I)-\pi-|s|)}\mu(I)\notag.
\end{align}
\end{proof}

\section{A priori bound}\label{SecState}

\begin{df}
Let $(M,h)$ be a Riemannian manifold, let $N$ be a totally geodesic submanifold possibly with boundary, and let $h_{N}$ be the induced metric on $N$. Let $p\in N.$ Define the \textbf{segment width} by
\begin{align*}
&SegWidth(N,p;h) = \\
& = \sup \{s > 0| B_r(p;(M,h))\cap N = B_r(p;(N,h_{N}))\text{ for all $r < s$}\}.
\end{align*}
\end{df}

In the following, an \textbf{embedded geodesic} is a one-dimensional totally geodesic submanifold possibly with boundary. Let $(M,h)$ be a Riemannian manifold. For $\gamma$ an embedded geodesic, we denote by $d\ell_h$ the line element, or volume form, of the induced metric on $\gamma$. Now, consider the special case when $M$ is a Riemann surface $\Sigma$. Let $\mu$ be a measure and $h$ a conformal metric on $\Sigma$. Define
\[
d\ell_\mu = \left .\sqrt{\frac{d\mu}{d\nu_h}}\right|_{\gamma} d\ell_h.
\]
It easy to see that $d\ell_\mu$ is independent of $h.$ If $\gamma$ is compact, define
\[
\ell_\mu(\gamma) = \int_\gamma d\ell_\mu.
\]

Let $\Sigma$ be a Riemann surface possibly with boundary. For the rest of the paper, denote by $h_{can}$ the unique conformal metric on $\Sigma_\C$ satisfying the following conditions. If $\chi(\Sigma_\C) \neq 0,$ then $h_{can}$ has constant curvature $\pm 1.$ If $\chi(\Sigma_\C) = 0,$ then $h_{can}$ has constant curvature $0$ and $\nu_{h_{can}}(\Sigma) = 1.$
\begin{tm}\label{TmApLenBound}
There are constants $b_1$ and $b_2$ with the following significance. Let $(\Sigma,\mu)\in\mathcal{M}$, let $\gamma\subset\Sigma_{\C}$ be a compact embedded conjugation invariant geodesic, and let $k\geq1$ be a constant such that for any $x\in\gamma,$
\begin{equation}\label{eq:swa}
SegWidth(\gamma,x;h_{can})>\frac1{k}InjRad(\Sigma_\C;h_{can},x).
\end{equation}
Then
\begin{equation}\label{eq:meq}
\ell_{\mu}(\gamma)\leq {k^2}\{b_1\mu(\Sigma_{\C})+b_2genus(\Sigma_{\C})\}.
\end{equation}
\end{tm}

For the rest of this discussion up to and including the proof of Theorem \ref{TmApLenBound}, we fix $\gamma$ and $k$.
\begin{rem}\label{rmScaling}
Recall the definition of $c_1'$ from Remark~\ref{rmGrad}. In the proof of Theorem~\ref{TmApLenBound}, without loss of generality, we may assume the constants $c_1',\delta_1,$ pertaining to the definition of thick-thin satisfy $c_1'\delta_1 = 1.$ This is true for two reasons. First, for $\tilde c_1 \geq c_1$ and $\tilde \delta_1 \leq \delta_1,$ we have
\[
\mathcal M(c_1,c_2,c_3,\delta_1,\delta_2) \subset \mathcal M(\tilde c_1, c_2,c_3,\tilde\delta_1,\delta_2).
\]
Such $\tilde c_1,\tilde\delta_1,$ can always be chosen so that $\tilde c_1' \tilde \delta_1 = 1.$ However, this will not yield the optimal constant $c$ for a given $c_1,\delta_1.$ To obtain the optimal value of $c,$ it is useful to note that for $\lambda > 0,$ the map
\[
\mathcal M(c_1,c_2,c_3,\delta_1,\delta_2) \rightarrow \mathcal M(c_1,c_2,c_3,\lambda \delta_1,\lambda\delta_2)
\]
given by $(\Sigma,j,\mu) \mapsto (\Sigma,j,\lambda \mu)$ scales the constants $b_i$ for $i=1,2$, by $b_1 \mapsto b_1/\sqrt{\lambda}$ and $b_2\mapsto \sqrt{\lambda}b_2$.
\end{rem}

For any metric $n$ on $\gamma,$ denote by $d\ell_{n}$ the line element. By definition,
\[
\ell_{\mu}(\gamma)=\int_{x\in\gamma}\frac{d\ell_{\mu}}{d\ell_{n}}d\ell_{n}(x).
\]
We derive Theorem~\ref{TmApLenBound} by studying the graph of the function
\[
g:=\ln\frac{d\ell_{\mu}}{d\ell_{n}}:\gamma\rightarrow(-\infty,\infty)
\]
for a convenient choice of the metric $n$. We define $n$ as follows. For any $x\in\gamma$, let
\[
r(x):=\min(\sinh^{-1}(1),InjRad(\Sigma;h_{can},x)).
\]
It turns out that for dealing with higher genus, where there is no a priori bound on the radius of injectivity of $\Sigma$, it is convenient to use the metric $n=\frac1{r(x)}h_{can}|_{\gamma}$.

We use the normalized metric $n$ only on $\gamma$. On $\Sigma_{\C}$ we continue to use the standard metric $h_{can}$. To translate from estimates in terms of the one to estimates in terms of the other metric, we will use the following lemma.
\begin{lm}\label{lm:cpmet}
Let $x_1,x_2\in\gamma$ such that $d_{{\gamma}}(x_1,x_2;h_{can})\leq r(x_1)/2.$ Then
\begin{align}\label{NormEstimate}
\frac{2d_{\gamma}(x_1,x_2;h_{can})}{3r(x_1)}\leq d_{\gamma}(x_1,x_2;h_{n})\leq \frac{2d_{\gamma}(x_1,x_2;h_{can})}{r(x_1)}.
\end{align}
\end{lm}
\begin{lm}\label{eq:drl1}
For all $t$ such that $r(\gamma(t))$ is differentiable,
\begin{equation}
\frac{dr(\gamma(t))}{dt}\leq 1.
\end{equation}
\end{lm}
\begin{proof}
If $h_{can}$ has curvature $-1,$ inequality~\eqref{eq:drl1} follows from Theorem~\ref{TmThTh2}\ref{it:lbd} and~\ref{it:fml}. If $h_{can}$ has non-negative curvature, then $\frac{dr(\gamma(t))}{dt} = 0.$
\end{proof}

\begin{proof}[Proof of Lemma \ref{lm:cpmet}]
Write $\Delta x = d_\gamma(x_1,x_2;h_{can}).$ Parameterize $\gamma$ by $h_{can}$-length so that $\gamma(0) = x_1$ and $\gamma(\Delta x) = x_2.$ It is easy to see that $r(\gamma(t))$ is piecewise smooth and thus differentiable almost everywhere with respect to $t.$  Applying Lemma \ref{eq:drl1} we calculate
\begin{align}
d_\gamma(x_{1},x_2;h_n)&= \left|\int_{0}^{\Delta x}\frac1{r(\gamma(t))}dt \right|\\
&\leq \frac{\Delta x}{\inf_{t\in[0,\Delta x]}r(\gamma(t))}\notag\\
&\leq \frac{\Delta x}{r(x_1)-\Delta x\ess\sup_{t\in[0,\Delta x]}\left|\frac{dr(\gamma(t))}{dt}\right|}\notag\\
&\leq \frac{2\Delta x}{r(x_1)}.\notag
\end{align}
For the last inequality we have used~\eqref{eq:drl1}. The upper bound of estimate~\eqref{NormEstimate} follows. A similar argument gives the lower bound.
\end{proof}

\begin{df}\label{dfDB}
Denote
\[
D:=\{(x,t)\in\gamma\times[\ln 2k,\infty)|\ln 2k\leq t\leq g(x)\}.
 \]
For any  $(x,t)\in D,$ denote $B(x,t):=B_{e^{-t}r(x)}(x;\Sigma,h_{can}).$
\end{df}
\begin{lm}\label{lmDDisc}
For $(x,t) \in D,$ we have $\mu(B(x,t))\geq\delta_1$.
\end{lm}
\begin{proof}
Since $k\geq 1$, we have $t>0.$ Therefore, $B(x,t)$ is an embedded disk.  By Remark \ref{rmGradApp}, $\mu(B(x,t))\geq \delta_1$.
\end{proof}
\begin{lm}\label{lmDisjDisc}
Let $(x_i,t_i)\in D$ for $i=1,2$. Suppose
\[
d_{\gamma}(x_1,x_2;h_n)>2(e^{-t_1}+e^{-t_2}).
\]
Then $B(x_1,t_1)\cap B(x_t,t_2)=\emptyset$.
\end{lm}
\begin{proof}
By Lemma~\ref{lm:cpmet} we have
\[
d_{\gamma}(x_1,x_2;h_{{can}})\geq e^{-t_1}r(x_1)+e^{-t_2}r(x_2).
\]
Since $t_i\geq\ln2k,$ the assumption of Theorem \ref{TmApLenBound} implies that
\[
SegWidth(\gamma,x_i;h_{can}) \geq 2 e^{-t_i}r(x_i).
\]
The claim now follows.
\end{proof}

\section{Partitions of hypographs}\label{SecTreePart}
Let $\gamma$ be a 1-dimensional manifold, let $f:\gamma\rightarrow[0,\infty)$ be a continuous function, and let $E$ be the hypograph of $f$. That is, $E$ is the set of points under the graph of $f$ in $\gamma\times [0,\infty)$. In this section we introduce a binary relation on subsets of $E$, which should be thought of intuitively as the relation of lying above. We prove two basic theorems about this order relation. Theorem \ref {TmTreeEsat} states that for any partition $P$ of $E$ into connected subsets by intersecting $E$ with horizontal segments, the binary relation on the elements of $P$ is a tree-like partial order. See Figure \ref{Fig1}. Theorem \ref{tmMinMaxTree} states that there is a particular such partition, denoted $\mathcal{T}_E$, such that the branchings in the tree associated with $\mathcal{T}_E$ correspond to local minima in the graph of $f$. See Figure \ref{Fig2}. After proving these theorems, we show that continuity of $f$ allows us two control the number of elements of $\mathcal{T}_E$ by the number of its maximal elements. Note that in general $\mathcal{T}_E$ might be infinite, and if $f$ is not continuous, there might not be any maximal elements.

\begin{figure}[h]
\centering
\includegraphics[scale=0.5]{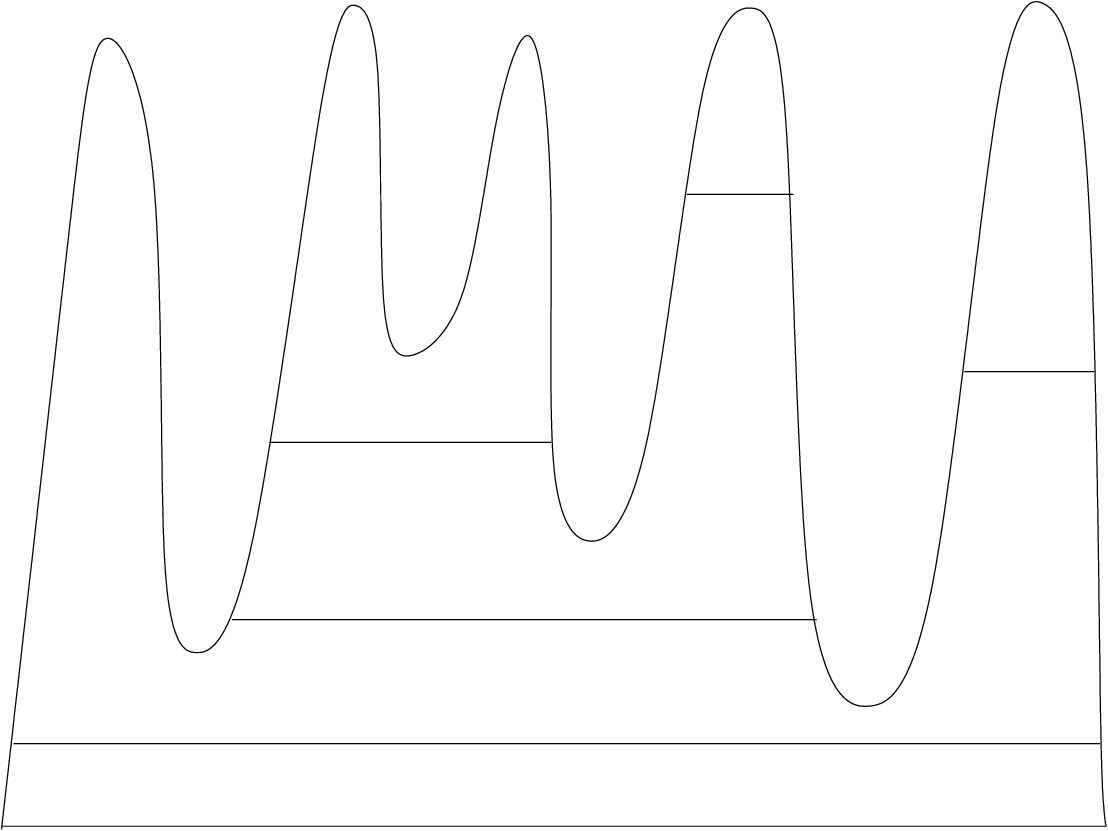}
\caption{}\label{Fig1}
\end{figure}

\begin{figure}[h]
\includegraphics[scale=0.5]{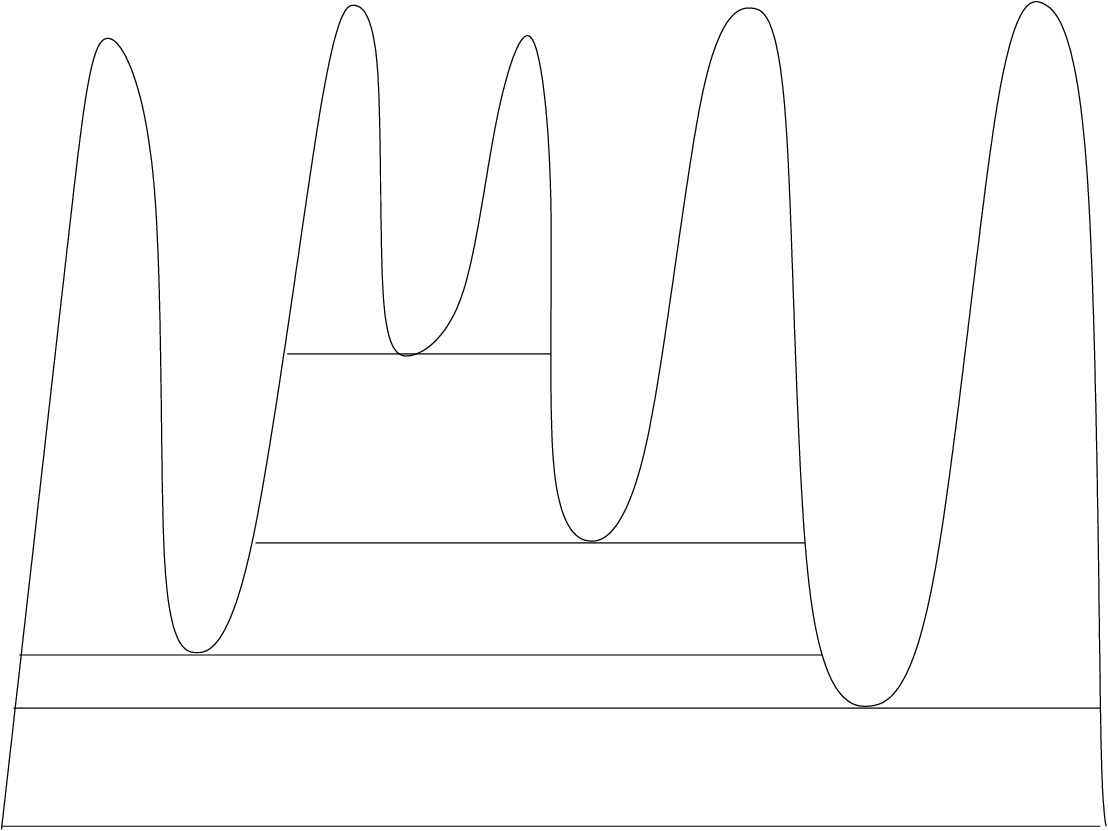}
\caption{}\label{Fig2}
\end{figure}

\subsection{A binary relation}
Let $\gamma$ be a compact 1-dimensional manifold with or without boundary. Let $\xi\in\mathbb{R}$, write $X=\gamma\times[\xi,\infty),$ and denote by $p_1:X\rightarrow\gamma$  and $p_2:X\rightarrow [\xi,\infty)$ the canonical projections. Denote $X_t:=\gamma\times\{t\}$. For any subset $S\subset X$ denote $S_t:=X_t\cap S$. If $p_2(S)\subset[\xi,\infty)$ is bounded, denote
\[
T_f(S):=\sup\{ p_2(S)\},
\]
\[
T_i(S):=\inf\{ p_2(S)\},
\]
and $T(S):=T_f(S)-T_i(S)$.

Let $f:\gamma\rightarrow [\xi,\infty)$ be a continuous function. Denote the region under the graph of $f$ by
\[
E:=\{y\in X|p_2(y)\leq f(p_1(y))\}.
 \]
For a topological space $Y,$ denote by $\pi_0(Y)$ the set of path-connected components.
An \textbf{$E$-segment} is an element of $\cup_{t\in[\xi,\infty)}\pi_0(E_t)$. For any $t\geq \xi$, and for any $x\in p_1(E_t)$ we denote by $e(x,t)$ the $E$-segment containing $(x,t)$.

\begin{rem}\label{rem:cont}
It follows from the continuity of $f$ that $E$ is a closed set. So, all $E$-segments are closed.
It also follows from the continuity of $f$ that if $e$ is an $E$-segment, $x$ is not a boundary point of $\gamma,$ and $(x,t)$ is a boundary point of $e,$ then $f(x)=t$.
\end{rem}

We define a relation on the power set $P(E)$ as follows. Let $S_1,S_2\subset E.$ We say that $S_1\leq_1 S_2$ if $p_1(S_2)\subseteq p_1(S_1)$. We say that $S_1\leq_2 S_2$ if  for any $t\in  p_2(S_2)$ there is a $t'\in p_2(S_1)$ such that $t'\leq t$. Finally we say that $S_1\leq S_2$ is $S_1\leq_1 S_2$ and $S_1\leq_2 S_2$.

The following properties of $\leq$ are obvious and are stated without proof.
\begin{lm}\label{lm:ob}
\begin{enumerate}
\item
The relation $\leq$ is reflexive and transitive.

\item\label{it:ob2}
Let $\Pi\subset P(E)$ be the collection of subsets of the form
\[
s_1\times \{t\},
\]
where $s_1\subset\gamma$, $t\in [\xi,\infty)$. The restriction of $\leq$ to $\Pi$ is antisymmetric. $\Pi$ contains the singletons of $E$ and the $E$-segments.
\item
For any two sets $S_1,S_2\in P(E)$, $S_1\leq S_2$ if and only if for any $(x,t)\in S_2$, $S_1\leq \{(x,t)\}$
\end{enumerate}
\end{lm}

\begin{lm}\label{lmEsegChar}
Let $\xi\leq t\leq t_1,t_2$, let $x\in p_1(E_t)$, $x_i\in p_1(E_{t_i})$ for $i=1,2$ and denote $e_i=e(x_i,t_i)$.
\begin{enumerate}
\item\label{i1}
For any $t'\in[\xi,t]$, $e(x,t')$ is well defined.
\item\label{i2}
If $e(x_1,t)=e(x_2,t)$ then for any $t'\in[\xi,t]$, $e(x_1,t')=e(x_2,t')$.
\item\label{i3}
For any $t'\in[\xi,t]$, $e(x,t')\leq e(x,t)$.
\item \label{i5}
$e_1\leq e_2$ if and only if $t_1\leq t_2$ and $e(x_2,t_1)=e_1$.
\item \label{i4}
If $e_1$ and $e_2$ are incomparable with respect to $\leq$, then
\[
d(p_1(e_1),p_1(e_2))>0.
\]
\item\label{i6}
If $e(x_1,t_1)\leq e(x_2,t_2),$ then for any $t\in [t_1,t_2]$, $e(x_1,t_1)\leq e(x_2,t)$.
\end{enumerate}
\end{lm}
\begin{proof}
\begin{enumerate}
\item
We have $t'\leq t\leq f(x)$, so $(x,t')\in E$.
\item
Let $S=p_1(e(x_1,t))$ then $S$ is a segment which by assumption contains $x_1$ and $x_2$. Since $S\times\{t\}\subset E_t$, for all $t'\in[\xi,t]$, and all $x\in S$, $t'\leq f(x).$ Therefore, $S\times\{t'\}\subset E_{t'}$. $S\times\{t'\}$ is connected and contains $(x_i,t')$ for $i=1,2$. In particular $e(x_1,t')=e(x_2,t')$.
\item
It is clear that $e(x,t')\leq_2 e(x,t)$. If $x'\in p_1(e(x,t)),$ then $e(x',t)=e(x,t)$. By \ref{i2}, $e(x',t')=e(x,t').$ In particular, $x'\in p_1(e(x,t'))$. Thus $e(x,t')\leq_1e(x,t)$.
\item
Assume first that $e_1\leq e_2.$ Then $t_1\leq t_2$ by definition. Further, $x_2\in p_1(e_1),$ so $e(x_2,t_1)\subset e_1$. But $e_1$ is an $E$-segment, so $e(x_1,t_2)=e_1$ as required. Assume now that $e(x_2,t_1)=e_1$, and $t_1\leq t_2$. Then by \ref{i3}, $e_1\leq e(x_2,t_2)$.
\item
Assume without loss of generality that $t_1\leq t_2$ and denote $e'_2=e(x_2,t_1)$. By Remark~\ref{rem:cont}, $E_{t_1}$ is closed. So, since $e'_2$ and $e_1$ are both connected components of $E_{t_1},$ we have that either
$e'_2 = e_1$ or $d( p_1(e'_2), p_1(e_1))>0$. Thus, by~\ref{i5}, $d( p_1(e'_2), p_1(e_1))>0$. By~\ref{i3}, $e'_2\leq e_2.$ In particular, $p_1(e_2)\subset p_1(e'_2)$, so
\[
d( p_1(e_1), p_1(e_2))>d( p_1(e'_2), p_1(e_1))>0.
\]
\item
Using \ref{i5} twice, $e(x_1,t_1)\leq e(x_2,t_2)$ implies $e(x_2,t_1)=e(x_1,t_1),$ which implies $e(x_1,t_1)\leq e(x_2,t)$.
\end{enumerate}
\end{proof}

\subsection{Tree-like partial order}

\begin{df}
Let $S\subset E.$ $S$ is said to be \textbf{$E$-saturated} if $S$ is a union of $E$-segments.
\end{df}
\begin{rem}\label{rmEsat}
Clearly, any union or intersection of $E$-saturated sets is $E$-saturated. Moreover, a connected component of an $E$-saturated set and the complement in $E$ of an $E$-saturated set are $E$-saturated.
\end{rem}
\begin{df}
Let $S$ be a set. A tree-like order relation on $S$ is a partial order relation $\leq$ which satisfies for any $v,v_1,v_2\in S,$
\begin{align}
v_1\leq v \text{ and } v_2\leq v \quad \Rightarrow \quad  v_1\leq v_2 \text{ or } v_2\leq v_1.\notag
\end{align}
\end{df}
\begin{tm}\label{TmTreeEsat}
Let $P$ be a collection of pairwise disjoint $E$-saturated connected sets. Then the restriction of $\leq$ to $P$ is a tree-like order relation.
\end{tm}
For the proof of Theorem \ref{TmTreeEsat}, we first prove a few lemmas.
\begin{lm}\label{lm3rdpoint}
Let $S\subset E$ be connected. For any compact set $K \subset S$ there exists a point $(x,t)\in S$ such that $e(x,t)\leq K.$
\end{lm}
\begin{proof}

Using the compactness of $K,$ choose $(x_1,t_1) \in K$ such that $t_1 = T_{i}(K).$ Let $L \subset p_1(S)$ be a connected compact subset containing $p_1(K).$ Using the continuity of $f,$ choose $x_2 \in L$ such that
\[
f(x_2) = \inf_{y \in L} f(y).
\]
Since $x_2 \in p_1(S),$ and $S \subset E,$ there exists $t_2 \leq f(x_2)$ such that $(x_2,t_2) \in S.$ Choose $i$ such that $t_i = \min(t_1,t_2)$ and set $(x,t) = (x_i,t_i).$ Clearly, $(x,t) \in S.$ Since $t \leq t_2,$ we have
\[
L\times\{t\} \subset E.
\]
So, since $L \times\{t\}$ is connected and contains $(x,t),$ we have
\[
L \times\{t\} \subset e(x,t).
\]
Therefore, since $p_1(K) \subset L$ and $t \leq t_1,$ we have $e(x,t) \leq K.$

\end{proof}

\begin{lm}\label{lm:sco}
Let $S_1 \subset E$ be $E$-saturated and let $S_2 \subset E$ be connected such that $S_1 \cap S_2 = \emptyset.$ Suppose there exist $x \in \gamma$ and $t_1 < t_2 \in [\xi,\infty)$ such that $(x,t_i) \in S_i$ for $i =1,2.$ Then $S_1 \leq S_2.$
\end{lm}
\begin{proof}
Let $e = e(x,t_1).$ Since $S_1$ is $E$-saturated, we have $e \subset S_1.$ So,
\begin{equation}\label{eq:eS2e}
e \cap S_2 = \emptyset.
\end{equation}
Let $x_1,x_2 \in \gamma$ be the boundary points of $p_1(e).$ By Remark~\ref{rem:cont},
\begin{equation}\label{eq:fxit}
f(x_i) = t_1, \qquad \text{for $i \in \{1,2\}$ such that $x_i \notin \partial\gamma.$}
\end{equation}
Define
\[
r_i =
\begin{cases}
\{x_i\} \times (t_1,\infty), & x_i \notin \partial \gamma \\
\emptyset & x_i \in \partial \gamma.
\end{cases}
\]
By equation~\eqref{eq:fxit}, the rays $r_i$ are disjoint from $E.$ In particular,
\begin{equation}\label{eq:Sri}
S_2 \cap r_i = \emptyset, \qquad i = 1,2.
\end{equation}
Define disjoint open sets $U_1$ and $U_2$ by
\begin{align*}
&\qquad\quad U_2 = \{ (x,s) \in X | x \in (x_1,x_2) \cup (\partial \gamma \cap \{x_1,x_2\})\text{ and } s > t_1 \}, \\
&\qquad \quad U_1 = X \setminus \overline U_2.
\end{align*}
Clearly,
\[
U_1 \cup U_2 = X \setminus(e \cup r_1 \cup r_2).
\]
So, by equations~\eqref{eq:eS2e} and~\eqref{eq:Sri}, $S_2 \subset U_1 \cup U_2.$ Since $(x,t_2) \in S_2 \subset E,$ we have $f(x)\geq t_2 > t_1.$ Thus by equation~\eqref{eq:fxit}, we have $x \notin \{x_1,x_2\}\setminus \partial \gamma.$ So, by definition of $U_2,$ we have $(x,t_2) \in U_2.$ Therefore, $S_2 \cap U_2 \neq \emptyset.$ Since $S_2$ is connected, it follows that $S_2 \subset U_2.$ So, $U_2 \leq S_2.$ By definition of $U_2,$ we have $e \leq U_2.$ Since $e \subset S_1,$ we have $S_1 \leq e.$ Combining the foregoing inequalities, we have
\[
S_1 \leq e \leq U_2 \leq S_2,
\]
which proves the lemma.
\end{proof}

\begin{figure}[h]
\centering
\includegraphics[scale=0.7]{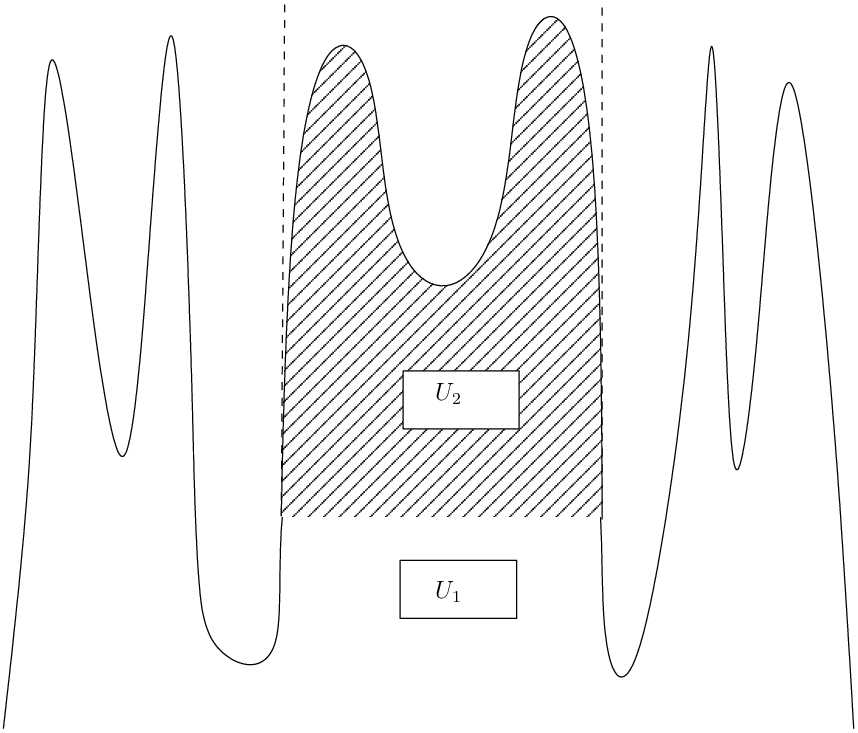}
\caption{}\label{Fig3}
\end{figure}

\begin{lm}\label{lmEsat}
Let $S\subset E$ be $E$-saturated and connected.
\begin{enumerate}
\item\label{lmEsat2}
Let $e_1, e_2\subset S$ be $E$-segments. If $e_3$ is an $E$-segment such that $e_1\leq e_3\leq e_2$, then $e_3\subset S$.
\item\label{lmEsat1}
For any $t_1\leq t_2\in p_2(S)$, $S_{t_1}\leq S_{t_2}$.
\end{enumerate}
\end{lm}
\begin{proof}
\begin{enumerate}
\item
Suppose that $e_3 \not \subset S.$ Since $S$ is $E$-saturated it follows that
\begin{equation}\label{eq:e3s}
e_3 \cap S = \emptyset.
\end{equation}
Take $S_1 = e_3$ and $S_2 = S.$ Set $t_0 = p_2(e_1), t_1 = p_2(e_3)$ and $t_2 = p_2(e_2).$ Since $e_1,e_2 \subset S,$ by equation~\eqref{eq:e3s} we have $t_0 < t_1 < t_2.$ Choose $x \in p_2(e_2).$ Then $S_1,S_2,x,t_1,t_2$ satisfy the hypotheses of Lemma~\ref{lm:sco}. We conclude that $e_3 \leq S.$ So,
\[
p_2(e_1) = t_0 < t_1 = p_2(e_3) \leq T_i(S)
\]
contradicting the assumption that $e_1 \subset S.$
\item
By Lemma~\ref{lm3rdpoint} with $K = S_{t_1} \cup S_{t_2},$ there exists $(x,t) \in S$ such that $e(x,t) \leq S_{t_i}$ for $i = 1,2.$ Since $S$ is $E$-saturated, $e(x,t) \subset S.$ Let $y \in p_1(S_{t_2}).$ Since $S$ is $E$-saturated, $e(y,t_2) \subset S.$ In particular, $e(x,t) \leq e(y,t_2).$ So, by Lemma~\ref{lmEsegChar}\ref{i6}, $e(x,t) \leq e(y,t_1).$ By Lemma~\ref{lmEsegChar}\ref{i3}, $e(y,t_1) \leq e(y,t_2).$ Therefore, by~\eqref{lmEsat2} we have $e(y,t_1) \subset S.$ Since $y \in p_1(S_{t_2})$ was arbitrary, it follows that $S_{t_1} \leq S_{t_2}.$
\end{enumerate}
\end{proof}

\begin{cy}\label{cyAntiSym}
Let $F$ be a collection of pairwise disjoint connected $E$-saturated sets. Then the restriction of the relation $\leq$ to $F$ is antisymmetric.
\end{cy}

\begin{proof}
Let $S_1,S_2\in F$ and let $x\in p_1(S_1)$. Assume $S_1\leq S_2$ and $S_2\leq S_1$. Then in particular, $ p_1(S_1)= p_1(S_2)$ so that $x\in p_1(S_2)$. Therefore, there are $t_1,t_2\in  p_2(S_1)$ such that $(x,t_1)\in S_1$ and $(x,t_2)\in S_2$. Assume without loss of generality that $t_1\leq t_2$. Since $S_2\leq_2 S_1$ there is a $t\in p_2(S_2)$ with $t\leq t_1$. Since $S_2$ is connected, $p_2(S_2)$ is connected and so $[t,t_2]\subset p_2(S_2)$. In particular $t_1\in p_2(S_2)$. By Lemma \ref{lmEsat}\ref{lmEsat1}, $S_{2,t_1}\leq S_{2,t_2},$ so  $(x,t_1)\in S_2$. Thus $S_1\cap S_2$ is nonempty, and by the assumption on $F$, $S_1=S_2$.
\end{proof}

\begin{lm}\label{lmEitherOrSat}
Let $S_1,S_2\subset E$. Suppose that $S_1$ is $E$-saturated and connected, that $S_2$ is connected and that $S_1\leq_2 S_2$. Then either $S_1\leq S_2$ or
\[
 p_1(S_1)\cap p_1(S_2)=\emptyset.
\]
\end{lm}

\begin{proof}
Suppose that there is a point $x_1\in p_1(S_1)\cap p_1(S_2)$. Then there is a $t_1\in p_2(S_2)$ such that $(x_1,t_1)\in S_2$. Let $(x_2,t_2)\in S_2$. By Lemma \ref{lm3rdpoint} with $K = (x_1,t_1) \cup (x_2,t_2),$ there is an $(x_3,t_3)\in S_2$ such that $e(x_3,t_3)\leq (x_i,t_i)$ for $i=1,2$. In particular,
\[
 p_1(e(x_3,t_3))\cap p_1(S_1)\neq\emptyset.
\]
Let $t\in p_2(S_1)$ such that $t\leq \min(t_1,t_3)$. Such a $t$ exists by the assumption $S_1\leq_2 S_2$. By Lemma \ref{lmEsegChar}\ref{i4} and the fact that $t\leq t_3$, $e(x_1,t)\leq e(x_3,t_3)$. By Lemma~\ref{lmEsat}\ref{lmEsat1} and the fact that $t \leq t_1,$ we have $S_t \leq S_{t_1}.$ In particular, $(x_1,t) \in S_1.$ Since $S_1$ is $E$-saturated, $e(x_1,t)\subset S_1.$ Therefore,
\[
S_1\leq e(x_1,t)\leq e(x_3,t_3)\leq \{(x_2,t_2)\}.
\]
But $(x_2,t_2)$ was an arbitrary point of $S_2,$ so the claim follows.
\end{proof}

\begin{lm}\label{lmNoCycles}
Let $S_1,S_2 \subset E$ be connected and let $S_1$ be $E$-saturated.
Suppose $S_1\leq_2 S_2$. If there is a nonempty set $S\subset E$ such that $S_i\leq S$ for $i=1,2,$ then $S_1\leq S_2$.

\end{lm}
\begin{proof}
By assumption, $ p_1(S_1)\cap p_1(S_2)\neq\emptyset$. Therefore, by Lemma \ref{lmEitherOrSat} $S_1\leq S_2$.

\end{proof}

\begin{proof}[\textbf{Proof of Theorem \ref{TmTreeEsat}}]
By Cor. \ref{cyAntiSym}, $\leq$ is an order relation when restricted to $P$. By Lemma \ref{lmNoCycles}, this order is tree-like.
\end{proof}

\subsection{Equivalence relation}
\begin{df}
A \textbf{branching point} is a point $(x,t)\in\partial E$ that is contained in an open segment $s\subset e(x,t)$ such that $\partial s\subset E^o$. Here, $E^o$ denotes the interior of $E$.
\end{df}
\begin{df}\label{df:gal}
Let $(x_i,t_i) \in E$ for $i = 1,2.$
If $t_1\leq t_2,$ we say that $(x_1,t_1)\sim(x_2,t_2)$ if the following two conditions hold:
\begin{enumerate}
\item
$e(x_1,t_1)\leq e(x_2,t_2)$.
\item
The rectangle $R=p_1(e(x_1,t_1))\times[t_1,t_2)$ contains no branching points.
\end{enumerate}
If $t_2<t_1,$ we reverse the roles of $t_1$ and $t_2$.
\end{df}

\begin{lm}\label{lm:er}
$\sim$ is an equivalence relation.
\end{lm}

Before proving Lemma~\ref{lm:er}, we prove the following preparatory lemma.

\begin{lm}\label{lmNoBp}
Let $t_1>\xi$ and $t_2\geq t_1$. Let $x\in p_1(E_{t_2})$. Let $R= p_1(e(x,t_1))\times[t_1,t_2)$ and $\overline{R}=p_1(e(x,t_1))\times[t_1,t_2]$. Assume that $R$ contains no branching points.
\begin{enumerate}
\item\label{Bp1}
$(\overline{R}\cap E)_{t_2}$ is connected.
\item\label{Bp2}
The order $\leq$ on the set of $E$-segments contained in $\overline{R}$ is linear.
\item\label{Bp3}
\[
( p_1(e(x,t_2))\times[t_2,\infty))\cap E=( p_1(e(x,t_1))\times[t_2,\infty))\cap E
\]
\end{enumerate}
\end{lm}

\begin{proof}
\begin{enumerate}
\item
We prove this by contradiction. Choose an orientation on $\gamma$ so that intervals between points on $\gamma$ are well defined. Let $e_1$ and $e_2$ be distinct connected components of $(\overline{R}\cap E)_{t_2}.$ Let $(x_1,t_2)$ and $(x_2,t_2)$ be boundary points of $e_1$ and $e_2$ respectively such that the segment
$(x_1,x_2)\times\{t_2\}$ is contained in $\overline{R}$ and is disjoint from both $e_1$ and $e_2$.  By continuity of $f,$ choose $x_3\in[x_1,x_2]$ where $f|_{[x_1,x_2]}$ obtains its minimum and let $t_3=f(x_3)$. Since
\[
x_3 \in [x_1,x_2] \subset p_1(\overline R) = p_1(e(x,t_1)),
\]
it follows that $t_3 = f(x_3) \geq t_1.$
Since
\[
[x_1,x_2]\times{t_2}\cap (X\backslash E)_{t_2}\neq\emptyset,
\]
there is an $x'\in(x_1,x_2)$ such that $f(x')<t_2$, which implies that $t_3<t_2$. For each $x'\in[x_1,x_2]$, $f(x')\geq t_3.$ Therefore,
\[
[x_1,x_2]\times\{t_3\}\subset e(x,t_3)\subset E.
\]
On the other hand, since $f(x_i)=t_2>t_3$ for $i=1,2,$ we have that $x_3\in(x_1,x_2).$ Furthermore, by continuity of $f$ and the fact that $t_3 \geq t_1 > \xi,$ we have $(x_1,t_3),(x_2,t_3)\in E^o$. Thus $(x_3,t_3)$ is a branching point. Since $t_2 > t_3 \geq t_1,$ $(x_3,t_3)$ is contained in $R$ contradicting the assumption.
\item
Let $e_1=e(x_1,t'_1)$ and $e_2=e(x_2,t'_2)$ for $(x_1,t'_1),(x_2,t'_2)\in \overline{R}$. Suppose without loss of generality that $t'_1\leq t'_2$, then by Lemma \ref{lmEsegChar}\ref{i5} it suffices to show that $e(x_1,t'_1)=e(x_2,t'_1)$. Since $e_i \in \overline{R},$ it follows that $e(x,t_1)\leq e(x_i,t'_1)$ for $i=1,2.$ So, $e(x_i,t'_1)$ are connected components of $(R\cap E)_{t'_1}$. The claim now follows immediately from \ref{Bp1}.
\item
We show the less obvious inclusion. Let
\[
(x',t')\in( p_1(e(x,t_1))\times[t_2,\infty))\cap E.
\]
 By \ref{Bp1},
 \[
 e(x',t_2)=e(x,t_2).
 \]
In particular,
 \[
 x'\in p_1(e(x,t_2)).
 \]
\end{enumerate}
\end{proof}
\begin{rem}\label{rmBp1Conv}
It is immediate from the definition of a branching point that the converse to \ref{lmNoBp}\ref{Bp1} is also true. Namely, if there is a $t\in(t_1,t_2)$ for which $R_t$ contains a branching point, then there is a $t'>t$ with $t'\in p_2(R\cap E)$ such that $R_{t'}$ has at least two components. It is clear that $t'$ can be taken arbitrarily close to t.
\end{rem}

\begin{proof}[\textbf{Proof of Lemma~\ref{lm:er}}]
Suppose $(x_1,t_1)\sim(x_2,t_2)$ and $(x_2,t_2)\sim(x_3,t_3)$. We wish to prove that $(x_1,t_1)\sim (x_3,t_3)$. Without loss of generality we assume $t_1\leq t_3$. For $i=1,2,3,$ denote $e_i=e(x_i,t_i)$, $c_i= p_1(e_i)$  and let $R=c_1\times[t_1,t_3)$. We need to prove that $e_1\leq e_3$ and that $R$ contains no branching points.

We distinguish between the three possibilities for the order of $t_1,t_2$ and $t_3$. We start with the case $t_2\leq t_1\leq t_3$. Then $R\subset c_2\times[t_2,t_3)$ and thus $R$ contains no branching points. By Lemma \ref{lmNoBp}\ref{Bp2}, $e_1$ and $e_3$ are comparable, and since $t_1\leq t_3,$ we have $e_1\leq e_3$.

If $t_1\leq t_2\leq t_3,$ then by Lemma \ref{lmNoBp}\ref{Bp3},
\[
R\cap E\subset p_1(e_1)\times[t_1,t_2)\cup p_1(e_2)\times [t_2,t_3),
\]
and thus $R$ contains no branching points. Furthermore,
\[
e_1\leq e_2\leq e_3
\]
by assumption.

If $t_1\leq t_3\leq t_2,$ then $R\subset c_1\times[t_1,t_2).$ So, there are no branching points in $R$. By Lemma~\ref{lmNoBp}\ref{Bp2}, $e_1$ and $e_3$ are comparable. Since $t_1 \leq t_3,$ it follows that $e_1 \leq e_3.$
\end{proof}

\subsection{Tree-like partition}
\begin{df}
A subset $S\subset X$ is \textbf{closed from above} if for any $x\in\gamma$ the intersection $\{x\}\times [\xi,\infty)\cap S$ is closed from the right.
\end{df}
\begin{rem}\label{remClosedAb}
Any finite union, intersection and relatively closed subset of sets that are closed from above is closed from above.
\end{rem}
Let $\mathcal{T}_E$ denote the partition of $E$ into $\sim$ equivalence classes.

\begin{tm}\label{tmMinMaxTree}
Each $c\in \mathcal{T}_E$ satisfies the following properties.
\begin{enumerate}
\item\label{tmMinMaxTree1}
$c$ is $E$-saturated and connected.
\item\label{tmMinMaxTree2}
For any $c_1\neq c\in\mathcal{T}_E$ such that $c\leq c_1,$ there exists a $c_2\in\mathcal{T}_E$ that is incomparable to $c_1$ and such that $c\leq c_2$.
\item\label{tmMinMaxTree3}
Let $c'\subset c$ be $E$-saturated, connected and closed from above. Let $S\subset E$ be disjoint from $c'$. Then
\[
c'\leq S\Rightarrow c'_{T_f(c')}\leq S.
\]
\end{enumerate}
\end{tm}
In Figure \ref{Fig5} the shaded and white parts correspond to different elements in a partition of $E$. The right side of the figure shows what part~\ref{tmMinMaxTree2} of Theorem \ref{tmMinMaxTree} rules out.  The left side shows what is ruled out by part~\ref{tmMinMaxTree3}.
\begin{figure}[h]
\centering
\includegraphics[scale=0.5]{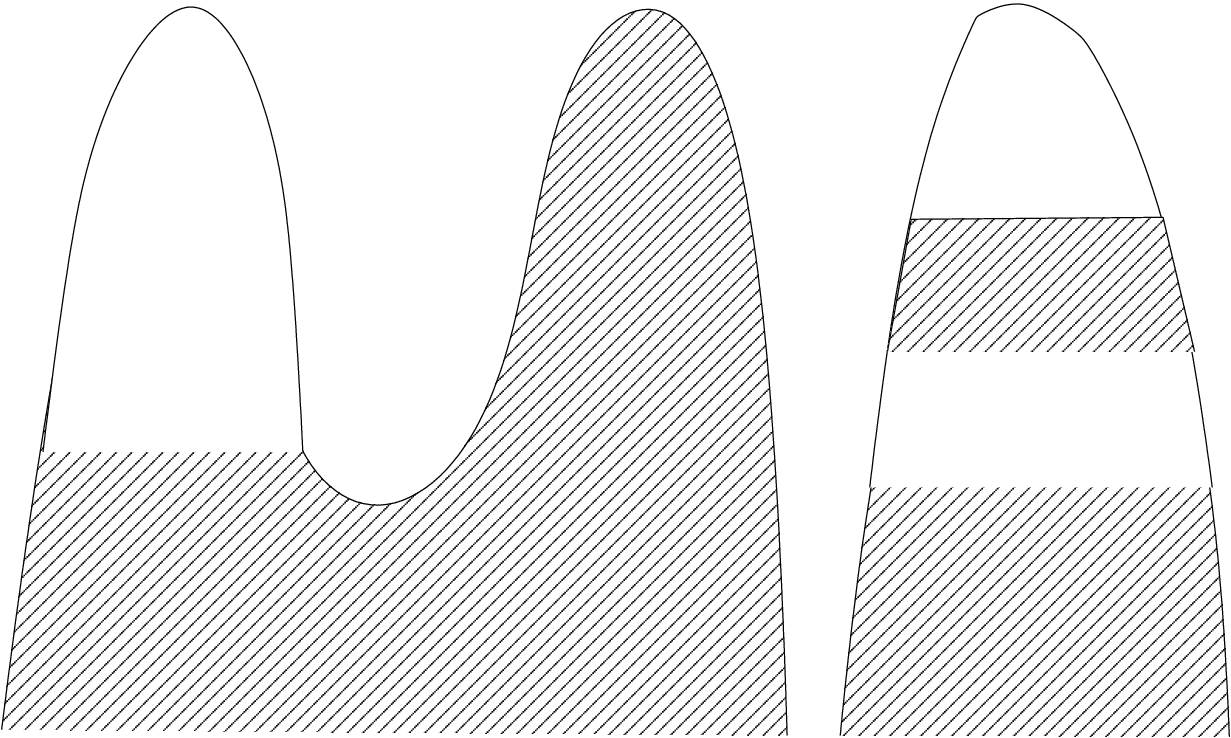}
\caption{}\label{Fig5}
\end{figure}

\begin{cy}\label{cy:tl}
The restriction of $\leq$ to $\mathcal{T}_E$ is a tree-like order relation.
\end{cy}
\begin{proof}
The corollary follows from Theorem~\ref{tmMinMaxTree}\ref{tmMinMaxTree1} and Theorem~\ref{TmTreeEsat}.
\end{proof}

\begin{lm}\label{lmcClosed}
Let $c\subset E$ be an equivalence class under $\sim$. Then $c$ is closed from above.
\end{lm}
\begin{proof}
Let $x\in  p_1(c)$. Choose $t\geq\xi$ such that $(x,t)\in c.$ Let
\[
t_1 = \sup \{s \in [\xi,\infty)|(x,s) \in c \}.
\]
Since $E$ is closed and $c \subset E,$ we have $(x,t_1) \in E.$ It suffices to show that $(x,t_1) \in c.$ Suppose it is not. By Lemma~\ref{lmEsegChar}\ref{i3}, $e(x,t) \leq e(x,t_1).$ So, the rectangle $R= p_1(e(x,t))\times [t,t_1)$ must contain a branching point $(x_1,t_2)$. But then for any $t_3\in(t_2,t_1)$, $(x,t_3)$ is not contained in $c$. This contradicts the definition of $t_1.$
\end{proof}

\begin{lm}\label{lm:tconn}
Let $S\subset E$ be $E$-saturated and contained in a single $\sim$ equivalence class. Then $S_t$ is a single $E$-segment for all $t \in p_2(S).$
\end{lm}
\begin{proof}
Since $S$ is $E$-saturated, $S_t$ is union of $E$-segments. Suppose $e(x_1,t),e(x_2,t) \subset S_t.$ By assumption, $(x_1,t) \sim (x_2,t),$ so by definition of $\sim,$ we have $e(x_1,t) \leq e(x_2,t)$ and $e(x_2,t) \leq e(x_1,t).$ Thus $e(x_1,t) = e(x_2,t)$ by Lemma~\ref{lm:ob}\ref{it:ob2}. Therefore, $S_t$ is a single $E$-segment.
\end{proof}

\begin{lm}\label{lm:cfatf}
Let $S\subset E$ be $E$-saturated, connected, closed from above and contained in a single $\sim$ equivalence class. Then $T_f(S) \in p_2(S).$
\end{lm}
\begin{proof}
By Lemma~\ref{lm:tconn}, $S_t$ is a single $E$-segment for each $t \in p_2(S).$ So, by Remark~\ref{rem:cont} $p_1(S_t)$ is closed. Since $\gamma$ is compact, so is $p_1(S_t).$ By Lemma~\ref{lmEsat}\ref{lmEsat1}, $p_1(S_{t'}) \subset p_1(S_t)$ for all $t' \geq t.$ So $p_1(S_t)$ for $t \in p_2(S)$ is a nested family of compact non-empty sets. Therefore, we may choose
\[
x \in \bigcap_{t \in p_2(S)}  p_1(S_t).
\]
So,
\[
\{x\} \times p_2(S) = (\{x\} \times [\xi,\infty)) \cap S.
\]
Therefore, since $S$ is closed from above, $(x,T_f(S)) \in S,$ and $T_f(S) \in p_2(S).$
\end{proof}

\begin{lm}\label{lmStroncont}
Let $S\subset E$ be $E$-saturated, connected, closed from above and contained in a single $\sim$ equivalence class. Then for any $T\subset E\backslash S$,
\[
S\leq T\Rightarrow S_{T_f(S)}\leq T.
\]
\end{lm}

\begin{proof}
Let $(x,t)\in T$. We show first that $t> T_f(S)$. By assumption there is a $t_1\in p_2(S)$ such that $(x,t_1)\in S$ and a $t_2\in p_2(S)$ such that $t_2\leq \min\{t,t_1\}$. By \ref{lmEsat}\ref{lmEsat1}, $(x,t_2)\in S$. Assume now that $t\leq T_{f}(S)$. Then we have that
\[
R= p_1(e(x,t_2))\times[t_2,t)\subset p_1(S)\times(T_i(S),T_f(S)),
\]
and thus $R$ contains no branching point. By Lemma~\ref{lm:cfatf}, there is an $x'\in\gamma$ such that $(x',{T_f(S)})\in S$. By Lemma~\ref{lmEsat}\ref{lmEsat1}, $e(x',t)\subset S$. Thus by Lemma~\ref{lm:tconn}, $e(x',t)=e(x,t).$ In particular, $(x,t)\in S$ contradicting the assumption that $T\subset E\backslash S$. Thus $t> T_f(S)$, so $e(x,{T_f(S)})$ is well defined. By Lemma~\ref{lm:tconn}, $e(x,{T_f(S)})=S_{T_f(S)},$ so by Lemma \ref{lmEsegChar}\ref{i5}, $S_{T_f(S)}\leq e(x,t) \leq (x,t)$. Since $(x,t) \in T$ was arbitrary, the lemma follows.
\end{proof}

\begin{lm}\label{lm:seg}
Let $c \in \mathcal{T}_E,$ and let $t_1, t_2 \in p_2(c)$ satisfy $t_1 \leq t_2.$ Let $x_2 \in p_1(c_{t_2}).$ Then $\{x_2\} \times [t_1,t_2] \subset c.$
\end{lm}
\begin{proof}
Let $x_1 \in p_1(c_{t_1}).$ Then $(x_1,t_1) \sim (x_2,t_2).$ By definition of $\sim$, $e(x_1,t_1)\leq e(x_2,t_2)$. Let $t\in [t_1,t_2].$ Then by Lemmas \ref{lmEsegChar}\ref{i3} and \ref{lmEsegChar}\ref{i6},
\[
e(x_1,t_1)\leq e(x_2,t)\leq e(x_2,t_2).
\]
Let $R= p_1(e(x_2,t))\times[t,t_2).$ Then $R\subset p_1(e(x_1,t_1))\times[t_1,t_2)$. Therefore $R$ contains no branching points. Thus $(x_2,t)\sim (x_2,t_2)$. The lemma follows.
\end{proof}

\begin{lm}\label{lm:path}
Let $c \in \mathcal{T}_E,$ and $(x_i,t_i) \in c$ for $i = 1,2.$ Suppose $t_1 \leq t_2.$ There exists a path $\rho \subset c$ connecting $(x_1,t_1)$ and $(x_2,t_2)$ such that $p_2(\rho) \subset [t_1,t_2].$
\end{lm}

\begin{proof}
Let $\rho' = \{x_2\}\times[t_1,t_2].$
By Lemma~\ref{lm:seg}, $\rho'\subset c$. By definition of $\sim,$ $e(x_1,t_1) \leq e(x_2,t_2).$ Thus $\rho'$ connects the path connected sets $e(x_1,t_1)$ and $e(x_2,t_2).$ So, choose
\[
\rho \subset e(x_1,t_1) \cup \rho' \cup e(x_2,t_2)
\]
connecting $(x_1,t_1)$ to $(x_2,t_2).$
By definition of $\sim,$ $c$ is $E$-saturated, so $e(x_i,t_i) \subset c.$ Thus $\rho \subset c.$ Since
\[
p_2(e(x_1,t_1) \cup \rho' \cup e(x_2,t_2)) \subset [t_1,t_2],
\]
we have $p_2(\rho) \subset [t_1,t_2].$
\end{proof}

\begin{proof}[\textbf{Proof of Theorem \ref{tmMinMaxTree}}]
Any $c\in\mathcal{T}_E$ is $E$-saturated by definition.
Such $c$ is connected (in fact, path connected) by Lemma~\ref{lm:path}. Thus $\mathcal{T}_E$ satisfies \ref{tmMinMaxTree1}. That it satisfies \ref{tmMinMaxTree2} follows from the definition of $\sim$ and Remark \ref{rmBp1Conv}. Condition \ref{tmMinMaxTree3} follows from Lemma \ref{lmStroncont}.
\end{proof}

\begin{cy}\label{lmTE}
Let $s\in\mathcal{T}_E$.
\begin{enumerate}
\item\label{lmTE1}
$s$ is closed from above.
\item\label{lmTE2}
For any $t\in p_2(s)$, $s_t$ is connected.
\end{enumerate}
\end{cy}
\begin{proof}
By definition, $s$ is a $\sim$-equivalence class. We rely on this in the following.
\begin{enumerate}
\item
The claim is Lemma \ref{lmcClosed}.
\item
The claim is Lemma~\ref{lm:tconn}.
\end{enumerate}
\end{proof}

\begin{cy}\label{cy:pro}
Let $s \in \mathcal T_E$ and let $s' \subset s$ be $E$-saturated. Then for any $a \subset p_2(s'),$ we have
\[
p_2^{-1}(a)\cap s = p_2^{-1}(a) \cap s'.
\]
\end{cy}
\begin{proof}
Let $t \in a.$ Since $s'$ is $E$-saturated, $s'_t$ is an $E$-segment. So, Corollary \ref{lmTE}\ref{lmTE2} implies $s'_t=s_t$.
\end{proof}

\begin{cy}\label{lmTECon}
Let $s\in\mathcal{T}_E$ and let $s'\subset s$ be $E$-saturated. Suppose $p_2(s')$ is connected. Then $s'$ is connected.
\end{cy}
\begin{proof}
Let $(x_i,t_i)\in s'$ for $i=1,2$. Suppose $t_1\leq t_2$. Since $p_2(s')$ is connected, for any $t\in[t_1,t_2],$ $s'_t\neq\emptyset$.  By Corollary~\ref{cy:pro} with $a = [t_1,t_2],$ we have
\[
p_2^{-1}([t_1,t_2])\cap s = p_2^{-1}([t_1,t_2]) \cap s'.
\]
So, by Lemma~\ref{lm:path} we can connect $(x_1,t_1)$ to $(x_2,t_2)$ by a path in $s'.$ Since $(x_i,t_i)\in s'$ were arbitrary, $s'$ is (path) connected.
\end{proof}

\begin{lm}\label{lm:lino}
Let $c \in \mathcal{T}_E.$ Let $\mathcal P$ be a collection of disjoint connected $E$-saturated subsets of $c.$ The relation $\leq$ induces a linear order on $\mathcal P.$
\end{lm}
\begin{proof}
First, we prove that for $P \subset \mathcal P$ we have
\begin{equation}\label{eq:PcTf}
P \leq c_{T_f(c)}.
\end{equation}
Indeed, for $t\in p_2(P)$ since $P$ is $E$-saturated, Lemma~\ref{lmTE}\ref{lmTE2} implies that $P_t = c_t.$ By Lemma~\ref{lmEsat}\ref{lmEsat1}, we have $c_t \leq c_{T_f(c)}.$ So,
\[
P \leq P_t = c_t \leq c_{T_f(c)}
\]
as desired.

Suppose $P_1,P_2 \in \mathcal P.$ Without loss of generality, we may assume $P_1 \leq_2 P_2.$ So, by relation~\eqref{eq:PcTf} and Lemma~\ref{lmNoCycles}, we have $P_1\leq P_2,$ which implies the lemma.
\end{proof}

\begin{lm}\label{lm:max}
For each $s\in\mathcal{T}_E$ there exists a maximal element $m \in \mathcal{T}_E$ such that $s \leq m.$
\end{lm}
\begin{proof}
Let $x\in\overline{p_1(s)}$ be the point where $f$ obtains its maximum. Then $f(x)\geq T_f(s)$. If $f(x)=T_f(s)$ then $s$ itself is maximal. If $f(x)>T_f(s)$ we claim that $s\leq \{(x,f(x))\}$. Assume by contradiction otherwise. Then $x\not\in p_1(s_{T_f(s)})$. By Lemma~\ref{lmTE}\ref{lmTE2} $s_{T_f(s)}$ is a single $E$-segment. Therefore $p_1(s_{T_f}(s))$ is closed. It follows that $x$ has an open neighborhood $v\subset\gamma\backslash p_1(s_{T_f(s)})$. Since  $x\in\overline{p_1(s)}$ and $f$ is continuous, there is a point $x'\in p_1(s)\cap v$ close enough to $x$ so that $f(x')>T_f(s)$. Therefore $s\leq \{(x',f(x'))\}$. On the other hand, since $f(x')>T_f(s)$, $(x',f(x'))\not\in s$. But since $x'\not\in p_1(s_{T_f(s)})$, $s_{T_f(s)}\not\leq \{(x',f(x'))\}$ in contradiction to \ref{tmMinMaxTree}\ref{tmMinMaxTree3}.

Let $m\in \mathcal{T}_E$ be the element containing $(x,f(x))$. Then we have that $s\leq \{(x,f(x))\}$ and $m\leq\{(x,f(x))\}$. Therefore, by Lemma \ref{lmNoCycles} $s\leq m$.
\end{proof}

Given a finite collection $V$ of connected $E$-saturated sets, we define a graph $F_V$ as follow. $V$ is the set of vertices of $F$. We connect the vertex $v_1$ to $v_2$ if $v_1\leq v_2$ and there is no $v_3\in V$ such that $v_1\leq v_3\leq v_2$. By virtue of Corollary~\ref{cy:tl}, $F_V$ has no cycles and so is a forest. Again by Corollary~\ref{cy:tl}, each tree $T$ in $F_V$ has a unique minimal vertex $r_T$, which we designate as the root of $T.$ Thus the leaves of $F_V$ are the maximal vertices. Denote by $R(V)\subset V$ the roots of $F_V,$ by $L(V)\subset V$ the leaves, and by $I(V)\subset V$ the vertices which are neither roots nor leaves. Denote by $E(V)$ the edges of $F_V.$

A finite forest $F$ is called \textbf{stable} if any $v\in F$ which is not a leaf has at least two direct descendants. The proof of the following lemma is standard and we omit it.
\begin{lm}\label{lmStabFor}
Let $F$ be a finite stable forest and let $L$ be the number of its leaves. Then $|F|\leq 2L$.
\end{lm}

\begin{lm}\label{lmMaxfin}
Let $M\subset \mathcal{T}_E$ be the set of maximal elements under $\leq$. $\mathcal{T}_E$ is finite if and only if $M$ is finite. Moreover, in that case $|\mathcal{T}_E| \leq 2|M|.$
\end{lm}
\begin{proof}
First, we let $S$ be an anti-chain in $\mathcal{T}_E$ and show that $|S|\leq  |M|$. By Lemma~\ref{lm:max}, for each $s \in S$ there is at least one $m\in M$ such that $s\leq m$. Since the elements of $S$ are pairwise incomparable, and since by Corollary~\ref{cy:tl} the order on $\mathcal{T}_E$ is tree-like, for any element $m\in M$ there is at most one $s\in S$ such that $s\leq m$. Thus $|S|\leq | M|$.

It suffices to prove the bound for any finite subset $V \subset \mathcal{T}_E.$  Given such $V,$ by Theorem \ref{tmMinMaxTree}\ref{tmMinMaxTree3} we may choose $V' \subset \mathcal{T}_E,$ such that $V \subset V'$ and $F_{V'}$ is stable. The claim now follows from Lemma \ref{lmStabFor}.
\end{proof}

\section{Thick thin partition}\label{SecThickThinPart}

\subsection{Thickened hypograph}
We now specialize the discussion of the previous section to the case where $\gamma$ is a geodesic in $\Sigma$ as in Theorem \ref{TmApLenBound}. Denote the connected components of $\gamma$ by $\gamma_i.$ We will assume throughout this section that for all $i$ we have
\begin{equation}\label{eq:long}
4e^{-\max_{x \in \gamma_i}g(x)} \leq \ell_n(\gamma_i),
\end{equation}
and
\begin{equation}\label{eq:hot}
2k \leq \max_{x \in \gamma_i} e^{g(x)}.
\end{equation}

Recall Definition \ref{dfDB}. Here and below, we abbreviate $\ell_n(c)=\ell_n(p_1(c))$.
Let $s_t:=\{c\in\pi_0(X_t\backslash D)|\ell_n(c)> 4e^{-t}\},$ and let $S_t$ be the union of elements of $s_t$. Let $E_t:=X_t\backslash S_t$, and $E:=\cup_tE_t\cup\gamma\times\{\ln 2k\}$. We will show that $E$ is the hypograph of a continuous function. Figure \ref{Fig4} gives a picture of a typical $E$ compared with $D$.
\begin{figure}[h]
\includegraphics[scale=0.6]{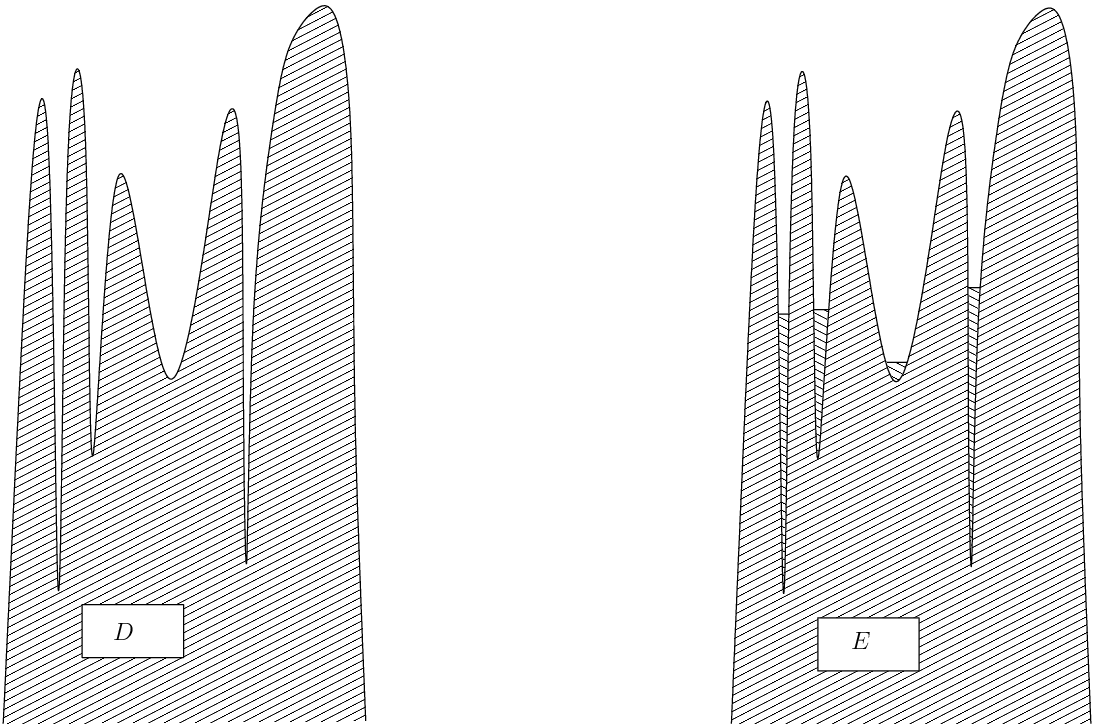}
\caption{}\label{Fig4}
\end{figure}

\begin{rem}\label{remEnPart}
By definition $D_t$ is $2e^{-t}$-dense in $E_t$. Let $t_1,t_2\geq\ln(2k).$ Let $S_1\subset E_{t_1},S_2\subset E_{t_2},$ be segments such that $ p_1(S_1)\cap p_1(S_2)=\emptyset$ and such that $\ell_n(S_i)\geq 8e^{-t_i}$ for $i=1,2$. Then $S_i$ contains a point $x_i\in D_t$ such that $d(x_i,\partial S_i;h_n)\geq 2e^{-t_i}$. Let $B_i=B(x_i,t_i)$. By Lemma \ref{lmDDisc}, $\mu(B_i)\geq \delta_1$. Furthermore, by Lemma \ref{lmDisjDisc}
\[
B_1\cap B_2=\emptyset.
\]
This observation will play a key role in the following.
\end{rem}

\begin{lm}\label{lmEclosed}
$E$ is closed.
\end{lm}
\begin{proof}
Let $(x,t)\in\partial E$. We show that $(x,t)\in E$. If $(x,t)\in D,$ we are done since $D\subset E$. Otherwise, let $s$ be the connected component of $X_t\backslash D$ containing $(x,t)$. We show that $\ell_n(s)\leq 4e^{-t},$ which implies that $s\subset E$. Assume by contradiction otherwise. Then there is a compact segment $s'\subset s$ such that  $\ell_n(s')>4e^{-t}$ and such that $(x,t)\in s'^o$. By the closedness of $D$, compactness of $s'$,  and continuity of the exponent, there is an $\epsilon>0$ such that $\pi_1(s')\times(t-\epsilon,t+\epsilon)\subset X\backslash D$ and such that for any $t'\in(t-\epsilon,t+\epsilon)$, $\ell_n(s')>4e^{-t'}$. This implies that $\pi_1(s')\times(t-\epsilon,t+\epsilon)\subset X\backslash E$. But $\pi_1(s')\times(t-\epsilon,t+\epsilon)$ contains an open neighborhood of $(x,t)$ in contradiction to the fact that $(x,t)\in\partial E$.
\end{proof}

\begin{lm}\label{LmEBalanced}
Let $t\in[\ln 2k,\infty)$. For all $x\in p_1(E_t)$,  $\{x\}\times[\ln 2k,t)\subset E^o$.
\end{lm}
\begin{proof}
First, we prove that for all $y \in p_1(E_t),$
\begin{equation}\label{eq:st1}
\{y\} \times [\ln 2k,t) \subset E.
\end{equation}
Let $t'\in[\ln 2k,t)$. We show that $(y,t')\in E$. Indeed, if $(y,t')\in D,$ we are done since $D\subset E$. If $(y,t')\not\in D,$ then $(y,t)\not\in D$. Let $S$ and $S'$ be the components of $X_{t}\backslash D$ and $X_{t'}\backslash D$ containing $(y,t)$ and $(y,t')$ respectively. Clearly $p_1(S')\subset p_1(S)$. Since $(y,t)\in E,$ we have by definition of $E$ that
\[
\ell_n(S')\leq\ell_n(S)\leq 4e^{-t}< 4e^{-t'}.
\]
By the same definition, we conclude that $(y,t')\in E$.

To prove the claim, it suffices to show for any $t' \in [\ln 2k,t)$ that $(x,t')\in E^o$. If $t'<g(x),$ then $(x,t')\in D^o\subset E^o$ by continuity of $g$. If $t'\geq g(x),$ then $t>g(x)$. Let $S$ be the component of $X_t\setminus D$ containing $(x,t).$ By definition of $E,$ we have $S \subset E_t.$ So, invoking inclusion~\eqref{eq:st1} for each $y \in \pi_1(S),$ we conclude that the open neighborhood $\pi_1(S)\times[\ln 2k,t)$ of $(x,t')$ is contained in $E.$
\end{proof}

It is clear from the definition that for any $x\in\gamma$, the set
\[
\{t\in[\ln 2k,\infty)|(x,t)\in E\}
\]
is bounded from above. It follows from Lemma \ref{LmEBalanced} that $E$ is the hypograph of the function $f_{\partial E}$ defined by
\[
f_{\partial E}(x)=\sup\{t\in[\ln 2k,\infty)|(x,t)\in E\}.
\]
\begin{lm}
$\partial E$ is the graph of $f_{\partial E}$.
\end{lm}
\begin{proof}
Clearly, the graph of $f$ is contained in $\partial E.$ So, it suffices to prove the opposite inclusion.
Let $(x,t)\in\partial E.$ Then by Lemma~\ref{lmEclosed} $(x,t) \in E,$ so by Lemma \ref{LmEBalanced} $\{x\}\times[\ln 2k,t)\subset E^o$.
On the other hand, we show that $\{x\}\times(t,\infty)\subset X\backslash E$. Indeed, if $t' \in (t,\infty)$ and $(x,t') \in E,$ then by Lemma~\ref{LmEBalanced}
\[
(x,t) \in \{x\}\times[\ln 2k,t') \subset E^o
\]
contradicting the choice of $(x,t).$
Thus, by definition $f_{\partial E}(x)=t$.
\end{proof}

\begin{cy}
$f_{\partial E}$ is continuous.
\end{cy}
\begin{proof}
The graph of $f_{\partial E}$ is closed being the boundary of $E$. $f_{\partial E}$ is bounded, so its graph is compact. Thus, the projection $p_1$ restricted to the graph of $f_{\partial E}$ is closed. It follows that $f_{\partial E}$ is continuous.
\end{proof}

\begin{rem}\label{rm:long}
It follows from equations~\eqref{eq:long} and~\eqref{eq:hot} that
\[
\max_{x \in \gamma} f_{\partial E}(x) = \max_{x \in \gamma} g(x).
\]
\end{rem}

\begin{lm}\label{lmEsegChar4}
Let $e_1$ and $e_2$ be $E$-segments that are incomparable with respect to $\leq$. Let $t_i = p_2(e_i).$ Then,
\[
d( p_1(e_1), p_1(e_2);h_n)\geq 4e^{-\min(t_1,t_2)}.
\]
\end{lm}
\begin{proof}
Assume without loss of generality that $t_1\leq t_2.$ Let $x \in p_1(e_2)$ and denote $e'_2=e(x,t_1)$. Since $e'_2$ and $e_1$ are both connected components of $E_{t_1},$ we have by the definition of $E_{t_1}$ that either
$e'_2 = e_1$ or $d( p_1(e'_2), p_1(e_1))\geq 4e^{-t_1}$. By Lemma~\ref{lmEsegChar}\ref{i3}, $e'_2\leq e_2$. Thus $ p_1(e_2)\subset p_1(e'_2)$ giving the claim.
\end{proof}
\begin{lm}\label{lmEitherOrSat2}
Let $S_1,S_2\subset E$. Suppose $S_1$ is $E$-saturated and connected, $S_2$ is connected and $S_1\leq_2 S_2$. Then either $S_1\leq S_2$ or
\[
d( p_1(S_1), p_1(S_2);h_n)\geq 4e^{-T_i(S_1)}.
\]
\end{lm}

\begin{proof}
Suppose $S_1\not\leq S_2$. Let $(x_1,t_1)\in S_1$ and $(x_2,t_2)\in S_2$. Let $e_i=e(x_i,t_i)$ for $i=1,2$. We show that $d(x_1,x_2)\geq 4e^{-T_i(S_1)}$. Suppose $t \in p_2(S_1)$ is such that $t \leq t_1.$ Since $S_1$ is $E$-saturated and connected, Lemma \ref{lmEsat}\ref{lmEsat1} implies that $(x_1,t)\in S_1$. Since $S_1\leq_2 S_2$, there is a a $t\in p_2(S_1)$ such that $t\leq t_2$. We may thus assume without loss of generality that $t_1\leq t_2$. Furthermore, $t_1$ may be assumed to be arbitrarily close to $T_i(S_1)$. Since $S_1$ is $E$-saturated, $e_1\subset S_1$. By Lemma \ref{lmEitherOrSat} we have $p_1(S_1)\cap p_1(S_2)=\emptyset$. In particular $x_2\not\in p_1(e_1),$ so $e_1\not\leq e_2$. Since $t_1\leq t_2$, $e_1$ and $e_2$ are incomparable. Therefore, by Lemma \ref{lmEsegChar4} we have
\[
d(x_1,x_2)\geq d(p_1(e_1),p_1(e_2))\geq 4e^{-t_1}.
\]
Since $t_1$ is arbitrarily close to $T_i(S_1),$ we have
\[
d(x_1,x_2)\geq 4e^{-T_i(S_1)}.
\]
Since $x_i$ were arbitrary points in $p_1(S_i),$ the claim follows.
\end{proof}
\begin{cy}\label{cyEitherOrSat2}
Let $S_1,S_2\subset E$ be incomparable connected $E$-saturated sets. Then
\[
d( p_1(S_1), p_1(S_2);h_n)\geq 4e^{-\min\{T_i(S_1),T_i(S_2)\}}.
\]
\end{cy}

\begin{proof}
Without loss of generality $S_1\leq_2 S_2$ . The claim thus follows from Lemma \ref{lmEitherOrSat2}.
\end{proof}
\begin{cy}\label{lmEquivFin}
$\mathcal{T}_E$ is finite.
\end{cy}
\begin{proof} Let $M$ be the set of maximal elements of $\mathcal{T}_E$. Then the elements of $M$ are pairwise incomparable under $\leq$. Let $T=\sup_{x\in\gamma}g(x)$. By Remark~\ref{rm:long} we have $T_i(c)\leq T$ for any $c\in M$.  It follows from Corollary~\ref{cyEitherOrSat2} that
\[
|M|\leq \frac1{4} e^T\ell_n(\gamma).
\]
In particular, $M$ is finite. The claim therefore follows from Lemma \ref{lmMaxfin}.
\end{proof}

\subsection{Thick thin partition}
We now wish to partition $E$ into a thin part and a thick part. To this end, let
\[
N:=\{(x,t)\in E|\ell_n(e(x,t))\leq  24 e^{-t}\},
\]
and let
\[
K:=E\backslash N.
\]

\begin{lm}\label{lmNclosed}
$N$ is $E$-saturated and closed from above.
\end{lm}
\begin{proof}
It is obvious that $N$ is $E$-saturated. We show that $N$ is closed from above. Let $(x,t)$ be a right boundary point of $\{x\}\times [\ln 2k,\infty)\cap N$. Since $E$ is closed, we have $(x,t)\in E$. Thus we need to show that $\ell_n(e(x,t))\leq 24e^{-t}$. Assume by contradiction otherwise. Let $\epsilon>0$ be so small that $\ell_n(e(x,t)) > 24e^{\epsilon-t}$. By Lemma \ref{LmEBalanced}, we have
\[
\{x\}\times[t-\epsilon,t)\subset E.
\]
On the other hand, for any $t'\in [t-\epsilon,t)$, by Lemma~\ref{lmEsat}\ref{i3}
\[
p_1(e(x,t)) \subset p_1(e(x,t')).
\]
So,
\[
\ell_n(e(x,t'))\geq \ell_n(e(x,t)) >  24e^{-t + \epsilon} \geq 24 e^{-t'}.
\]
Therefore, $\{x\}\times[t-\epsilon,t)\subset E\backslash N$ giving a contradiction.
\end{proof}

\begin{df}
A \textbf{thin neck} is a connected component of $c\cap N$ where
$c\in\mathcal{T}_E$. Given a thin neck $L,$ we denote by $c_L$ the unique $c \in \mathcal{T}_E$ such that $L \subset c.$
\end{df}

\begin{lm}\label{lmThinNeck}
Let $L$ be a thin neck.
\begin{enumerate}
\item\label{Tn1}
$L$ is $E$-saturated and closed from above.
\item\label{Tn2}
Let $S\subset E$ be disjoint from $L$, then
\[
L\leq S\Rightarrow L_{T_f(L)}\leq S.
\]
\item\label{Tn3}
For any $x\in p_1(L)\backslash p_1(L_{T_f(L)})$,
 \[
 (x,f_{\partial E}(x))\in L.
 \]
\item\label{Tn4}
For any $x\in p_1(L)\backslash p_1(L_{T_f(L)})$,
\[
g(x)\leq \ln24-\ln d(x, p_1(L_{T_f(L)});h_n).
\]
\end{enumerate}
\end{lm}
\begin{proof}
\begin{enumerate}
\item
Let $c\in\mathcal{T}_E$ such that $L$ is a connected component of $c\cap N$. $c$ is $E$ saturated by definition and closed from above by Corollary~\ref{lmTE}\ref{lmTE1}. $N$ is $E$-saturated and closed from above by Lemma~\ref{lmNclosed}. It follows from Remark \ref{rmEsat} that $L$ is $E$-saturated. By Remark \ref{remClosedAb}, $L$ is closed from above.
\item
$L$ is connected and contained in an equivalence class $c \in \mathcal{T}_E$ by definition. $L$ is $E$-saturated and closed from above by \ref{Tn1}. The claim thus follows from Theorem \ref{tmMinMaxTree}\ref{tmMinMaxTree3}.
\item
Let $t\geq\ln 2k$ be such that $(x,t)\in L$. Since $L\subset E$, we have $t\leq f_{\partial E}(x)$. This means that $L\leq \{(x,f_{\partial E}(x))\}$. Suppose
\[
(x,f_{\partial E}(x))\not\in L.
\]
Then by \ref{Tn2}, $L_{T_f(L)}\leq \{(x,f_{\partial E}(x))\}$. This produces the contradiction $x\in p_1(L_{T_f(L)})$.
\item
Let $x'\in  p_1(L_{T_f(L)})$.  By~\ref{Tn3}, $(x,f_{\partial E}(x))\in L$, so $f_{\partial E}(x)\in p_2(L).$ Thus by~\ref{lmEsat}\ref{lmEsat1} we have $x'\in p_1(L_{f_{\partial E}(x)})$. Let $c \in \mathcal{T}_E$ be such that $L \subset c.$ By~\ref{Tn1}, $L$ is $E$-saturated, so by Corollary \ref{lmTE} we have $L_{f_{\partial E}(x)}=c_{f_{\partial E}(x)}$ and $L_{f_{\partial E}(x)}$ is connected. Thus $e(x,f_{\partial E}(x))=e(x',f_{\partial E}(x))$. But
\[
\ell_n(e(x',f_{\partial E}(x)))\leq 24 e^{-f_{\partial E}(x)},
\]
so
\begin{align}\label{EqExpTn4}
d(x, p_1(L_{T_f(L)});h_n)&\leq d(x,x';h_n)\\
&\leq \ell_n(e(x',f_{\partial E}(x)))\notag\\
&\leq 24 e^{-f_{\partial E}(x)}\notag\\
&\leq 24 e^{-g(x)}\notag.
\end{align}
The claim follows by taking logarithms in \eqref{EqExpTn4}.
\end{enumerate}
\end{proof}

\begin{lm}\label{lmTNCI}
For $i=1,2,$ let $L_i$ be disjoint thin necks such that $p_2(L_1)\cup p_2(L_2)$ is connected.
\begin{enumerate}
\item\label{lmTNCI1}
$c_{L_1}\neq c_{L_2}$.
\item\label{lmTNCI2}
If $L_1\leq L_2,$ then $L_{1,T_f(L_1)}$ contains a branching point.
\end{enumerate}
\end{lm}
\begin{proof}
\begin{enumerate}
\item
Suppose the contrary. Then Corollary \ref{lmTECon} implies that $L_1\cup L_2$ is connected. The definition of thin necks implies the contradiction $L_1=L_2$.
\item
Let $(x,t)\in L_2$. For any $t'\in [T_f(L_1),t),$ denote
\[
R(t')=p_1(e(x,t'))\times [t',t).
\]
Since $L_1\leq L_2$, by Lemma~\ref{lmThinNeck}\ref{Tn2}, $(x,T_f(L_1))\in L_1$. By \ref{lmTNCI1}, $c_{L_1}\neq c_{L_2}$. Therefore, $(x,t)\not\sim(x,T_f(L_1))$. Thus $R(T_f(L_1))$ contains a branching point $p$. But for any $t'\in(T_f(L_1),t)$, by Lemma~\ref{lm:seg}, $\{x\}\times[t',t]\subset c_{L_2}$. In particular, $(x,t)\sim(x,t')$. So, $R(t')$ contains no branching point. Thus $p\in e(x,T_f(L_1)).$ Since $(x,T_f(L_1)) \in L_1,$ by Lemma~\ref{lmThinNeck}\ref{Tn1}, $e(x,T_f(L_1))\subset L_{1,T_f(L_1)}.$ Therefore, $p \in L_{1,T_f(L_1)}.$

\end{enumerate}
\end{proof}

\begin{df}\label{dfThinBranch}
Let $L$ be a thin neck. $L$ is \textbf{exceptional} if it satisfies the following two conditions.
\begin{enumerate}
\item
\[
T(L)<\ln 3.
\]
\item\label{dTB2}
For any $\epsilon>0$ there are $(x_i,t_i)\in K$ for $i=1,2,$ such that $t_1\in(T_i(L)-\epsilon,T_i(L)]$, $t_2\in(T_f(L),T_f(L)+\epsilon)$, and
\[
e(x_1,t_1)\leq L\leq e(x_2,t_2).
\]
\end{enumerate}
\end{df}

\begin{df}
Let $L$ be a thin neck. We write
\[
\mathcal{N^+}(L,\epsilon):=\left(p_1(L)\times [T_i(L),T_f(L)+\epsilon)\right)\cap E,
\]
and
\[
\mathcal{N^-}(L,\epsilon):=\left(p_1(L)\times[T_i(L)-\epsilon,T_f(L))\right)\cap E.
\]
We say that $L$ is \textbf{upper interior} if there is an $\epsilon>0$ such that $\mathcal{N^+}(L,\epsilon)\subset N$. Similarly, we say that $L$ is \textbf{lower interior} if there is an $\epsilon >0$ such that $\mathcal{N^-}(L,\epsilon)\subset N.$
\end{df}

\begin{rem}\label{remULint}
By Definition \ref{dfThinBranch}\ref{dTB2}, every non-exceptional thin neck $L$ such that $T(L)<\ln 3$ is either upper interior or lower interior.
\end{rem}

\begin{rem}\label{remNLepsCon}
It is clear that for any $\epsilon>0$, $\mathcal{N^+}(L,\epsilon)$ and $\mathcal{N^-}(L,\epsilon)$ are connected.
\end{rem}

\begin{df}
The \textbf{thin part} $A \subset E$ is the union of non-exceptional thin necks. The \textbf{thick part} of $E$ is
\[
C:=E\backslash A.
\]
\end{df}

\begin{lm}\label{lm:CEsat}
$C$ is $E$-saturated.
\end{lm}
\begin{proof}
This follows from Lemma~\ref{lmThinNeck}\ref{Tn1} and Remark~\ref{rmEsat}.
\end{proof}

\subsection{Energy bound on the number of thin necks}
Let $H$ denote the set of non-exceptional thin necks $L$ such that $T(L)<\ln3$. Let $G$ denote the set of all non-exceptional thin necks.
\begin{lm}\label{lmNumETNbound}
\[
|H|\leq 2|\mathcal{T}_E|.
\]
\end{lm}
\begin{proof}
We will show that the map $H\rightarrow\mathcal{T}_E$ defined by $L\mapsto c_L$ is at most two to one. Let $c\in\mathcal{T}_E$. By Remark \ref{remULint} every $L\in H$ is either upper interior or lower interior. We claim that that there is at most one lower interior $L\in H\cap\pi_0(N\cap c)$ . Similarly we claim that there is at most one upper interior $L\in H\cap\pi_0(N\cap c)$.

Indeed, suppose $L\subset c$ is a lower interior element of $H$. First, we show that
\begin{equation}\label{eq:Lc}
T_i(L) = T_i(c).
\end{equation}
Since $L \subset c,$ we have $T_i(c) \leq T_i(L).$ Suppose by contradiction the inequality is strict. Let $\epsilon>0$ be such that $\mathcal{N^-}(L,\epsilon)\subset N$. By Theorem~\ref{tmMinMaxTree}\ref{tmMinMaxTree1}, $c$ is connected, so $p_2(c)$ is an interval. Thus we may choose $t_1 \in p_2(c) \cap [T_i(L)-\epsilon,T_i(L)).$ Choose $(x_2,t_2) \in L.$ By Lemma~\ref{lm:seg}, $\{x_2\}\times[t_1,t_2] \subset c.$ By definition of $\mathcal{N^-}(L,\epsilon),$ we have $\{x_2\} \times [t_1,t_2] \subset \mathcal{N^-}(L,\epsilon) \subset N.$ It follows that $(x_2,t_1)$ belongs to the same connected component of $c \cap N$ as $(x_2,t_2).$ However $t_1 < T_i(L),$ so $(x_1,t_1) \notin L$ contradicting the definition of $L.$ Equation~\eqref{eq:Lc} follows.

Let $L'\in H\cap\pi_0(N\cap c)$ also be lower interior. Then $T_i(L')=T_i(c)=T_i(L)$. In particular, $p_2(L)\cup p_2(L')$ is connected. Therefore, by Lemma \ref{lmTNCI}\ref{lmTNCI1}, $L=L'$.

The upper interior case follows similarly. Thus the map $L\mapsto c_L$ is at most two to one.
\end{proof}

\begin{cy}\label{cyNETNfin}
$|G|<\infty.$
\end{cy}
\begin{proof}
In light of Lemmas \ref{lmNumETNbound} and \ref{lmEquivFin} we need only bound $G\setminus H$. Again relying on Lemma~\ref{lmEquivFin}, it suffices to bound the set $S_c$ defined for any $c\in\mathcal{T}_E$ by
\[
S_c=(G\setminus H)\cap\pi_0(N\cap c).
\]
By Remark~\ref{rm:long} we have
\[
T(c)\leq\sup_{x\in\gamma}g(x)-\ln2k=:M.
\]
Lemma \ref{lmTNCI}\ref{lmTNCI1} implies that $p_2(L_1)\cap p_2(L_2)=\emptyset$ for any two elements $L_1,L_2\in S_c$. Therefore $|S_c|\leq \frac{T(c)}{\ln3}\leq\frac{M}{\ln3}$.
\end{proof}

\begin{lm}\label{lm:sand}
Let $c \in \mathcal{T}_E$ and let $L_1,L_2 \subset c$ be non-exceptional thin necks. Let $k$ be a connected component of $c \cap C$ satisfying $L_1 \leq k \leq L_2.$ Then $k$ is a connected component of $C.$
\end{lm}
\begin{proof}
By Remark~\ref{rmEsat} and Lemma~\ref{lm:CEsat}, $k$ is $E$-saturated. Choose $(x,t_2) \in L_2.$ Since $L_1 \leq k \leq L_2,$ there exist $t_1 \in p_2(L_1)$ and $t \in p_2(k)$ such that $(x,t_1) \in L_1$ and $(x,t) \in k.$ Since $L_1,L_2 \subset A$ and $k \subset C,$ we have
\[
L_i \cap k = \emptyset, \qquad i = 1,2.
\]
So, by Lemma~\ref{lm:sco} and Corollary~\ref{cyAntiSym}, we deduce that $t_1 < t < t_2.$

Let $k'$ be the connected component of $C$ containing $k.$
We prove that $k' \subset c,$ which immediately implies the lemma. Indeed, let
\[
R = p_1(e(x,t_1)) \times [t_1,t_2).
\]
Since $L_1,L_2 \subset c,$ we have $(x,t_1) \sim (x,t_2).$ So, $e(x,t_1) \leq e(x,t_2)$ and $R$ contains no branch points. Since $e(x,t_1) \leq e(x,t_2),$ we have
\[
e(x,t_2) \subset (\overline{R} \cap E)_{t_2}.
\]
Since $R$ contains no branched points, by Lemma~\ref{lmNoBp} we deduce that $(\overline{R} \cap E)_{t_2}$ is connected. So,
\[
e(x,t_2) = (\overline{R} \cap E)_{t_2}.
\]
Since $L_i \subset A$ does not intersect $C,$
\[
e(x,t_i) \subset L_i
\]
does not intersect $C$ for $i = 1,2.$ Let $x_1$ and $x_2$ be the endpoints of the interval $p_1(e(x,t_1)).$ By Remark~\ref{rem:cont},
\[
(\{x_i\} \times (t_1,\infty)) \cap C \subset (\{x_i\} \times (t_1,\infty))\cap E  = \emptyset.
\]
It follows that $\partial R \cap C = \emptyset.$ So, $C_1 = C \setminus \overline R$ and $C_2 = C \cap R^o$ constitute a partition of $C$ into relatively open subsets. Since $(x,t) \in R^o$ and $(x,t) \in k \subset k',$ the connectedness of $k'$ implies that $k' \subset C_1.$ Therefore, by the definition of $\sim,$ we conclude that $k' \subset c$ as desired. The lemma follows.
\end{proof}

\begin{lm}\label{lmTNNBP}
Let $L$ be an exceptional thin neck. Then $L_{T_f(L)}$ does not contain a branching point.
\end{lm}
\begin{proof}
Let $R=p_1(L_{T_f(L)})\times(T_f(L),\infty)\cap E$, and let $\epsilon>0$. By Definition \ref{dfThinBranch}\ref{dTB2} there exists $t\in (T_f(L),T_f(L)+\epsilon)$ and a component $k$ of $R_t$ satisfying $k\subset K$. Suppose now that $L_{T_f(L)}$ contains a branching point. Then by Remark \ref{rmBp1Conv} there exists $t'\in (T_f(L),t]$ such that $R_{t'}$ has at least two components. Therefore, there is at least one component $e$ of $R_{t'}$ such that $e\not\leq k$. By Lemma \ref{lmEitherOrSat2}, $d(p_1(e), p_1(k))\geq 4e^{-t'}$. Furthermore, $L_{T_f(L)} \leq e \cup k.$ So, since $k \subset K,$
\begin{align}\label{eqlmTNNBP}
\ell_n(p_1(L_{T_f(L)}))&>\ell_n(p_1(e))+\ell_n(p_1(k))+4e^{-{t'}}\\
&>24e^{-t}+4e^{-t'}\notag\\
&> 28e^{-(T_f(L)+\epsilon)}\notag.
\end{align}
On the other hand,  since $L\subset N$, $\ell_n(L_{T_f(L)})\leq 24e^{-T_f(L)}$. This together with Equation \eqref{eqlmTNNBP} implies that
\[
28e^{-(T_f(L)+\epsilon)} < 24e^{-T_f(L)}.
\]
Since $\epsilon$ is arbitrary, we obtain a contradiction.
\end{proof}

\begin{cy}\label{lmTNgap}
Let $L$ be an exceptional thin neck. Let $L'$ be any thin neck such that $L\lneq L'$. Then $T_i(L')>T_f(L)$.
\end{cy}
\begin{proof}
Suppose the contrary. Then by Lemma \ref{lmTNCI}\ref{lmTNCI2}, $L_{T_f(L)}$ contains a branching point. This contradicts Lemma \ref{lmTNNBP}.
\end{proof}

\begin{cy}\label{cyETNInt}
Let $L$ be an exceptional thin neck. There is an $\epsilon>0$ such that $\mathcal{N^+}(L,\epsilon)\subset C$.
\end{cy}
\begin{proof}
Let $S$ be the set of non-exceptional thin necks $L'$ such that $L\leq L'$. It follows from Corollary \ref{cyNETNfin} that $S$ is finite. Let $t=\min_{s\in S}T_i(s)$. By Corollary \ref{lmTNgap}, $t>T_f(L)$. Let $\epsilon=t-T_f(L)$. Then $\epsilon$ satisfies the requirement.
\end{proof}

\begin{lm}\label{rmThickPt}
Let $t\in p_2(C)$ and let $c$ be a connected component of $C_t$. There is a $t'\in[t,t+\ln3)$ and a component $c'$ of $K_{t'}$ such that $c\leq c'$. Furthermore, $c'$ can be taken to belong to the same connected component of $C$ as $c.$
\end{lm}
\begin{proof}
If $c\subset K,$ take $t'=t$. Otherwise, since by Lemma~\ref{lm:CEsat}, $c$ is an $E$-segment, and by Remark~\ref{rmEsat} $K$ is $E$-saturated, we have $c \cap K = \emptyset.$ So, $c\subset N\cap C$. Therefore, $c$ is contained in an exceptional thin neck $L$. Using Corollary~\ref{cyETNInt}, choose $\epsilon_1$ small enough that $\mathcal{N}^+(L,\epsilon_1)\subset C$. Let
\[
0<\epsilon<\min(\ln3-(T_f(L)-t),\epsilon_1).
\]
By Definition~\ref{dfThinBranch}\ref{dTB2} and Lemma~\ref{lmThinNeck}\ref{Tn2}, there is a
\[
t'\in(T_f(L),T_f(L)+\epsilon)
\]
and an
\[
x\in p_1(K_{t'})
\]
such that $L_{T_f(L)}\leq e(x,t')$. By Lemma~\ref{lmThinNeck}\ref{Tn1}, again using the fact that $c$ is an $E$-segment, $c=L_t$. So, by Lemma~\ref{lmEsat}\ref{lmEsat1}
\[
c\leq L_{T_f(L)}\leq e(x,t').
\]
By the choice of $\epsilon$,
\[
t<t'<T_f(L)+\epsilon<\ln 3+t.
\]
So, we take $c' = e(x,t').$  By Remark \ref{remNLepsCon}, $\mathcal{N}^+(L,\epsilon_1)$ is a connected subset of $C$ that contains both $c$ and $c'$.
\end{proof}

\begin{df}
A subset $Z\subset\Sigma$ is \textbf{dense} if $\mu(Z)\geq \delta_1$.
\end{df}
\begin{df}
Let $S$ be a set. An \textbf{energy partition} for $S$ is a map which assigns to each element $e\in S$ a dense subset $Z(e)\subset \Sigma$ in such a way that if $e_1\neq e_2\in S$, $Z(e_1)\cap Z(e_2)=\emptyset$.
\end{df}
\begin{rem}\label{rmEpartition}
Any set $S$ that carries an energy partition satisfies $|S|\leq \frac{E}{\delta_1}$. Let $F$ be a finite forest, and denote by $Y_F$ the set of vertices of $F$ with at most one child. It is easy to see that $|F| \leq 2|Y_F|.$ So, if $Y_F$ admits an energy partition, then
\[
|F|\leq2\frac{\mu\{\Sigma\}}{\delta_1}.
\]
\end{rem}

\begin{lm}\label{lmNumTEbound}
$|\mathcal{T}_E|\leq 2\frac{\mu\{\Sigma\}}{\delta_1}$.
\end{lm}
\begin{proof}
Let $c\in\mathcal{T}_E$ and let $t_c\in p_2(c)$. By definition of $E$ and Remark~\ref{rm:long}, we have $c_{t_c}\cap D_{t_c}\neq\emptyset$. Let $x_c\in c_{t_c}\cap D_{t_c}$ and let $B_c=B(x_c,t_c)$. By Lemma \ref{lmDDisc}, $\mu\{B_c\}\geq\delta_1$. Let $c_1\neq c_2\in\mathcal{T}_E$ be maximal elements.  By Corollary \ref{cyEitherOrSat2}
\[
d(x_{c_1},x_{c_2})>2(e^{-t_{c_1}}+e^{-t_{c_2}}).
\]
Thus, by Lemma \ref{lmDisjDisc}, $B_{c_1}\cap B_{c_2}=\emptyset$. The claim follows from Remark \ref{rmEpartition} and Lemma \ref{lmMaxfin}.
\end{proof}

\begin{lm}\label{lmNumCBound}
$|\pi_0(C)|\leq 10\frac{\mu\{\Sigma\}}{\delta_1}.$
\end{lm}
\begin{proof}
Let $P\subset\pi_0(C)$ be finite. Let $Q$ be the set of vertices of $F_P$ with at most one child. We partition $Q$ into two sets, $Q_1$ and $Q_2$, such that $Q_1$ carries an energy partition, and $Q_2$ is mapped two to one into $\mathcal{T}_E$. Let $Q'\subset Q$ be the set of vertices with exactly one child. Let $q\in Q'$, let $q'$ be the unique child of $q$, and let $(x,t')\in q'$. Since $q \leq q',$ there exists $t \in p_2(q)$ such that $(x,t) \in q.$ By Lemma~\ref{lm:sco}, we have $t < t'.$ Let
\[
L_q=\{x\} \times [t,t'].
\]
Since $q$ and $q'$ are different components of $C$, and $L_q$ is a path connecting them, it follows that $L_q$ intersects at least one non-exceptional thin neck. So, for each $q \in Q',$ choose a non-exceptional thin neck $g_q$ such that $g_q \cap L_q \neq \emptyset.$ Denote by $Q_2$ the subset of elements $q\in Q'$ for which $g_q\in H$. By Lemmas \ref{lmNumETNbound} and \ref{lmNumTEbound}, $|Q_2|\leq4 \frac{\mu\{\Sigma\}}{\delta_1}$.

Denote $Q_1=Q\setminus Q_2$. We define an energy partition for $Q_1$ as follows. To each $q\in Q_1$ we associate a $t_q\in p_2(q)$ and a connected segment $s_q\subset q_{t_q}$ in such a way that the following conditions hold:
\begin{enumerate}
\item\label{PCondA}
\[
\ell_n(s_q)\geq 8e^{-t_q}.
\]
\item\label{PCondB}
\[
p_1(s_q)\times[t_q,\infty)\cap (\cup_{\{q''\in Q_1|q\leq q''\}}q'')=\emptyset.
\]
\end{enumerate}

Indeed, suppose $q \in Q_1$ has no children. By Lemma \ref{rmThickPt} there is a $t_q\in p_2(q)$ such that $q_{t_q}$ contains a component of $K_{t_q}$. Let $s_q$ be a connected component of $K_{t_q} \cap q_{t_q}$. Such $s_q$ satisfies condition~\ref{PCondA} by definition of $K$ and condition~\ref{PCondB} because $q$ has no descendants.

Otherwise $q \in Q'.$ Let $q',x,t,t'$ and $g_q$ be as above. By Lemma~\ref{rmThickPt} we may choose $t_q \in p_2(q)$ and a connected component $k_q \subset q_{t_q}\cap K$ such that $e(x,t) \leq k_q.$ In particular, $(x,t_q) \in k_q.$ By choice of $g_q$, there exists $t_{g_q} \in p_2(q_q)$ such that $(x,t_{g_q}) \in L_q.$ Since $q,q'\subset C,$ we have
\begin{equation}\label{eq:qgqd}
q \cap g_q = \emptyset = q' \cap g_q.
\end{equation}
In particular, $t' > t_{g_q} > t.$ It follows from Lemma~\ref{lm:sco} that
\begin{equation}\label{eq:qgqo}
q \leq g_q \leq q'.
\end{equation}
By Lemma~\ref{lmThinNeck}\ref{Tn2}, we have
\begin{equation}\label{eq:gqTfq'}
(g_q)_{T_f(g_q)} \leq q'.
\end{equation}
We show that
\begin{equation}\label{eq:tggtq}
t_{g_q}>t_q.
\end{equation}
Indeed, if $t_{g_q} < t_q,$ then Lemma~\ref{lm:sco} would imply that $g_q \leq q,$ which in light of equation~\eqref{eq:qgqd} and relation~\eqref{eq:qgqo}, contradicts Corollary~\ref{cyAntiSym}. Also, $t_{g_q} \neq t_q$ by equation~\eqref{eq:qgqd}. Inequality~\eqref{eq:tggtq} follows. Lemma~\ref{lm:sco} and inequality~\eqref{eq:tggtq} imply that $k_q \leq g_q.$ So,
\begin{equation}\label{eq:tigq}
t_q \leq T_i(g_q).
\end{equation}
Since $q \in Q_1,$ we have $g_q\in G\setminus H,$ that is $T(g_q) \geq \ln 3.$ So, by inequality~\eqref{eq:tigq}, we have
\begin{equation}\label{eq:Tftq}
T_f(g_q) \geq T_i(g_q) + \ln 3 \geq t_q + \ln 3.
\end{equation}
By definition of $K$ and $N$ respectively, we have
\[
\ell_n( p_1(k_{q})) > 24 e^{-t_q}, \qquad \ell_n\!\left(p_1\!\left((g_q)_{T_f(g_q)}\right)\right) \leq 24 e^{-T_f(g_q)}.
\]
So, by relation~\eqref{eq:gqTfq'} and inequality~\eqref{eq:Tftq}, we conclude that

\[
\ell_n( p_1(q'))\leq \ell_n\!\left(p_1\!\left((g_q)_{T_f(g_q)}\right)\right) \leq 24e^{-(\ln3+t_q)} < \frac1{3}\ell_n(p_1(k_q)).
\]
Therefore, there exists an $s_q\subset k_q \subset q_{t_q}$ such that $\ell_n(s_q)\geq 8e^{-t_q}$ and such that $p_1(q')\cap p_1(s_q)=\emptyset$.  By construction $s_q$ satisfies condition~\ref{PCondA}. We show it satisfies condition~\ref{PCondB} as follows. Since $q'$ is the unique child of $q$, any $q''\in Q_1\setminus \{q\}$ such that $q\leq q''$ must satisfy $q'\leq q''$. In particular, $p_1(q'')\subset p_1(q')$. Thus $p_1(q'')\cap p_1(s_q)=\emptyset$ implying condition~\ref{PCondB}.

We claim that if $q \neq q' \in Q_1,$ then
\begin{equation}\label{eq:dis}
p_1(s_q) \cap p_1(s_{q'}) = \emptyset.
\end{equation}
Indeed, if  $q\neq q'\in Q_1$ are comparable,
condition~\ref{PCondB} implies equation~\eqref{eq:dis}. So, suppose they are incomparable. We may assume without loss of generality that $q \leq_2 q'.$ Thus since $s_q\subset q,$ equation~\eqref{eq:dis} follows from Lemma~\ref{lmEitherOrSat}. By condition~\ref{PCondA}, equation~\eqref{eq:dis} and Remark~\ref{remEnPart}, we can associate with each $q$ a dense disk $B_q$ such that if $q\neq q'$, then $B_q\cap B_{q'}=\emptyset$. The assignment $q\mapsto B_q$ is an energy partition.

Therefore, we have
\[
|P|\leq 2(|Q_1|+|Q_2|)\leq 10\frac{\mu(\Sigma)}{\delta_1}.
\]
Since $P$ was an arbitrary finite subset, it follows that $\pi_0(C)$ itself is finite and thus satisfies the same inequality.
\end{proof}

For $c \in \mathcal T_E,$
let $\mathcal P_{c}$ denote the partition of $c$ into the connected components of $c\cap C$ together with the connected components of $c \cap A.$ That is, $\mathcal P_{c}$ is the partition of $c$ into non-exceptional thin necks and the connected components of their complement.
\begin{lm}\label{lm:Pc}
Let $c \in \mathcal T_E.$
\begin{enumerate}
\item\label{it:fwo}
The set $\mathcal P_c$ is well-ordered under the relation $\leq$ and finite.
\item\label{it:suc}
If $L \in \mathcal P_c$ is a non-exceptional thin neck that has a successor, its successor is a component of $c \cap C.$
\end{enumerate}
\end{lm}
\begin{proof}
\begin{enumerate}
\item
By Corollary~\ref{cyNETNfin}, the set of non-exceptional thin necks in $\mathcal P_c$ is finite. We denote them by $L_1,\ldots,L_n.$ Let $a_i = p_2(L_i)$ for $i = 1,\ldots,n.$ Since $L_i$ is connected, $a_i$ is an interval. Since $c$ is connected, $p_2(c)$ is an interval. So, $p_2(c) \setminus (L_1 \cup \cdots \cup L_n)$ is a finite collection of intervals $b_1,\ldots,b_k,$ with $k \leq n+1.$ Let
\[
M_j = p_2^{-1}(b_j) \cap c
\]
for $j = 1,\ldots,k.$ By definition, $M_j$ is $E$-saturated. By Corollary~\ref{lmTECon}, $M_j$ is connected. By Corollary~\ref{cy:pro},
\[
c = \bigcup_{i=1}^n L_i \cup \bigcup_{j = 1}^k M_j.
\]
By the construction of the $M_j,$ the union is disjoint. Therefore,
\[
\mathcal P_c = \{L_1,\ldots,L_n,M_1,\ldots,M_k\}.
\]
In particular, $P_c$ is finite. By Lemma~\ref{lm:lino}, it is well-ordered.
\item
We continue using the notation of part~\ref{it:fwo}. Without loss of generality, we may assume $L_i \leq L_{i+1}$ for $i = 1,\ldots,n-1.$ Let $i$ be such that $L = L_i.$ If the successor of $L_i$ were not a component of $c \cap C,$ it would have to be $L_{i+1}.$ Assume so by way of contradiction. But the union $a_i \cup a_{i+1}$ cannot be connected by Lemma~\ref{lmTNCI}\ref{lmTNCI1}. So, there is a $j$ such that $a_i \leq b_j \leq a_{i+1},$ where $\leq$ denotes the usual order on $\R.$ It follows that $L_i \leq_2 M_j \leq_2 L_{i+1}.$ By Lemma~\ref{lm:lino}, we conclude that $L_i \leq M_j \leq L_{i+1},$ contradicting the assumption that $L_{i+1}$ is the successor of $L_i.$
\end{enumerate}
\end{proof}

\begin{lm}\label{lmNumLnBound}
$|G|\leq 12\frac{\mu\{\Sigma\}}{\delta_1}.$
\end{lm}
\begin{proof}
It suffices to define an injective map $i:G\rightarrow \pi_0(C)\coprod \mathcal{T}_E$. Let $L\in G.$ If $L$ is maximal in the set of thin necks which are components of $c_L$, map $L$ to $c_L$. Otherwise, $L$ is not a maximal element of $\mathcal P_{c_L}.$ By Lemma~\ref{lm:Pc}\ref{it:fwo}, $\mathcal P_{c_L}$ is well-ordered. So, we let $j_L \in P_{c_L}$ be the successor of $L.$ By Lemma~\ref{lm:Pc}\ref{it:suc}, we have $j_L \subset c_L \cap C.$ Since $L$ is not a maximal thin-neck in $c_L,$ Lemma~\ref{lm:sand} implies that $j_L$ is a connected component of $C.$ So, we map $L$ to $j_L.$

We prove that $i$ is injective. Indeed, suppose $i(L) = c_L \in \mathcal{T}_E.$ By Lemma~\ref{lm:lino} maximal thin necks in $c_L$ are unique. So, there can be no other thin neck mapped to $c_L.$ On the other hand, suppose $i(L) = j_L \in \pi_0(C).$ By construction $j_L$ is a subset of a unique $c \in \mathcal T_E.$ It has a unique predecessor in $\mathcal P_c,$ which is $L.$ So, no other thin neck can be mapped to $j_L.$
\end{proof}

\begin{rem}\label{remConjInv}
It is straightforward to verify that all constructions of this section, namely $E$, $N$, $K$, $C$, $A$ and $G$, are conjugation invariant.
\end{rem}

\section{Tame geodesics}\label{SecTameGeo}
\begin{df}
Let $k>0$ and let $I$ be a cylinder. A compact embedded geodesic $\gamma \subset I$ is said to be $k$-tame if for any sub-cylinder $I'\subset I,$ we have
 \[
 \ell_{h_{st}}(\gamma\cap I')\leq 2\pi k \max(Mod I',1).
 \]
\end{df}
\begin{lm}\label{lmLNCutoffBounf}
There are constants $a_1,a_2$ such that the following holds. Let $k_1,k_2>0$ and let $(\Sigma,j,\mu)\in\mathcal{M}$. Let $I\subset\Sigma_{\C}$ be clean and doubly connected. Let $\gamma:[0,1]\rightarrow I$ be $k_1$-tame. Let $\mu$ be a thick thin measure on $I$ and let
\[
f(x):=\min\left(\frac{d\ell_{\mu}}{d\ell_{h_{st}}},k_2\right).
\]
Then
\[
\int_0^1f(\gamma(t))\vectornorm{\dot{\gamma}(t)}_{st}dt\leq k_1(k_2+1)(a_1\mu(I)+a_2)
\]
\end{lm}
\begin{proof}
First assume that either $I\cap\overline{I}=\emptyset$ or the conjugation on $I$ is longitudinal as in Definition~\ref{df:cln}. Let $e:[0,Mod I] \rightarrow \R$ be given by
\[
e(t)=\mu(S(0,t;I)).
\]
Let $M= \lfloor\frac{2\mu(I)}{\delta_2}\rfloor$, let $0=\alpha_0<\alpha_1<...<\alpha_M$ be the sequence defined by $e(\alpha_i)=i\frac{\delta_2}{2}$ and let $\alpha_{M+1}=L$. By our current assumption, $S(\alpha_i,\alpha_{i+1})$ is clean. Therefore, by Lemma \ref{ExpCylDEst} there is a constant $K\geq 1$ such that for any $1\leq j\in\mathbb{N}$, any $0\leq i\leq M$ such that $\alpha_{i+1}-\alpha_i>2jK$ and any $x\in S(\alpha_{i}+jK,\alpha_{i+1}-jK;I)$, we have
\[
f(x)\leq 2^{-j}.
\]
Therefore, we calculate
\begin{align}
\int_0^1f(\gamma(t))\vectornorm{\dot{\gamma}(t)}_{st}dt&=\sum_{i=0}^{M}\int_{\gamma^{-1}((\alpha_i,\alpha_{i+1})\times S^1)}f(\gamma(t))\vectornorm{\dot{\gamma}(t)}_{st}dt\\
&\leq\sum_{i=0}^M\left( 4\pi k_2k_1K+\sum_{j=1}^{\lfloor\frac{\alpha_{i+1}-\alpha_i}{L}\rfloor}4\pi k_1K\times 2^{-i}\right)\notag\\
&\leq\sum_{i=0}^{M} 4\pi k_1(k_2+ 1) K\notag
\end{align}
By Lemma \ref{lmMerLong} it remains only to treat the case where the conjugation on $I$ is latitudinal. Denote $I_1=S(0,\frac1{2}ModI;I)$ and $I_2=S(\frac1{2}ModI,ModI;I)$. Then $I_i\cap \overline{I_i}=\emptyset$ for $i=1,2$. So, by what has already been proved, the claim of the lemma holds for $I_1$ and $I_2$ separately. But up to addition of a constant the claim is additive. Thus the claim holds for $I=I_1\cup I_2$.
\end{proof}

For the following few lemmas let $\Sigma_\C$ be a surface with $genus(\Sigma_\C)>1$ equipped with the metric $h_{can}$ of constant curvature $-1$, and let $\beta\subset\Sigma_\C$ be a simple closed geodesic. Recall the notation of Theorem~\ref{TmThTh}. Let $I=\mathcal{C}(\beta)$. Denote by $(\rho,\theta)$ the coordinates on $I$ given by Theorem~\ref{TmThTh}\ref{it:spco}. Note that $\rho$ gives the distance from $\beta$. The coordinates $(\rho,\theta)$ are axially symmetric in the sense of equation~\eqref{eq:axsy} with
\begin{equation}\label{eq:hth}
h_\theta(\rho) = \frac{\ell(\beta) \cosh \rho}{2\pi}.
\end{equation}

\begin{lm}\label{lmConstInjRad}
There is a constant $c$ with the following significance. For any $x\in I$,
\[
\frac{1}{\pi} \leq \frac{h_{\theta}(x)}{InjRad(\Sigma_\C;h_{can},x)}\leq c.
\]
\end{lm}
\begin{proof}
First we prove the lower bound. Indeed, the non-contractible loop $\rho = \rho(x)$ has $h_{can}$-length $2\pi h_\theta(x).$ So, there is a non-constant geodesic beginning and ending at $x$ with length less than $2\pi h_\theta(x).$ Therefore, $Injrad(\Sigma_\C;h_{can},x) \leq \pi h_\theta(x)$ as desired.

We turn now to the proof of the upper bound.
Denote $b=\frac1{2}\ell(\beta)$. Let $d_0(b)=w(\beta)$ and for any $x\in I$, let $d(x)=w(\beta)-\rho(x)$. By Theorems~\ref{TmThTh} and~\ref{TmThTh2}, we need only bound the expression
\[
E(d,b)=\frac{b\cosh\rho}{\pi\sinh^{-1}(\cosh b\cosh d-\sinh{d})}
\]
in the region $0\leq\rho\leq d_0(b)$. We have
\begin{align}
E(d,b)&\leq\frac{b\cosh\rho}{\pi\sinh^{-1}(e^{-d})}\notag\\
&\leq\frac{be^{\rho}}{2\pi\sinh^{-1}(e^{-d})}\label{eq:2l}\\
&=\frac{be^{\rho}}{2\pi(e^{-d}+o(e^{-d}))}\notag\\
&=\frac{b e^{\rho+d}}{2\pi(1+o(1))}\notag\\
&=\frac{b e^{d_0(b)}}{2\pi(1+o(1))}.\label{eq:ll}
\end{align}

Pick a $d'$ large enough  that
\[
|\sinh^{-1}(e^{-d'})-e^{-d'}|\leq\frac1{2}.
\]
A bound
\begin{equation}\label{eq:bnum}
\frac{b e^{d_0(b)}}{2\pi} \leq c',
\end{equation}
combined with bound~\eqref{eq:ll} suffices to give the desired bound on $E(d,b)$ in the region $d>d'$. On the other hand, for $d\leq d'$ the desired bound on $E(d,b)$ follows from inequalities~\eqref{eq:2l} and~\eqref{eq:bnum} since $\rho\leq d_0(b)$. To prove inequality~\eqref{eq:bnum}, we use Theorem~\ref{TmThTh}\ref{it:coll} to calculate
\[
\frac{be^{d_0(b)}}{2\pi}=\frac{b}{2\pi}\left(\frac{1}{\sinh b}+\sqrt{\frac{1}{\sinh^2b}+1}\right),
\]
which is clearly bounded for $b \in [0,\infty)$.
\end{proof}

\begin{lm}\label{lm:eppi}
Let $\epsilon \leq \pi.$
\begin{enumerate}
\item\label{it:B}
$
B_{\epsilon/2}(x;h_{st})\subset B_{\epsilon h_\theta(x)}(x;h_{can}).
$
\item\label{it:dg}
Let $\gamma \subset I$ be a embedded $1$-manifold possibly with boundary.
If $x_1,x_2 \in \gamma$ satisfy $d_\gamma(x_1,x_2;h_{st}) \geq \epsilon/2,$ then
\[
d_\gamma(x_1,x_2;h_{can}) \geq \epsilon h_{\theta}(x_1).
\]
\end{enumerate}
\end{lm}
\begin{proof}
\begin{enumerate}
\item
Let $\tau$ be the function on $I$ defined by
\[
\tau(y) = \tau(\rho(y)) = \int_{\rho(x)}^{\rho(y)}\frac{d\rho}{h_{\theta}(\rho)}.
\]
So, $\tau\times \theta$ is a conformal diffeomorphism of $I$ to the flat cylinder. In particular,
\[
h_{st} = d\theta^2 + d\tau^2.
\]
Suppose $y \in B_{\epsilon/2}(x;h_{st}).$ First, we show that
\begin{equation}\label{eq:rho}
|\rho(y) - \rho(x)| \leq \epsilon h_\theta(x).
\end{equation}
Indeed, assume the contrary. By the formula~\eqref{eq:hth} for $h_\theta,$ and the fact from Theorem~\ref{TmThTh}\ref{it:coll} that
\[
|\rho| \leq w(\beta) = \sinh^{-1}(1/\sinh(\ell(\beta)/2)),
\]
we have
\begin{equation}\label{eq:dhdr}
\left|\frac{dh_\theta}{d\rho}(\rho)\right| = \left|\frac{\ell(\beta) \sinh \rho}{2\pi}\right| \leq \frac{1}{\pi}.
\end{equation}
So,
\begin{align*}
|\tau(y)| &= \left|\int_{\rho(x)}^{\rho(y)}\frac{d\rho}{h_{\theta}(\rho)}\right| \\
&\geq \left|\int_{\rho(x)}^{\rho(y)}\frac{d\rho}{h_{\theta}(x) + \frac{1}{\pi}|\rho - \rho(x)|}\right| \\
& > \int_{|\rho(x)|}^{|\rho(x)| + \epsilon h_\theta(x)} \frac{d\rho}{h_\theta(x) + \frac{\epsilon h_\theta(x)}{\pi}} \\
& \geq \frac{\epsilon}{2},
\end{align*}
which is a contradiction. Inequality~\eqref{eq:rho} follows.
Combining inequalities~\eqref{eq:rho} and~\eqref{eq:dhdr}, we obtain
\[
h_\theta(y) \leq 2 h_\theta(x), \qquad \forall y \in B_{\epsilon/2}(x;h_{st}).
\]
So, since $h_{can} = h_\theta^2 h_{st},$ on $B_{\epsilon/2}(x;h_{st})$ we have the inequality of bilinear forms
$h_{can} \leq (2h_\theta(x))^2 h_{st}.$ Therefore, for $y \in B_{\epsilon/2}(x;h_{st}),$ and $\alpha$ an $h_{st}$ geodesic from $x$ to $y,$ we have
\[
\ell_{h_{can}}(\alpha) \leq 2h_{\theta}(x) \ell_{h_{st}}(\alpha) \leq h_{\theta}(x) \epsilon.
\]
That is, $y \in B_{\epsilon h_{\theta}(x)}(x;h_{can}).$
\item
By way of contradiction, suppose $d_\gamma(x_1,x_2;h_{can}) < \epsilon h_\theta(x_1).$ Then the segment $a$ in $\gamma$ between $x_1$ and $x_2$ must be contained in $B_{\epsilon h_\theta(x_1)}(x_1;h_{can}).$ Moreover, by inequality~\eqref{eq:dhdr} we have
\[
h_\theta(y) \leq 2h_\theta(x_1), \qquad \forall y \in B_{\epsilon h_\theta(x_1)}(x_1;h_{can}).
\]
So $h_{st} \leq (2h_\theta(x_1))^{-2} h_{can}$ on $B_{\epsilon h_\theta(x_1)}(x_1;h_{can}).$ It follows that
\[
\ell_{h_{st}}(a) \leq (2 h_\theta(x_1))^{-1} \ell_{h_{can}}(a) < \frac{\epsilon}{2},
\]
which is a contradiction.
\end{enumerate}
\end{proof}

\begin{lm}\label{lmKTame}
There is a constant $C$ with the following significance. Let $k\geq1$ and let $\gamma$ be a compact embedded geodesic in $I$ such that
\[
SegWidth(\gamma,x;h_{can})\geq\frac1{k}InjRad(\Sigma_\C;h_{can},x).
\]
for all $x\in\gamma$. Then $\gamma$ is ${Ck}$-tame.
\end{lm}
\begin{proof}
Let $I' \subset I$ with $Mod(I') \geq 1.$
Choose a collection of points $x_1,\ldots,x_n \in \gamma \cap I'$ maximal with respect to the condition that
\[
d_\gamma(x_i,x_j;h_{st}) \geq \frac{\pi}{2k}, \qquad i \neq j.
\]
In particular,
\begin{equation}\label{eq:gn}
\ell_{h_{st}}(\gamma \cap I') < \frac{n\pi}{k}.
\end{equation}
By Lemmas~\ref{lm:eppi}\ref{it:dg} and~\ref{lmConstInjRad}, we have
\begin{align*}
d_\gamma(x_i,x_j;h_{can}) &\geq \frac{\pi}{k}\max(h_\theta(x_i),h_\theta(x_j)) \\
&\geq \frac1{k} \max(InjRad(x_i),InjRad(x_j)).
\end{align*}
So, by the assumption on $SegWidth,$
\[
B_{\frac{InjRad(x_i)}{2k}}(x_i;h_{can})\cap B_{\frac{InjRad(x_j)}{2k}}(x_j;h_{can}) = \emptyset, \quad i \neq j.
\]
Thus by Lemma~\ref{lmConstInjRad} we have
\[
B_{\frac{h_\theta(x_i)}{2ck}}(x_i;h_{can})\cap B_{\frac{h_\theta(x_j)}{2ck}}(x_j;h_{can}) = \emptyset, \quad i \neq j.
\]
It follows by Lemma~\ref{lm:eppi}\ref{it:B} that
\begin{equation}\label{eq:ijem}
B_{\frac{1}{4ck}}(x_i;h_{st}) \cap B_{\frac{1}{4ck}}(x_j;h_{st})= \emptyset, \quad i \neq j.
\end{equation}
Consider the case $Mod(I') \geq 1.$ Even if $x_i \in \partial I',$ at least half of $B_{\frac{1}{4ck}}(x_i;h_{st})$ is contained in $I'.$ So,
\[
Area\left(B_{\frac{1}{4ck}}(x_i;h_{st})\cap I'\right) \geq \frac{\pi}{32 c^2 k^2}, \quad Area(I') = 2\pi Mod(I').
\]
Therefore, by equation~\eqref{eq:ijem} we have
\begin{equation}\label{eq:nmod}
n \leq 64 c^2 k^2 Mod(I').
\end{equation}
Combining inequalities~\eqref{eq:gn} and~\eqref{eq:nmod}, we deduce the lemma with
\[
C = 32 c^2.
\]
On the other hand, suppose $Mod(I') < 1.$ Then
\[
Area\left(B_{\frac{1}{4ck}}(x_i;h_{st})\cap I'\right) \geq \frac{\pi}{32 c^2 k^2} Mod I'.
\]
So, by equation~\eqref{eq:ijem} we have
\begin{equation}\label{eq:n1}
n \leq 64 c^2 k^2.
\end{equation}
Combining inequalities~\eqref{eq:gn} and~\eqref{eq:n1}, the lemma follows with the same value of $C.$
\end{proof}

\section{Proof of the theorem}\label{SecProof}
The goal of this section is to prove Theorem~\ref{TmApLenBound}. We use the notation and terminology of Theorem~\ref{TmApLenBound} as well as Sections~\ref{SecTreePart},~\ref{SecThickThinPart}, and~\ref{SecTameGeo}. To begin, observe that
\begin{align}\label{eq:I0I1}
\ell_{\mu}(\gamma)&=\int_{x\in\gamma}e^{g(x)}d\ell_n(x)\\
&=\int_{x\in\gamma}\min\{e^{g(x)},2k\}d\ell_n(x)+\int_{x\in\gamma}\max\{e^{g(x)}-2k,0\}\notag.
\end{align}
We denote the first term by $I_0$ and the second by $I_1$ and deal with each term separately.

We first dispense with $I_0$.
\begin{lm}\label{lmI0}
There are constants $d_1$ and $d_2$ such that
\[
I_0 \leq k^2\{d_1\mu(\Sigma_\C)+d_2genus(\Sigma_{\C})\}.
\]
\end{lm}
\begin{proof}
Let
\[
Thick(\Sigma_{\C}):= \{x\in\Sigma_{\C}|InjRad(x;h_{can})\geq \sinh^{-1}(1)\},
\]
and
\[
Thin(\Sigma_{\C})=\Sigma_{\C}\backslash Thick(\Sigma_{\C}).
\]
In the following write $\Gamma= genus(\Sigma_{\C})$. Note that when $\Gamma>1$, Theorem~\ref{TmThTh2} implies $Thin(\Sigma_{\C})$ is contained in the union of at most $3\Gamma-3$ clean geodesic cylinders. Cleanness follows from the conjugation invariance of $h_{can}$. When $\Gamma=1,$ $Thin(\Sigma_{\C})$ is either empty or all of $\Sigma_\C.$  So, $Thin(\Sigma_{\C})$ is contained in the union of at most $2$ clean geodesic cylinders. When $\Gamma=0,$ $Thin(\Sigma_\C)$  is always empty. So, we define
\[
N_\Gamma =
\begin{cases}
3\Gamma-3, & \Gamma \geq 2,\\
2, & \Gamma  = 1,\\
0, & \Gamma = 0.
\end{cases}
\]
We calculate
\begin{align}\label{eq:I0thth}
I_0&=\int_{x\in\gamma}\min\{e^{g(x)},{2k}\}d\ell_n(x)\\
&\leq\int_{Thick(\Sigma_{\C})\cap\gamma}\frac{2k}{\sinh^{-1}(1)}d\ell_{h_{can}}+\int_{Thin(\Sigma_{\C})\cap\gamma}\min\{e^{g(x)},2k\}d\ell_n(x)\notag
\end{align}
Let $\chi(\Sigma_\C)$ denote the Euler characteristic and let
\[
M_\Gamma = \max(|\chi(\Sigma_\C)|,1).
\]
We claim that
\begin{equation}\label{eq:gth}
\int_{Thick(\Sigma_{\C})\cap\gamma}\frac{2k}{\sinh^{-1}(1)}d\ell_{h_{can}} \leq \frac{2k^2M_\Gamma}{\pi(\sinh^{-1}(1))^2}.
\end{equation}
Indeed, let $x_1,\ldots,x_n \in \gamma \cap Thick(\Sigma_\C)$ be a collection of points maximal with respect to the condition that
\[
d_\gamma(x_i,x_j;h_{can}) \geq \frac{\sinh^{-1}(1)}{k}, \quad i\neq j.
\]
In particular,
\begin{equation}\label{eq:lhn}
\ell_{h_{can}}(\gamma \cap Thick(\Sigma_\C)) < 2n \frac{\sinh^{-1}(1)}{k}.
\end{equation}
By assumption~\eqref{eq:swa}, we have
\[
B_{\frac{\sinh^{-1}(1)}{k}}(x_i;h_{can}) \cap B_{\frac{\sinh^{-1}(1)}{k}}(x_j;h_{can}) = \emptyset, \quad i \neq j.
\]
Moreover,
\begin{equation*}
Area\left(B_{\frac{\sinh^{-1}(1)}{k}}(x_i;h_{can})\right) \geq \pi \left(\frac{\sinh^{-1}(1)}{k}\right)^2, \quad Area(\Sigma_\C) = M_\Gamma.
\end{equation*}
So,
\begin{equation}\label{eq:nch}
n \leq \frac{k^2 M_\Gamma}{\pi(\sinh^{-1}(1))^2}.
\end{equation}
Combining inequalities~\eqref{eq:lhn} and~\eqref{eq:nch} we obtain
\[
\ell_{h_{can}}(\gamma \cap Thick(\Sigma_\C)) < \frac{k M_\Gamma}{\pi\sinh^{-1}(1)},
\]
which implies inequality~\eqref{eq:gth} as claimed.

Furthermore, we calculate
\begin{align}
&\int_{Thin(\Sigma_{\C})\cap\gamma}\min\{e^{g(x)},2k\}d\ell_n(x)\leq \label{eq:gtn}\\
&\qquad \qquad \leq \int_{Thin(\Sigma_{\C})\cap\gamma}\min\{e^{g(x)},2k\}\frac{d\ell_{h_{can}}(x)}{InjRad(x)}\notag\\
&\qquad \qquad  =  \int_{Thin(\Sigma_{\C})\cap\gamma}\min\{e^{g(x)},2k\}\frac{h_{\theta}(x)}{InjRad(x)}d\ell_{h_{st}}(x)\notag\\
&\qquad \qquad \leq Ck(2k+1)( a_2N_\Gamma a_1\mu(\Sigma_\C))\notag
\end{align}
using Lemmas \ref{lmLNCutoffBounf}, \ref{lmKTame} and \ref{lmConstInjRad} in the last inequality. Combining inequalities~\eqref{eq:I0thth},~\eqref{eq:gth}, and~\eqref{eq:gtn}, we obtain
\begin{align}\label{I0Ineq}
I_0 \leq \frac{2k^2M_\Gamma}{\pi(\sinh^{-1}(1))^2} + Ck(2k+1)( a_2N_\Gamma+a_1\mu(\Sigma_\C)).
\end{align}

Note that when $\Gamma=0$, the first term in the right hand side of \eqref{I0Ineq} can be absorbed in the coefficient $d_1$. This is so because the gradient inequality implies that for any $(\Sigma,\mu)\in\mathcal{M}$ with $genus(\Sigma_{\C})=0$, we have $\mu(\Sigma)\geq\delta_1.$
\end{proof}

\begin{as}\label{as:longhot}
From here until Lemma~\ref{lm:short}, we assume that conditions~\eqref{eq:long} and~\eqref{eq:hot} are satisfied. Thus we may use the constructions and results of Section~\ref{SecThickThinPart}.
\end{as}
The following arguments deal with $I_1$ given the preceding assumption.
\begin{align}
I_1&=\int_{x\in\gamma}\max\{e^{g(x)}-2k,0\}d\ell_n(x)\notag\\
&=\int_{x\in\gamma}\int_{2k}^{\max\{e^{g(x)},2k\}}dyd\ell_n(x)\notag.
\end{align}

Change variables to $y=e^t$. Then, using Fubini's theorem,
\begin{align}\label{eq:I2I3}
I_1&=\int_{\ln2k}^{\infty}\int_{x\in D_t}d\ell_n(x)e^tdt\\
&=\int_{\ln2k}^{\infty}\int_{x\in D_t\cap C}d\ell_n(x)e^tdt+\int_{\ln2k}^{\infty}\int_{x\in D_t\cap A}d\ell_n(x)e^tdt\notag\\
&\leq \int_{\ln2k}^{\infty}\int_{x\in C_t}d\ell_n(x)e^tdt+\int_{\ln2k}^{\infty}\int_{x\in D_t\cap A}d\ell_n(x)e^tdt.\notag
\end{align}
In the last line, denote the first term, giving the contribution of the thick part, by $I_2$, and the second, giving the contribution of the thin part, by $I_3$.

\begin{lm}\label{lmI1}
There is a constant $d_3$ such that $I_2\leq d_3\mu(\Sigma_{\C}).$
\end{lm}
\begin{proof}
Discretizing, we have the bound
\begin{equation}\label{eq:dscr}
I_2\leq(2\ln 3)(2k)\sum_{i=0}^{\infty}3^{2(i+1)}\ell_n\{C_{2i\ln 3 + \ln 2k}\}.
\end{equation}
Write $c_i:=C_{2i\ln 3 +\ln 2k}$ and denote by $c_{ij}$ the components of $c_i$. By Lemma~\ref{rmThickPt}, for each $j$ choose
\[
t_{ij}\in [2i\ln3 +\ln 2k,(2i+1)\ln3 +\ln 2k)
\]
and a component $k_{ij}\subset K_{t_{ij}}$ such that $c_{ij}\leq k_{ij}$. In the case that $c_{ij} \subset K,$ we insist upon taking $t_{ij} = 2i\ln 3 + \ln 2k$ and $k_{ij} = c_{ij}.$ For each $j$ let $T_{ij}\subset k_{ij}$ be a segment of length
\begin{equation}\label{eq:dTij}
\ell_n(T_{ij}) = \frac{8}{e^{t_{ij}}}\left\lfloor\frac{e^{t_{ij}}\ell_n(k_{ij})}{8}\right\rfloor.
\end{equation}
Since $k_{ij} \subset K_{t_{ij}},$ we have
\begin{equation}\label{eq:kij}
\ell_n(k_{ij}) \geq 24e^{-t_{ij}}.
\end{equation}
So,
\begin{equation}\label{eq:Tij}
\ell_n(T_{ij}) \geq \frac{3}{4}\ell_n(k_{ij}).
\end{equation}
Partition $T_{ij}$ into segments $T_{ijk}$ such that
\begin{equation}\label{eq:Tijk}
\ell_n(T_{ijk})=8 e^{-t_{ij}}.
\end{equation}
Let $n_{ij}$ be the number of segments $T_{ijk}$. We claim that
\begin{equation}\label{eq:ncij}
\frac{8n_{ij}}{3^{2i}2k}\geq\frac3{4}\ell_n\left(c_{ij}\right).
\end{equation}
Indeed, suppose first that $c_{ij} = k_{ij}.$ Then $t_{ij} = 3^{2i}2k.$ So, using inequality~\eqref{eq:Tij}, we have
\[
\frac{8n_{ij}}{3^{2i}2k} = 8e^{-t_{ij}}n_{ij} = \ell_n(T_{ij}) \geq \frac{3}{4}\ell_n(k_{ij}) = \frac{3}{4} \ell_n(c_{ij}).
\]
On the other hand, suppose $c_{ij} \neq k_{ij}.$ Then $c_{ij} \subset N,$ so $\ell(c_{ij}) \leq \frac{24}{3^{2i} 2k}.$ But
\[
n_{ij} = \frac{e^{t_{ij}}}{8} \ell_n(T_{ij}) \geq \frac{3}{4} \frac{e^{t_{ij}}}{8}\ell_n(k_{ij}) \geq \frac{3}{4} \frac{24}{8} > 2.
\]
Since $n_{ij}$ is an integer, we have $n_{ij} \geq 3.$ Thus
\[
\frac{8n_{ij}}{3^{2i}2k} \geq \frac{24}{3^{2i}2k} \geq \ell(c_{ij}).
\]
Inequality~\eqref{eq:ncij} follows. Setting
\begin{equation}\label{eq:dni}
n_i=\sum n_{ij},
\end{equation}
inequality~\eqref{eq:ncij} implies
\begin{equation}\label{eq:ni}
\frac{8n_{i}}{3^{2i}2k}\geq\frac3{4}\ell_n\left(C_{2i\ln3 + \ln 2k}\right).
\end{equation}
To each segment $T_{ijk},$ assign a disk $B_{ijk}\subset \Sigma_{\C}$ as in Remark \ref{remEnPart}. So, if $k_1\neq k_2$, then $B_{ijk_1}\cap B_{ijk_2}=\emptyset$. Let $m_i$ denote the number of disks $B_{i'j'k'}$ with $i'<i$ that intersect one of the disks $B_{ijk}$ but do not intersect any disk $B_{i''j''k''}$ for $i'<i''<i$. We claim that
\begin{equation}\label{eq:mi}
m_i\leq \frac{3}{4}n_i.
\end{equation}
Indeed, let $m_i'$ denote the number of intervals $T_{i'j'k'}$ with $i' < i$ such that $p_1(T_{i'j'k'}) \cap p_1(T_{ijk}) \neq \emptyset$ for some $j,k,$ and $p_1(T_{i'j'k'}) \cap p_1(T_{i''j''k''}) = \emptyset$ for $i' < i'' < i.$ By Remark~\ref{remEnPart} we have $m_i \leq m_i'.$ So inequality~\eqref{eq:mi} will follow if we show
\begin{equation}\label{eq:mi'}
m_i' \leq \frac{3}{4}n_i.
\end{equation}
Let $m_{ij}'$ denote the number of intervals $T_{i'j'k'}$ with $i' < i$ such that $p_1(T_{i'j'k'}) \cap p_1(T_{ijk}) \neq \emptyset$ for some $k$ and
\begin{equation}\label{eq:i'i''}
p_1(T_{i'j'k'}) \cap p_1(T_{i''j''k''}) = \emptyset
\end{equation}
for $i' < i'' < i.$ Then
\[
m_i' \leq \sum_j m_{ij}'.
\]
So, by equation~\eqref{eq:dni} inequalities~\eqref{eq:mi'} and~\eqref{eq:mi} will follow if we prove
\begin{equation}\label{eq:mij}
m_{ij}' \leq \frac{3}{4}n_{ij}
\end{equation}
for all $j.$
To see this, choose integers $a_1$ and $0\leq a_2<3$ such that $n_{ij}=3a_1+a_2$. By equations~\eqref{eq:dTij}, ~\eqref{eq:kij} and~\eqref{eq:Tijk}, we have $a_1 \geq 1.$ Enumerate the intervals $T_{i'j'k'}$ by $T_{i_lj_lk_l}$ for $l =1,\ldots,m_{ij}'.$
By~\eqref{eq:i'i''} we conclude the intervals $p_1(T_{i_lj_lk_l})$ are pairwise disjoint.
Since $i_l<i$ for all $l,$ by equation~\eqref{eq:Tijk} we have $\ell_n(T_{i_lj_lk_l})\geq 3\ell_n(T_{ijk})$. Therefore it is easy to verify that  $m_{ij}'$ is at most $a_1+2$ when $a_2\neq 0$, and at most $a_1+1$ when $a_2=0$. But in the first case,
\[
\frac{m_{ij}'}{n_{ij}} = \frac{a_1+2}{3a_1+a_2}\leq\frac{3}{4},
\]
and in the second case,
\[
\frac{m_{ij}'}{n_{ij}} = \frac{a_1+1}{3a_1}\leq\frac{2}{3}\leq \frac{3}{4}.
\]
Inequalities~\eqref{eq:mij},~\eqref{eq:mi'} and~\eqref{eq:mi} follow.

Combining inequalities~\eqref{eq:mi},~\eqref{eq:ni} and~\eqref{eq:dscr}, we obtain
\begin{align}
\frac{\mu(\Sigma_{\C})}{\delta_1}&\geq\sum_{i=0}^{\infty}n_i-m_i\\
&\geq\sum_{i=0}^{\infty}\frac1{4}n_i\notag\\
&\geq\frac{2k}{4\cdot8\cdot4}\sum_{i=0}^{\infty}3^{2i+1}\ell_n(C_{2i\ln3+\ln 2k})\notag\\
&\geq\frac1{256\ln3}I_2\notag.
\end{align}
\end{proof}

We now deal with the thin part.
\begin{lm}\label{lmI2}
There is a constant $d_4$ such that $I_3\leq d_4\mu(\Sigma_{\C})$.
\end{lm}
\begin{proof}
Recall that $G$ denotes the collection of all non-exceptional thin necks. Let $G_1=\{v\in G|T(v)\leq\ln 48\}$. The contribution of the thin part is
\begin{align}
I_3&=\int_{\ln{2k}}^{\infty}\int_{x\in A_t\cap D}d\ell_n(x)e^tdt\notag\\
&=\sum_{v\in G}\int_{T_i(v)}^{T_f(v)}\int_{x\in v_t\cap D}d\ell_n(x)e^tdt\notag\\
&\leq \sum_{v\in G_1}\int_{T_i(v)}^{T_f(v)}24dt+\sum_{v\in G\backslash G_1}\int_{T_i(v)}^{T_f(v)}\int_{x\in v_t\cap D}d\ell_n(x)e^tdt\notag\\
&\leq 24\ln48\frac{12\mu(\Sigma_{\C})}{\delta_1}+ \sum_{v\in G\backslash G_1}\int_{T_i(v)}^{T_f(v)}\int_{x\in v_t\cap D}d\ell_n(x)e^tdt.\notag
\end{align}
The last transition relies on Lemma~\ref{lmNumLnBound}.

To each $v\in G\backslash G_1$ we associate an annulus as follows. Let $x_v$ denote the midpoint of $v_{T_f(v)}$ with respect to the metric $h_{can}$ and let
\[
r_v=\frac1{2}\ell_{h_{can}\left(v_{T_f(v)}\right)}.
\]
Let
\begin{gather*}
r^1_v=r_v+2r(x_v)e^{-T_f(v)},\\
r^2_v=r(x_v)e^{-T_i(v)},\\
B^1_{v}:=B_{r^1_v}(x_v;h_{can})\subset\Sigma_{\C}, \quad B^2_v=B_{r^2_v}(x_v;h_{can})\subset\Sigma_{\C},
\end{gather*}
and $A_v:=B_v^2\backslash B_v^1.$

The following claims justify the definitions of the preceding paragraph.

\begin{cm}\label{cmr1r2estimates}
$r^1_v< r^2_v\leq \frac{1}{2k}r(x_v)<\frac1{2}SegWidth(\gamma,x_v;h_{can})$.
\end{cm}
\begin{proof}
The inequality $r^2_v\leq \frac{1}{2k}r(x_v)$ follows from the fact that $T_i(v)\geq\ln 2k$. The inequality $\frac{1}{2k}r(x_v)<\frac1{2}SegWidth(\gamma,x_v;h_{can})$ is a special case of assumption~\eqref{eq:swa}. By inequality~\eqref{NormEstimate} and the definition of $N$,
\[
r_v\leq \frac3{4}\ell_n(v_{T_f(v)})r(x_v)\leq\frac3{4}\cdot 24 e^{-T_f(v)}r(x_v).
\]
Therefore, $r^1_v\leq20 r(x_v)e^{-T_f(v)}$. Since
\[
T_f(v)-T_i(v)=T(v)>\ln20,
\]
the inequality $r^1_v< r^2_v$ follows.
\end{proof}

\begin{cm}\label{cmSegWidthBound}
Let $t\geq T_i(v)+\ln 48$. Then $p_1(v_t)\subset B^2_v$.
\end{cm}
\begin{proof}
This is immediate from  inequality~\eqref{NormEstimate} and the definition of $N$.
\end{proof}
\begin{cm}\label{cmAvDisj}
For any $v_1,v_2\in G\setminus G_1,$ we have
\[
v_1\neq v_2\Rightarrow A_{v_1}\cap A_{v_2}=\emptyset.
\]
\end{cm}
\begin{proof}
Consider first the case that $p_1(v_1) \cap p_1(v_2) \neq \emptyset.$ By Lemmas~\ref{lmEitherOrSat} and~\ref{lmThinNeck}\ref{Tn2}, we may assume without loss of generality that
\[
(v_1)_{T_f(v_1)} \leq v_2.
\]
So, we have
\begin{equation}\label{eq:x1x2rv1}
d(x_{v_1},x_{v_2};h_{can})\leq r_{v_1}.
\end{equation}
Thus by Claim~\ref{cmr1r2estimates} we obtain
\[
d(x_{v_1},x_{v_2};h_{can}) < r(x_{v_1}).
\]
So, it follows from Lemma~\ref{eq:drl1} that
\[
r(x_{v_2})\leq 2r(x_{v_1}).
\]
Thus we have
\[
r^2_{v_1}=r(x_{v_2})e^{-T_i(v_2)}\leq2r(x_{v_1})e^{-T_f(v_1)}.
\]
The preceding inequality and inequality~\eqref{eq:x1x2rv1} together imply
\[
d(x_{v_1},x_{v_2};h_{can})+r^2_{v_1}\leq r_{v_1}+2r(x_{v_1})e^{-T_f(v_1)}.
\]
It follows that $B^2_{v_2}\subset B^1_{v_1}$ giving the claim.

On the other hand, if $p_1(v_1) \cap p_2(v_2) = \emptyset,$  Corollary \ref{cyEitherOrSat2} implies
\[
d(x_{v_1},x_{v_2};h_n)\geq 4e^{-\min\{T_i(v_1),T_i(v_2)\}}.
\]
But by inequality~\eqref{NormEstimate} and the definition of $B^2_{v_j}$ for $j=1,2$, we have the inequality
\[
d(x_{v_j},\partial B^2_{v_j}\cap\gamma;h_n)\leq 2e^{-T_i(v_j)}.
\]
Therefore, $B^2_{v_1}\cap B^2_{v_2}\cap\gamma=\emptyset$. By definition of $SegWidth$ and by Claim~\ref{cmr1r2estimates}, we conclude $B^2_{v_1}\cap B^2_{v_2}=\emptyset$.

\end{proof}
\begin{cm}\label{CmAvClean}
$A_v$ is clean.
\end{cm}
\begin{proof}
Keeping in mind Remark \ref{remConjInv}, by construction $A_{\overline{v}}=\overline{A_v}$. In particular if $v=\overline{v}$, then $A_v$ is conjugation invariant. If $v\neq \overline{v},$ then by Claim~\ref{cmAvDisj} we have $A_{v}\cap\overline{A_v}=\emptyset$.
\end{proof}

\begin{cm}
\begin{equation}\label{eq:bv1}
\ell_n(B^1_v \cap \gamma)\leq 32e^{-T_f(v)}.
\end{equation}
\end{cm}
\begin{proof}
We have
\[
\ell_n(B^1_v \cap \gamma)=\ell_n(v_{T_f(v)})+\ell_n(B^1_v \cap \gamma\setminus p_1(v_{T_f(v)})).
\]
Denote the first term by $a$ and the second term by $b$. The definition of $N$ implies that $a\leq 24e^{-T_f(v)}$. Since $T(v) \geq \ln 48,$ we have $2r(x_v)e^{-T_f(x_v)} \leq \frac{1}{2}r(x_v).$ So, we use inequality~\eqref{NormEstimate} to verify that $b\leq 8e^{-T_f(v)}$.
\end{proof}

We return to estimating $I_3.$ By the inclusion of Claim~\ref{cmSegWidthBound}, for $v\in G\backslash G_1$ we obtain
\begin{align}
&\int_{T_i(v)}^{T_f(v)}\int_{x\in  p_1(v_t\cap D)}d\ell_n(x)e^tdt \leq \notag\\
&\qquad \qquad \qquad\leq 24\ln48 +\int_{T_i(v)+\ln48}^{T_f(v)}\int_{x\in  p_1(v_t\cap D)}d\ell_n(x)e^tdt \notag\\
&\qquad \qquad \qquad\leq 24\ln48 +\int_0^{T_f(v)}\int_{x\in B^2_v\cap  p_1(v_t\cap D)}d\ell_n(x)e^tdt\notag.
\end{align}
By equation~\eqref{eq:bv1} we have
\begin{align}
&\int_0^{T_f(v)}\int_{x\in B^2_v\cap  p_1(v_t\cap D)}d\ell_n(x)e^tdt \leq\notag\\
&\qquad \qquad \qquad\leq \int_{x\in A_v\cap\gamma}e^{g(x)}d\ell_n(x)+\int_{x\in B_v^1}e^{T_f(v)}d\ell_n(x)\notag\\
&\qquad\qquad\qquad \leq \int_{x\in A_v\cap\gamma}e^{g(x)}d\ell_n(x)+32\notag.
\end{align}

We bound the integral on the domain $A_v\cap\gamma$ as follows. By Lemma \ref{lmThinNeck}\ref{Tn4}, we have
\[
\frac{d\ell_{\mu}}{d\ell_n}(x)=e^{g(x)}\leq \frac{24}{d(x,x_v;h_n)},
\]
for any $x\in A_v\cap\gamma$. So, by inequality~\eqref{NormEstimate} we deduce that
\begin{equation}\label{eq:mucan}
r(x_v)\frac{d\ell_{\mu}}{d\ell_{h_{can}}}(x)\leq\frac{36r(x_v)}{d(x,x_v;h_{can})}.
\end{equation}

In polar coordinates on $B_v^2,$ the metric $h_{can}$ is given by formulas~\eqref{eq:pch} and~\eqref{eq:forhth}, and
$h_{can}|_{A_v} = h_\theta^2 h_{st}.$ So, since by Claim \ref{cmr1r2estimates} we have $r_v^2 \leq 1,$ it follows that
\begin{equation}\label{eq:canst}
\frac{d\ell_{h_{can}}}{d\ell_{h_{st}}}(x) = h_{\theta}(x)\leq 3 d(x,x_v;h_{can}), \qquad x \in A_v.
\end{equation}
Using the chain rule to combine inequalities~\eqref{eq:mucan} and~\eqref{eq:canst}, we obtain
\[
\frac{d\ell_{\mu}}{d\ell_{h_{st}}}(x)\leq 108
\]
for $x\in A_v\cap\gamma$. On the other hand, we have
\[
\int_{x\in A_v}e^{g(x)}d\ell_n(x)=\int_{x\in A_v}\frac{d\ell_{\mu}}{d\ell_{h_{st}}}d\ell_{h_{st}}.
\]

It follows from Claim \ref{cmr1r2estimates} that $\gamma \cap B_v^2$ is a radial geodesic, so $\gamma \cap A_v$ is $1$-tame. Thus by Claim \ref{CmAvClean} we may apply Lemma \ref{lmLNCutoffBounf} with $k_1=1$ and $k_2=108$ to deduce the bound
\[
\int_{x\in A_v}e^{g(x)}d\ell_n(x)\leq 108(a_1\mu(A_v)+a_2).
\]

Collecting the terms, using Claim~\ref{cmAvDisj}, and invoking Lemma \ref{lmNumLnBound} for the constant terms, we obtain
\begin{align}
&\sum_{v\in G\backslash G_1}\int_{t_0(v)}^{t_1(v)}\int_{x\in v_t\cap D}d\ell_n(x)e^tdt\leq \\
& \qquad \qquad \qquad \leq \left(\frac{12}{\delta_1}(24\ln48+108a_2+32)+108a_1\right)\mu(\Sigma_\C),\notag
 \end{align}
 which completes the proof.
\end{proof}

\begin{lm}\label{lm:short}
Let $\beta \subset \gamma$ be the union of the connected components $\beta_i$ of $\gamma$ such that
\begin{equation}\label{eq:short}
\ell_n(\beta_i) < 4 e^{-\max_{x \in \beta_i}g(x)},
\end{equation}
and
\begin{equation}\label{eq:warm}
\ln 2k \leq \max_{x \in \beta_i} g(x).
\end{equation}
We have
\[
\ell_\mu(\beta) \leq 4\frac{\mu(\Sigma_\C)}{\delta_1}.
\]
\end{lm}
\begin{proof}
Using inequality~\eqref{eq:short} we estimate
\[
\ell_\mu(\beta_i) \leq \ell_n(\beta_i) \max_{x \in \beta_i} \frac{d\ell_\mu}{d\ell_n}(x) \leq 4.
\]
Choose $x_i \in \beta_i$ such that
\[
\frac{d\ell_\mu}{d\ell_n}(x_i) = \max_{x \in \beta_i} \frac{d\ell_\mu}{d\ell_n}(x).
\]
Set $t_i = \frac{d\ell_\mu}{d\ell_n}(x_i).$ By inequality~
\eqref{eq:warm}, we have $(x_i,t_i) \in D.$ So, Lemma~\ref{lmDDisc} implies that $\mu(B(x_i,t_i)) \geq \delta_1,$ and Lemma~\ref{lmDisjDisc} implies the disks $B(x_i,t_i)$ are pairwise disjoint implying the claim.
\end{proof}

\begin{proof}[\textbf{Proof of Theorem \ref{TmApLenBound}}]
By equation~\eqref{eq:I0I1}, it suffices to bound $I_0$ and $I_1.$ Lemma~\ref{lmI0} takes care of $I_0$ unconditionally. The components $\gamma_i \subset \gamma$ that violate condition~\eqref{eq:hot} contribute trivially to $I_1,$ so we assume without loss of generality that condition~\eqref{eq:hot} does hold. Then Lemma~\ref{lm:short} allows us to disregard the contribution to $I_1$ of components $\gamma_i \subset \gamma$ that violate condition~\eqref{eq:long}. So, without loss of generality, we impose Assumption~\ref{as:longhot}.  Then equation~\eqref{eq:I2I3}, Lemma~\ref{lmI1} and Lemma~\ref{lmI2} bound $I_1$ implying Theorem~\ref{TmApLenBound}.
\end{proof}

\section{Applications}\label{SecApp}

We now apply Theorem \ref{TmApLenBound} to prove the theorems stated in the introduction.
\begin{lm}\label{lmDiscRef}
Let $\Sigma$ be a closed Riemann surface with an isometric involution $\psi$. Suppose $p\in\Sigma$ is fixed under $\psi$. Then for any $r>0$, $B_r(p;h_{can})$ is $\psi$ invariant.
\end{lm}
\begin{proof}
For any $q\in\Sigma$ we have
\[
d(p,q;h_{can})=d(\psi(p),\psi(q);h_{can})=d(p,\psi(q);h_{can}).
\]
In particular, $q\in B_r(p;h_{can})$ if and only if $\psi(q)\in B_r(p;h_{can})$.
\end{proof}

\begin{lm}\label{lmSegWid}
Let $\Sigma$ be a closed Riemann surface with an isometric involution $\psi$. Let $\gamma$ be either a minimal geodesic connecting two points or an embedded geodesic fixed under an isometric involution $\psi$. Then for any $p\in\gamma,$
\[
SegWidth(\gamma,p;h_{can})\geq InjRad(\Sigma;h_{can},p).
\]
\end{lm}

\begin{proof}
If $\gamma$ is a minimal geodesic the claim is obvious. Suppose $\gamma$ is fixed under $\psi$. For any $p\in\gamma$ and $r\in(0, InjRad(\Sigma;h_{can},p))$, write $B=B_{r}(p;h_{can})$. Let $\Sigma_{\psi}$ denote the fixed points of $\psi$. $B$ is conjugation invariant by Lemma \ref{lmDiscRef}. Denote by $B_\psi$ the fixed points of $\psi|_B$. It is clear that $B_{\psi}$ is a radial geodesic and that $B_{\psi}=\Sigma_{\psi}\cap B$. It follows that $B_{r}(p;h_{can})\cap\gamma$ is contained in a radial geodesic and so, that $SegWidth(\gamma,p;h_{can})\geq r$.
\end{proof}

Let $(M,\omega)$ be a symplectic manifold, $L\subset M$ a Lagrangian submanifold, and $J$ an $\omega$-tame almost complex structure.
\begin{df}\label{BoundCond}
Let $S$ be a family of compact Riemann surfaces, possibly with boundary. We say that the data of $S$ together with $(M,\omega,L,J)$ comprise a \textbf{K-bounded setting} if one of the following holds.
\begin{enumerate}
\item\label{BoundCond1} $M$ and $L$ are compact.
\item \label{BoundCond2}$L=\emptyset$ and
\[
\max\left\{\vectornorm{R},\vectornorm{J}_2,\frac1{i}\right\}<K.
\]
\item\label{BoundCond3}
$L$ is $\frac1{K}$-Lipschitz and
\[
\max\left\{\vectornorm{R}_2,\vectornorm{J}_2,\vectornorm{B}_2,\frac1{i}\right\}<K.
\]
\item\label{BoundCond4}
Each connected component $L'$ of $L$ is $\frac1{K}$-Lipschitz and
\[
\max\left\{\vectornorm{R}_2,\vectornorm{J}_2,\vectornorm{B}_2,\frac1{i}\right\}<K.
\]
Furthermore, for each $\Sigma \in S$, there is a conformal metric $h$ of constant curvature $0,\pm 1,$ and of unit area in case of zero curvature, such that $\partial\Sigma$ is totally geodesic and $\frac1{K}$-Lipschitz.

\end{enumerate}
\end{df}

Let $\mathcal{M}$ be a family of $J$-holomorphic curves in $M$ with boundary in $L$, and let $S$ be the family of Riemann surfaces that are domains of members of $\mathcal{M}$. For each $(u,\Sigma)\in \mathcal{M}$ let $\mu_u$ be the corresponding energy measure
\[
\mu_u(U):=\int_U\vectornorm{du}^2dvol_\Sigma,
\]
for $U\subset\Sigma$ an open subset. The following is Theorem~2.8 of~\cite{GS13}. In cases~\ref{BoundCond1} and~\ref{BoundCond2} of Definition~\ref{BoundCond}, it follows straightforwardly from the discussion in Chapter 4 of~\cite{MS2}.

\begin{tm}\label{TmThickThinM}
Suppose $S$ and $(M,\omega,J,L)$ comprise a $K$-bounded setting. Then the collection of measured Riemann surfaces
\[
\{(\Sigma,\mu_u)|(u,\Sigma)\in\mathcal{M}\}
\]
is uniformly thick thin with constants depending only on $K$.
\end{tm}

\begin{proof}[\textbf{Proof of Theorems \ref{TmRIIBountdary}, \ref{TmRIIReal}, \ref{Estimates} and \ref{Estimates2}}]
 Theorems \ref{TmRIIBountdary}, \ref{TmRIIReal}, \ref{Estimates} and \ref{Estimates2} are  immediate from  Theorem \ref{TmApLenBound}, Lemma \ref{lmSegWid}, Remark \ref{rmScaling} and Theorem \ref{TmThickThinM}. Note that Theorem \ref{Estimates2} is covered by Theorem \ref{TmThickThinM} applied to case \ref{BoundCond2} in Definition \ref{BoundCond}.
\end{proof}

To prove Theorem~\ref{tmDiamEst}, we need the following additional definitions and lemmas.
\begin{df}
Let $\Sigma$ be a Riemann surface with boundary. A \textbf{bridge} in $\Sigma$ is a length minimizing geodesic connecting two components $\gamma_1\neq \gamma_2$ of $\partial\Sigma$. An \textbf{admissible bridge} is a bridge $\gamma$ such that for any $p\in\gamma$ and any $r\in\frac1{3}InjRad(\Sigma_{\C};h_{can},p)$, we have $B_r(p;h_{can})\cap\partial\Sigma\subset\gamma_1\cup\gamma_2$.
\end{df}
In the following Lemmas \ref{lmSegWidth2} to \ref{lmBridge}, we shall consider a Riemann surface $\Sigma$ with boundary and its complex double $\Sigma_\C.$ All metric quantities and sets are with respect to the metric $h_{can}$ on $\Sigma_{\C}$, so we omit them from the notation.
\begin{lm}\label{lmSegWidth2}
Let $\gamma$ be an admissible bridge in $\Sigma$ connecting the components $\gamma_1$ and $\gamma_2$ of $\partial\Sigma$. Then for any $p\in\gamma\cup\overline{\gamma} \subset \Sigma_\C,$
\[
SegWidth(\gamma\cup\overline{\gamma},p)\geq\frac1{3}InjRad(p).
\]
\end{lm}
\begin{proof}
Let $p\in\gamma$ and $r\in(0,\frac1{3}InjRad(p))$. Write
\[
B=B_r(p),
\]
and
\[
B':=B_{3r}(p).
\]
We need to show that  $B\cap(\gamma\cup\overline{\gamma})$ is a radial geodesic. If $B\cap\partial\Sigma=\emptyset$ then $(\gamma\cup\overline{\gamma})\cap B = \gamma \cap B$ is a minimizing geodesic, so the claim follows from Lemma \ref{lmSegWid}. Otherwise, let $q$ be the point of $\gamma \cap \partial \Sigma$ closest to $p,$ and let
\[
B'':=B_{2r}(q).
\]
Since $\gamma$ is admissible, $B \cap \partial\Sigma \subset \gamma_1 \cup \gamma_2.$ Thus since $\gamma$ minimizes length between $\gamma_1$ and $\gamma_2,$ we have $q \in B.$ The triangle inequality implies that $B\subset B''\subset B'$. Since $h_{can}$ has constant curvature and $B'$ is a normal disk, it follows that $B''$ is a normal disk. By Lemma~\ref{lmSegWid} and the fact that $B''$ is a normal disk, we have that each of $\gamma\cap B''$ and $\overline{\gamma}\cap{B''}$ lies on a radial geodesic of $B''$. Furthermore $\gamma$ and $\overline \gamma$ intersect $\partial\Sigma$ at $q$ in a right angle. So, $B''\cap(\gamma\cup\overline{\gamma})$ is a radial geodesic which we denote by $C$. Since $p\in C$, again using the fact $h_{can}$ has constant curvature, $B\cap C = B \cap (\gamma \cup \bar \gamma)$ is a radial geodesic in $B$.
\end{proof}

\begin{df}
Let $\Sigma$ be a Riemann surface with boundary and let $\gamma_i,$ for $i=1,2$, be components of $\partial\Sigma$. An \emph{admissible chain} connecting $\gamma_1$ and $\gamma_2$ is a pairwise disjoint sequence
\[
\alpha_1=\gamma_1,\alpha_2,\ldots,\alpha_n=\gamma_2,
\]
of components of $\partial\Sigma$ with admissible bridges $\beta_i$ connecting $\alpha_i$ and $\alpha_{i+1}$ for each $1 \leq i \leq n-1.$
\end{df}

\begin{lm}\label{lmBridegChain}
Let $\Sigma$ be a Riemann surface with boundary and let $\chi_i,$ for $i=1,2$, be components of $\partial\Sigma$. There exists an admissible chain connecting $\chi_1$ and $\chi_2.$
\end{lm}

\begin{lm}\label{lmBridge}
Let $\beta$ be a bridge connecting boundary components $\gamma_1$ and $\gamma_2$. Suppose there is a boundary component $\gamma_3$ and a point $p\in \beta$ such that
\[
d(p,\gamma_3)<\frac13InjRad(p).
\]
Let $\beta_i$ for $i=1,2,$ be a bridge connecting $\gamma_3$ with $\gamma_i$. Then
\[
\ell(\beta_i)<\ell(\beta).
\]
\end{lm}
\begin{proof}
Let $\delta$ be the length minimizing geodesic from $p$ to $\gamma_3$ and let $q=\delta\cap\gamma_3$. Write $r= InjRad(p)$. Let $B_1=B_r(p)$ and $B_2=B_{\frac{2}{3}r}(q)$. By the triangle inequality, $B_2\subset B_1$. Since $h_{can}$ has constant curvature and $B_1$ is a normal disk, so is $B_2$. In particular $InjRad(q)\geq \frac23InjRad(p)$. It follows that
\begin{equation}\label{eqbrideg1}
d(p,q)<\frac13InjRad(p)\leq\frac12InjRad(q).
\end{equation}
On the other hand, by Lemma \ref{lmSegWid} applied to the Riemann surface $\Sigma_{\C}$ with conjugation as the isometric involution we have, for $i=1,2$,
\begin{equation}\label{eqbrideg2}
d(q,\gamma_i)\geq InjRad(q).
\end{equation}
Denote by $p_i$ for $i=1,2,$ the point where $\beta$ meets $\gamma_i$. Combining inequalities \eqref{eqbrideg1}, \eqref{eqbrideg2} and the triangle inequality,
\begin{equation}
d(p,p_i)\geq d(q,p_i)-d(p,q)\geq d(q,\gamma_i)-d(p,q)>\frac12InjRad(q).
\end{equation}
Therefore, since $d(p,p_1) + d(p,p_2) =\ell(\beta)$, we have
\begin{equation}\label{eqbrideg3}
d(p,p_i)<\ell(\beta)-\frac12InjRad(q).
\end{equation}
Combining inequalities \eqref{eqbrideg1} and \eqref{eqbrideg3}, we conclude
\begin{equation}
\ell(\beta_i)\leq d(p_i,q)\leq d(p_i,p)+d(p,q)<\ell(\beta).
\end{equation}
\end{proof}

\begin{proof}[Proof of Lemma \ref{lmBridegChain}]
For any graph $\Gamma,$ we denote by $\mathcal{E}(\Gamma)$ the set of edges of $\Gamma.$ Let $G$ be the complete graph with vertices the boundary components of $\Sigma$. For any $e \in \mathcal{E}(G),$ denote by $\ell(e)$ the length of a corresponding bridge. Let $\mathcal{S}$ be the set of spanning trees of $G$. For $T\in\mathcal{S},$ we write
\[
\ell(T)=\sum_{e\in\mathcal{E}(T)}\ell(e).
\]
Let $T\in\mathcal{S}$ be a tree that minimizes $\ell(\cdot)$. We claim that all bridges corresponding to an edge of $T$ are admissible. Indeed, suppose the contrary and let $e\in\mathcal{E}(T)$ correspond to a non-admissible bridge $\beta$ connecting $\gamma_1$ and $\gamma_2$. By definition, there is a $\gamma_3$ such that, using the notation of Lemma~\ref{lmBridge}, the condition of Lemma \ref{lmBridge} is satisfied. For $i =1,2,$ let $e_i$ be the edge of $G$ connecting $\gamma_3$ and $\gamma_i.$ Removing $e$ disconnects $T$ into two connected components $T_1$ and $T_2$ containing the vertices $\gamma_1$ and $\gamma_2$ respectively. Without loss of generality, suppose the vertex corresponding to $\gamma_3$ is in $T_2$ .Then, connecting $T_1$ and $T_2$ with the edge $e_1$ produces a spanning tree $T'$. By Lemma~\ref{lmBridge} we have
\[
\ell(T')=\ell(T)-\ell(e)+\ell(e_1)<\ell(T)
\]
contrary to the choice of $T$. The claim follows. So, the sequence of boundary components corresponding to the unique path in $T$ connecting $\chi_1$ with $\chi_2$ is an admissible chain.
\end{proof}

\begin{proof}[\textbf{Proof of Theorem~\ref{tmDiamEst}}]
To deduce the diameter estimate of Theorem \ref{tmDiamEst}, let $p_1,p_2\in\Sigma$. If $\partial\Sigma =\emptyset$, Theorem \ref{TmApLenBound} and Lemma \ref{lmSegWid} immediately imply the claim. Otherwise, let $\delta_i$, for $i=1,2$,  be the minimal geodesic connecting $p_i$ to $\overline{p_i}$. Clearly, $\delta_i$ is conjugation invariant and intersects a component $\gamma_i$ of $\partial\Sigma$. Let $(\alpha_i)_{i\leq n}$ be an admissible chain connecting $\gamma_1$ and $\gamma_2$ with bridges $\beta_i$ connecting $\alpha_i$ and $\alpha_{i+1}$ for $1 \leq i \leq n-1.$ It follows from the definition of an admissible chain that  $n\leq |\pi_0(\partial\Sigma)|$. We have the estimate
 \[
 d(u(p_1),u(p_2))\leq\ell(u|_{\delta_1})+\ell(u|_{\delta_2})+\ell(u|_{\partial\Sigma})+\sum_{i=1}^{n-1}\ell(u|_{\beta_i}),
 \]
 where all the lengths and distances are measured with respect to $g_J$. So, the claim again follows from Theorem \ref{TmApLenBound}, Lemmas \ref{lmSegWid} and \ref{lmSegWidth2}, and Theorem \ref{TmThickThinM}.
\end{proof}
\begin{rem}
The reader may have noted the unequal treatment of the elements $\alpha_i$ and $\beta_i$ in the admissible chain connecting two boundary components. For the $\alpha_i$ we have an estimate which is independent of the number of components since Lemma \ref{lmSegWid} bounds the segment widths of the boundary as a whole. For the $\beta_i$, on the other hand, Lemma \ref{lmSegWidth2} only provides estimates for each component separately. For now we leave open the question whether or not this may be improved to eliminate the dependence on the number of boundary components in Theorem~\ref{tmDiamEst} altogether.
\end{rem}

\section{Calculation of optimal isoperimetric constant}\label{sec:calc}
\begin{proof}[Proof of Proposition~\ref{pr:ex}]
Let $M=\mathbb{CP}^n$, let $L=\mathbb{RP}^n$ and let $\omega=\omega_{FS}$ be the Fubini Study form, normalized so the area of line is $\frac{1}{\pi}.$ Let $J$ be an $\omega$-tame almost complex structure. Let $\mathcal{M}_{0,2}(M,L,J)$ denote the moduli space of degree $1$  $J$-holomorphic disk maps with two marked points on the boundary modulo reparametrization. Since the energy of a degree $1$ map is minimal, there is no bubbling, and $\mathcal{M}_{0,2}(M,L,J)$ is compact. The structure theorem for the image of $J$-holomorphic disks~\cite{KO00,La11} implies that minimal energy disks are somewhere injective.  It follows that there is a dense set of regular $J$. Recall that for $J$ regular, $\mathcal{M}_{0,2}(M,L,J)$ is a manifold. Let
\[
ev_J:\mathcal{M}_{0,2}(M,L,J)\rightarrow L^2
\]
be the evaluation map. We claim that $ev_J$ is surjective for all $J.$ To see this, note that for $J$ regular $ev_J$ is relatively orientable by~\cite[Theorem 1]{So06}. In particular, it is not hard to check that the standard complex structure $J_{st}$ is regular. For $J =J_{st},$ degree $1$ disk maps are equivalent to oriented real lines. So $ev_{J_{st}}$ is $2$ to $1$ away from the diagonal. Using~\cite[Prop. 5.1]{So06}, one can check that conjugate disks contribute with equal sign, so $e_{J_{st}}$ has degree $2.$  A routine cobordism argument then shows $ev_J$ has degree $2$ for $J$ regular. Finally, Gromov compactness implies surjectivity for any $J$.

We deduce that for any smooth almost complex structure $J$ on $M$, there is a $u: (D^2,\partial D^2) \to (\C P^n,\R P^n)$ such that
\[
\ell(u|_{\partial D^2};g_J)\geq 2Diam(L;g_J)= 4\pi Diam(L;g_J)Area(u;g_J).
\]
So, for any $\omega$-tame $J$, we have $F_1(M,L,J)\geq 4\pi Diam(L;g_J)$, which implies
\[
h_1(\omega)\geq 2\pi.
\]
On the other hand, it follows from the discussion of Section~\ref{sec:cc} that $F_1( J_{st},\omega)\leq 2\pi$ and $Diam(L;g_{J_{st}})=\frac{1}{2}$, which implies that $h_1\leq 2\pi$. Combining the two inequalities, we conclude $h_1(\omega)=2\pi$.
\end{proof}

\bibliographystyle{amsabbrvc}
\bibliography{RefRII}

\vspace{.5 cm}
\noindent
Institute of Mathematics \\
Hebrew University, Givat Ram \\
Jerusalem, 91904, Israel \\

\end{document}